\documentclass[11pt]{article}
\usepackage{amssymb,amsmath,amsthm,verbatim,tikz,fullpage}
\newtheorem{theorem}{Theorem}
\newtheorem{conjecture}{Conjecture}
\newtheorem{proposition}[theorem]{Proposition}
\newtheorem{corollary}[theorem]{Corollary}
\newtheorem{claim}{Claim}
\newtheorem{lemma}[theorem]{Lemma}
\theoremstyle{definition}
\newtheorem{definition}{Definition}
\newtheorem{fact}{Fact}
\newtheorem{question}{Question}
\newtheorem{problem}{Problem}
\theoremstyle{definition}
\newtheorem{remark}{Remark}

\DeclareMathOperator{\Prob}{Prob}
\DeclareMathOperator{\Aut}{Aut}

\DeclareMathOperator{\CT}{CT}
\DeclareMathOperator{\GL}{GL}
\DeclareMathOperator{\Cay}{Cay}

\DeclareMathOperator{\Link}{Link}
\DeclareMathOperator{\Orb}{Orb}
\DeclareMathOperator{\Stab}{Stab}
\DeclareMathOperator{\Sym}{Sym}
\DeclareMathOperator{\poly}{poly}
\newcommand{\Id}{\textrm{Id}}

\newcommand{\A}{\mathcal{A}}
\newcommand{\B}{\mathcal{B}}
\newcommand{\G}{\mathcal{G}}
\newcommand{\cH}{\mathcal{H}}
\newcommand{\cJ}{\mathcal{J}}
\DeclareMathOperator{\Isom}{Isom}

\begin{document}

\title{On the structure of random graphs with constant $r$-balls}
\author{Itai Benjamini\footnote{Department of Mathematics, Weizmann Institute of Science, Rehovot 7610001, Israel.}\: and David Ellis\footnote{School of Mathematics, University of Bristol, Woodland Road, Bristol, BS8 1UG, United Kingdom.}}
\date{27th October 2020}
\maketitle

\begin{abstract}
We continue the study of the properties of graphs in which the ball of radius $r$ around each vertex induces a graph isomorphic to the ball of radius $r$ in some fixed vertex-transitive graph $F$, for various choices of $F$ and $r$. This is a natural extension of the study of regular graphs. 

More precisely, if $F$ is a vertex-transitive graph and $r \in \mathbb{N}$, we say a graph $G$ is {\em $r$-locally $F$} if the ball of radius $r$ around each vertex of $G$ induces a graph isomorphic to the graph induced by the ball of radius $r$ around any vertex of $F$. We consider the following random graph model: for each $n \in \mathbb{N}$, we let $G_n = G_n(F,r)$ be a graph chosen uniformly at random from the set of all unlabelled, $n$-vertex graphs that are $r$-locally $F$. We investigate the properties possessed by the random graph $G_n$ with high probability, i.e.\ with probability tending to 1 as $n \to \infty$, for various natural choices of $F$ and $r$.

We prove that if $F$ is a Cayley graph of a torsion-free group of polynomial growth, and $r$ is sufficiently large depending on $F$, then the random graph $G_n = G_n(F,r)$ has largest component of order at most $n^{5/6}$ with high probability, and has at least $\exp(n^{\delta})$ automorphisms with high probability, where $\delta>0$ depends upon $F$ alone. Both properties are in stark contrast to random $d$-regular graphs, which correspond to the case where $F$ is the infinite $d$-regular tree. We also show that, under the same hypotheses, the number of unlabelled, $n$-vertex graphs that are $r$-locally $F$ grows like a stretched exponential in $n$, again in contrast with $d$-regular graphs. In the case where $F$ is the standard Cayley graph of $\mathbb{Z}^d$, we obtain a much more precise enumeration result, and more precise results on the properties of the random graph $G_n(F,r)$. Our proofs use a mixture of results and techniques from geometry, group theory and combinatorics.
 
We make several conjectures regarding what happens for Cayley graphs of other finitely generated groups.
\end{abstract}

\section{Introduction}
Many results in Combinatorics concern the impact of `local' properties on `global' properties of combinatorial structures (e.g.\ graphs). Some of these results concern the global properties of {\em all} structures with a given local property. Others concern the `typical' global properties of a `random' structure with a given local property. As an example of the former, Dirac's classical theorem \cite{dirac} states that in any $n$-vertex graph where each vertex has degree at least $n/2$, there exists a Hamiltonian cycle. On the other hand, the literature on random graphs contains many examples of the latter, among which is the following. Let $G(n,m)$ denote the Erd\H{o}s-R\'{e}nyi random graph, chosen uniformly at random from the set of all $m$-edge graphs with vertex-set $\{1,2,\ldots,n\}$. Erd\H{o}s and Renyi \cite{erdos-renyi} proved that if $\epsilon >0$ if fixed, and $m = \lfloor (1+\epsilon) n/2 \rfloor$, then with high probability, the graph $G(n,m)$ has one component of order $\Theta(n)$ (called the `giant' component), and $\Theta(n)$ other connected components, each of order $O(\log n)$. (Here, and henceforth, `with high probability' means `with probability tending to 1 as $n \to \infty$'.) It is natural to ask what happens to the global properties of $G \in G(n,m)$ if we impose a local condition at each vertex. For example, if $m=dn/2$, where $d \in \mathbb{N}$ and $n$ is even, and we start with $G(n,m)$ and condition on the event that every vertex has degree $d$, then what we get is of course the random (labelled) $d$-regular graph: the graph $G_n(d)$, chosen uniformly at random from the set of all $d$-regular graphs on $\{1,2,\ldots,n\}$. The following are consequences of theorems of Bollob\'as, McKay and Wormald.
\begin{theorem}[Bollob\'as \cite{bollobas-connected} / Wormald \cite{wormald-thesis}]
\label{thm:d-connected}
Let $d \geq 3$ be a fixed integer. Then with high probability, $G_n(d)$ is $d$-connected.
\end{theorem}
\begin{theorem}[Bollob\'as \cite{bollobas-cycles} / Wormald \cite{wormald-thesis}]
\label{thm:short-cycles}
Let $d,g \geq 3$ be fixed integers. Then $$\textrm{Prob}\{\textrm{girth}(G_n(d)) \geq g\} = (1+o(1)) \frac{\exp(-\sum_{l=1}^{g-1} \lambda_l)}{1-\exp(-(\lambda_1+\lambda_2))},$$
where
$$\lambda_i = \frac{(d-1)^i}{2i} \quad (i \in \mathbb{N}).$$
In particular, $G_n(d)$ has girth at least $g$ with positive probability.
\end{theorem}

\begin{theorem}[Bollob\'as \cite{bollobas} / McKay and Wormald \cite{wormald-mckay}]
\label{thm:bollobas-asym}
Let $d \geq 3$ be a fixed integer. Then
$$\mathbb{E}[|\Aut(G_n(d))|] \to 1\quad \textrm{as }n \to \infty.$$
Hence, $G_n(d)$ is asymmetric with high probability, and so is a uniform random $d$-regular unlabelled graph on $n$ vertices.
\end{theorem}

For comparison, note that if $d \in \mathbb{N}$ is fixed, then the Erd\H{o}s-Renyi random graph $G(n,dn/2)$ (which has {\em average} degree $d$), also has girth at least $g$ with positive probability, for any fixed $g \geq 3$. On the other hand, $G(n,dn/2)$ has $\Theta(n)$ isolated vertices with high probability, and therefore has $\exp(\Theta(n \log n))$ automorphisms with high probability.

(Recall that if $G$ is a graph, the {\em girth} of $G$ is the length of the shortest cycle in $G$; if $G$ is acyclic, we define $\textrm{girth}(G) = \infty$. We write $\Aut(G)$ for the group of all automorphisms of $G$. We say that $G$ is {\em asymmetric} if it has no non-trivial automorphism, i.e.\ $|\Aut(G)| = 1$.)

It is natural to ask what happens to the global structure of a graph when we impose a `local' condition which is stronger than being $d$-regular. A natural condition to impose is that the subgraph induced by the ball of radius $r$ in $G$ around any vertex, is isomorphic to some fixed graph, for some fixed, small $r \in \mathbb{N}$. We proceed to recall (e.g.\ from \cite{be}) some definitions that make this precise.

If $G$ is a (simple, undirected) graph, we write $V(G)$ for the vertex-set of $G$ and $E(G)$ for its edge-set. If $S \subset V(G)$, we write $G[S]$ for the subgraph of $G$ induced by $S$, i.e.\ the maximal subgraph of $G$ with vertex-set $S$. If $v,w \in V(G)$, the {\em distance from $v$ to $w$ in $G$} is defined to be the minimum number of edges in a path from $v$ to $w$ in $G$; it is denoted by $d_{G}(v,w)$. If $v \in V(G)$, and $r \in \mathbb{N}$, we define $\Link_r(v,G)$ to be the subgraph of $G$ induced by the set of vertices of $G$ with distance at most $r$ from $v$, and we define
$$\Link_r^{-}(v,G) := \Link_r(v,G) \setminus \{\{x,y\} \in E(G):\ d_G(v,x)=d_G(v,y)=r\}.$$
A {\em rooted graph} is an ordered pair $(G,v)$ where $G$ is a graph, and $v \in V(G)$.

Our key definition, introduced in \cite{be}, is as follows.
\begin{definition}
\label{defn:r-locally}
If $(F,u)$ is a rooted graph, we say that a graph $G$ is {\em $r$-locally $(F,u)$} if for every vertex $v \in V(G)$, there exists a graph isomorphism $\phi: \Link_r(u,F) \to \Link_r(v,G)$ such that $\phi(u)=v$.
\end{definition}

We remark that if $F$ is vertex-transitive, then Definition \ref{defn:r-locally} is independent of the choice of $u$. Hence, if $F$ is a vertex-transitive graph, we say that a graph $G$ is {\em $r$-locally $F$} if there exists $u \in V(F)$ such that $G$ is $r$-locally $(F,u)$.

A slightly weaker notion, also studied in \cite{be}, is as follows.

\begin{definition}
\label{defn:weakly-r-locally} If $(F,u)$ is a rooted graph, we say that a graph $G$ is {\em weakly $r$-locally $(F,u)$} if for every vertex $v \in V(G)$, there exists a graph isomorphism $\phi: \Link_r^{-}(u,F) \to \Link_r^{-}(v,G)$ such that $\phi(u)=v$.
\end{definition}

Note that, if a connected graph is viewed as a 1-dimensional simplicial complex, equipped with the usual graph metric that assigns length 1 to every edge, then a graph $G$ is weakly $r$-locally $(F,u)$ if and only if for every vertex $v \in V(G)$, there exists a (metric) isometry $\psi$ from the (metric) ball of centre $u$ and radius $r$ in $F$, to the (metric) ball of centre $v$ and radius $r$ in $G$, with $\psi(u)=v$. Definition \ref{defn:weakly-r-locally} is therefore perhaps more natural from a geometric point of view, whereas Definition \ref{defn:r-locally} is perhaps more natural from a combinatorial point of view. To avoid confusion, we remark that Georgakopoulos \cite{georgakopoulos} and De La Salle and Tessera \cite{tessera} define a graph $G$ to be $r$-locally $F$ (for a connected, vertex-transitive graph $F$) if it satisfies Definition \ref{defn:weakly-r-locally} for some $u \in V(F)$. Note however that we have the implications
 $$ G \textrm{ is } r\textrm{-locally }(F,u) \Rightarrow G \textrm{ is weakly } r\textrm{-locally }(F,u) \Rightarrow G \textrm{ is }(r-1)\textrm{-locally }(F,u),$$
for any $r \in \mathbb{N}$. Hence, if we are interested in these notions for large $r$ (as is the case in \cite{tessera,georgakopoulos} and this paper), it does not really matter which notion we work with. In this paper, we work with Definition \ref{defn:r-locally}.

As a simple example, if $d \in \mathbb{N}$ with $d \geq 2$, let $T_d$ denote the infinite $d$-regular tree. A graph $G$ is $r$-locally $T_d$ if and only if it is a $d$-regular graph with girth at least $2r+2$, and $G$ is weakly $r$-locally $T_d$ if and only if it is $d$-regular with girth at least $2r+1$. We remark that regular graphs of high girth have been intensively studied (see e.g.\ \cite{hoory-thesis,luw,lps}).

Another natural example is as follows. 
If $d \in \mathbb{N}$, the {\em $d$-dimensional lattice} $\mathbb{L}^d$ is the graph with vertex-set $\mathbb{Z}^d$, and edge-set
$$\{\{x,x+e_i\}:\ x \in \mathbb{Z}^d, i \in [d]\},$$
where $e_i = (0,0,\ldots,0,1,0,\ldots,0)$ denotes the $i$th unit vector in $\mathbb{R}^d$. It is natural to ask what can be said about the global structure of graphs that are $r$-locally $\mathbb{L}^d$, for various integers $d$ and $r$. We note that $\mathbb{L}^d$ is the standard Cayley graph of the free Abelian group on $d$ generators, whereas $T_{2d}$ is the standard Cayley graph of the free group on $d$ generators, so the case $F=\mathbb{L}^d$ of Definition \ref{defn:r-locally} is perhaps a `second' natural case to study, the `first' natural case perhaps being $T_{2d}$. 

It turns out that for all $d \geq 2$ and all $r \geq 3$, graphs which are $r$-locally $\mathbb{L}^d$ have a very rigidly proscribed, `algebraic' global structure. Indeed, the authors proved the following in \cite{be}.

\begin{theorem}
\label{thm:quotient-d}
Let $d \in \mathbb{N}$ with $d \geq 2$. Let $r \in \mathbb{N}$ with $r \geq 2$ if $d=2$ and $r \geq 3$ if $d \geq 3$. Let $G$ be a connected graph. Then $G$ is $r$-locally $\mathbb{L}^d$ if and only if $G$ is isomorphic to $\mathbb{L}^d/\Gamma$, where $\Gamma \leq \Aut(\mathbb{L}^d)$ is such that $\Gamma$ has minimum displacement at least $2r+2$, and $G$ is weakly $r$-locally $\mathbb{L}^d$ if and only if $G$ is isomorphic to $\mathbb{L}^d/\Gamma$, where $\Gamma \leq \Aut(\mathbb{L}^d)$ is such that $\Gamma$ has minimum displacement at least $2r+2$. In particular, if $G$ is a connected graph that is weakly $3$-locally $\mathbb{L}^d$, then $G$ is covered by $\mathbb{L}^d$.
\end{theorem}

(Here, if $F$ is a graph and $\Gamma \leq \Aut(F)$, the {\em minimum displacement} of $\Gamma$ is defined to be $D(\Gamma) := \min\{ d_{F}(v,\gamma(v)):\ v \in V(F),\ \gamma \in \Gamma \setminus \{\Id\}\}.$ If $F$ is a graph and $\Gamma \leq \Aut(F)$ with $\Gamma$ acting freely on $F$, the {\em quotient} graph $F/\Gamma$ is as defined in Definition \ref{definition:quotient}, in Section \ref{sec:background}. If $F$ and $G$ are graphs, we say that $F$ {\em covers} G if there exists a graph homomorphism from $F$ to $G$ that is bijective on neighbourhoods; see Definition \ref{defn:cover} in Section \ref{sec:background}.)

We remark that similar structure theorems for the case $d=2$ (albeit with combinatorial rather than algebraic descriptions) were obtained by Thomassen \cite{thomassen-torus} and by M\'arquez, de Mier, Noy and Revuelta \cite{noy}.

Viewed as a subgroups of $\Isom(\mathbb{R}^d)$, the groups $\Gamma$ in Theorem \ref{thm:quotient-d} are $d$-dimensional crystallographic groups, and the orbit spaces $\mathbb{R}^d/\Gamma$ are compact topological orbifolds. It follows that for any $d \geq 2$, we can view any finite, connected graph which is $3$-locally $\mathbb{L}^d$, as a `quotient lattice' of $\mathbb{L}^d$ inside a compact $d$-dimensional topological orbifold. Bieberbach's theorems \cite{bieberbach-1911,bieberbach-1912} imply that for any $d \in \mathbb{N}$, there are only a finite number ($f(d)$, say) of affine-conjugacy classes of $d$-dimensional crystallographic groups (where two crystallographic groups are said to be {\em affine-conjugate} if they are conjugate via an affine transformation of $\mathbb{R}^d$). It follows that the orbit space $\mathbb{R}^d/\Gamma$ is homeomorphic to one of at most $f(d)$ topological spaces. (See Section \ref{sec:background} of this paper, or \cite{be}, for the definitions of any unfamiliar terms used above.)
\label{page:orbifolds}

The `rigidity' phenomenon in Theorem \ref{thm:quotient-d} can be seen as a discrete analogue (for graphs) of the curvature-zero case of the seminal rigidity results of Ballmann \cite{ballmann} and independently Burns and Spatzier \cite{burns-spatzier} on complete connected Riemannian manifolds of finite volume and bounded nonpositive sectional curvature.

Observe that the `algebraic' structure in Theorem \ref{thm:quotient-d} is in stark contrast to the situation for regular graphs of high girth (i.e., graphs which are $r$-locally $T_{d}$ for fixed $r,d \in \mathbb{N}$). Indeed, the uniform random $d$-regular graph $G_n(d)$ can be generated using a simple, purely combinatorial process, namely, the Configuration Model of Bollob\'as \cite{config}, and $G_n(d)$ has girth at least $g$ with positive probability for any fixed $d,g \geq 3$, by Theorem \ref{thm:short-cycles}. 

Georgakopoulos \cite{georgakopoulos} and the first author \cite{coarse-geometry} asked whether the phenomenon in Theorem \ref{thm:quotient-d} holds more generally. Following \cite{tessera}, we say that a locally finite, vertex-transitive graph $F$ is {\em local-to-global rigid} (or {\em LG-rigid}, for short) if there exists $r\in \mathbb{N}$ such that any connected graph that is $r$-locally $F$, must be covered by $F$. (A graph $F$ is said to be {\em locally finite} if every vertex of $F$ has only finitely many neighbours.) Georgakopoulos and the first author asked the following.
\begin{question}
Is every locally finite Cayley graph of a finitely presented group LG-rigid?
\end{question}

This question was answered in the negative by De La Salle and Tessera, who constructed in \cite{tessera} a locally finite Cayley graph of $\textrm{PGL}(4,\mathbb{Z})$ that is not LG-rigid. (See \cite{tessera,tessera-bruhat-tits} for other counterexamples.) On the other hand, De La Salle and Tessera proved the following beautiful positive result.

\begin{theorem}[De La Salle, Tessera, \cite{tessera}]
\label{thm:dlst-intro}
If $F$ is a locally finite Cayley graph of a torsion-free group of polynomial growth, then there exists $r \in \mathbb{N}$ such that for any connected graph $G$ that is $r$-locally $F$, there exists a group $\Gamma \leq \Aut(F)$ acting freely on $F$, such that the quotient graph $F/\Gamma$ is isomorphic to $G$. In particular, $F$ is LG-rigid.
\end{theorem}

Recall that a group $\Gamma$ is said to be {\em torsion-free} if the only element of $\Gamma$ with finite order is the identity, and it is said to have {\em polynomial growth} if there exists $K \in \mathbb{N}$ and a finite generating set $S$ of $\Gamma$ such that for any $n \in \mathbb{N}$,
$$|\{g \in \Gamma:\ \exists s_1,s_2,\ldots,s_n \in S \textrm{ such that } g = s_1 s_2 \ldots s_n\}| \leq n^K.$$

Strengthening a result of Georgakopoulos \cite{georgakopoulos}, De La Salle and Tessera also proved in \cite{tessera} that if $F$ is a locally finite Cayley graph of a finitely presented group and $F$ is quasi-isometric to a tree, then $F$ is LG-rigid. We conjecture that the the conclusion of LG-rigidity in Theorem \ref{thm:dlst-intro} holds without the torsion-free hypothesis; the same conjecture has been made independently by De La Salle and Tessera [personal communication].

In this paper, we study the {\em typical} properties of $n$-vertex graphs that are $r$-locally $F$, and the number of such graphs, for various choices of $F$ and $r$. We introduce a new random graph model as follows. 

\begin{definition}
\label{defn:rgmodel}
Let $F$ be a fixed infinite, locally finite, vertex-transitive graph, and let $r \in \mathbb{N}$ be a fixed integer. For each $n \in \mathbb{N}$, we let $G_n = G_n(F,r)$ be a graph chosen uniformly at random from the set $\mathcal{S}_n = \mathcal{S}_n(F,r)$ of all unlabelled, $n$-vertex graphs that are $r$-locally $F$.
\end{definition}

We investigate the properties possessed by the random graph $G_n$ with high probability (meaning, as usual, with probability tending to 1 as $n \to \infty$). This is a natural extension of the well-studied random regular graph $G_n(d)$, which corresponds roughly to the case $F = T_d$ and $r=1$. (The random $d$-regular graph $G_n(d)$ is a labelled graph, but as outlined in \cite{wormald-survey}, many properties possessed by $G_n(d)$ with high probability are also possessed with high probability by the uniform random unlabelled $n$-vertex $d$-regular graph.) However, the main combinatorial and probabilistic tools used for studying random regular graphs arise from the aforementioned Configuration Model of Bollob\'as, whereas in the cases we consider, there is no such simple combinatorial process that generates a uniform (or even approximately uniform) random $n$-vertex graph that is $r$-locally $F$. Instead, we use a mixture of geometric, algebraic and combinatorial arguments, and the results we obtain are somewhat less sharp than their analogues for random regular graphs.

Note that the random graph $G_n = G_n(F,r)$ in the above definition is defined only if $\mathcal{S}_n \neq \emptyset$. In the cases we study, $\mathcal{S}_n \neq \emptyset$ for all $n$ sufficiently large depending on $F$ and $r$ (see Remark \ref{remark:nonempty}), so there is no problem with considering properties possessed by $G_n$ with high probability. In the generality of Definition \ref{defn:rgmodel}, it is possible that $\mathcal{S}_n(F,r) = \emptyset$ for infinitely many $n$, for example $\mathcal{S}_n(T_3,1) = \emptyset$ for all odd $n$, but so long as $\mathcal{S}_n(F,r) \neq \emptyset$ for infinitely many $n$, we simply interpret `with high probability' to mean with probability tending to 1 as $n \to \infty$ over integers $n$ such that $\mathcal{S}_n(F,r) \neq \emptyset$, as is usual with random $d$-regular graphs when $d$ is odd. If $F$ is a locally finite Cayley graph of a residually finite group $\Gamma$, then for every $r \in \mathbb{N}$, there exist (arbitrarily large) finite connected graphs that are $r$-locally $F$, so in particular, $\mathcal{S}_n(F,r) \neq \emptyset$ for infinitely many $n$. (This follows easily by considering the quotient graph $F/\Gamma'$ for appropriate finite-index subgroups $\Gamma'$ of $\Gamma$, as in the proof of Proposition \ref{prop:linear}.) In all such cases, the random graph model above is non-trivial. However, it is in place to remark that there exists a locally finite Cayley graph $F$ of the Baumslag-Solitar group $\textrm{BS}(2,3)$, and a positive integer $r$, such that no nonempty finite graph is $r$-locally $F$; this follows for example from Corollary K in \cite{tessera}. For all such pairs $(F,r)$, the random graph model above is of course trivial; this causes a very slight issue in Section \ref{sec:related} on open problems.

Our first aim is to estimate the number of $n$-vertex graphs that are $r$-locally $\mathbb{L}^d$, for various pairs of integers $d,r$. We prove the following rather precise estimate.

\begin{theorem}
\label{thm:enumeration}
Define $r^*(2) = 2$, $r^{*}(d) = 3$ for all $3 \leq d \leq 7$, and $r^*(d) = \lceil (d-1)/2 \rceil$ for all $d \geq 8$. Let $d \in \mathbb{N}$ with $d \geq 2$, and let $r \in \mathbb{N}$ with $r\geq r^*(d)$. Let $a_{d,r}(n)$ denote the number of unlabelled $n$-vertex graphs that are $r$-locally $\mathbb{L}^d$. Then there exists $\epsilon_d >0$ depending upon $d$ alone such that
$$\log a_{d,r}(n) = (1+O_{d,r}(n^{-\epsilon_d}))K_d n^{d/(d+1)},$$
where $K_d : = \tfrac{d+1}{d}\left( \tfrac{1}{2^{d-1}}\prod_{i=2}^{d}\zeta(i)\right)^{1/(d+1)}$.
\end{theorem}

This theorem says that, if $r$ grows linearly with $d$, then the number of unlabelled graphs on $n$ vertices that are $r$-locally $\mathbb{L}^d$ grows like a stretched exponential in $n$. (This is in sharp contrast with the number of unlabelled $(2d)$-regular graphs on $n$ vertices, which grows superexponentially in $n$.) Note that, in contrast to Theorem \ref{thm:quotient-d}, Theorem \ref{thm:enumeration} applies only in the case where $r$ grows linearly with $d$. A less precise enumeration result could be proved under the same hypotheses as in Theorem \ref{thm:quotient-d}, but for brevity and for clarity of exposition, we do not prove such a result here.

We remark that a much stronger statement holds for $d=1$ with $r^*(1):=1$ (see Remark \ref{remark:1-stronger}). In particular, we have
$$\log a_{1,r}(n) = \pi \sqrt{2n/3} - r\log n + O_r(1)$$
as $n \to \infty$.

Our proof of Theorem \ref{thm:enumeration} has several steps (of which the second, in particular, may be of independent interest):
\label{proof-sketch}
\begin{itemize}
\item We first use Theorem \ref{thm:quotient-d}, together with some standard results and arguments from group theory and topological graph theory, to prove that the {\em connected} unlabelled $n$-vertex graphs that are $r$-locally $\mathbb{L}^d$ are in one-to-one correspondence with conjugacy-classes of subgroups $\Gamma \leq \Aut(\mathbb{L}^d)$ such that $|\mathbb{Z}^d/\Gamma| = n$ and $D(\Gamma) \geq 2r+2$, provided $r \geq r^*(d)$. Our next task is to enumerate the latter.

\item We show that for any $x \geq 0$, all but an $o(1)$-fraction of the conjugacy-classes of subgroups $\Gamma \leq \Aut(\mathbb{L}^d)$ such that $|\mathbb{Z}^d/\Gamma| \leq x$ and $D(\Gamma) \geq 2r+2$, are conjugacy-classes of subgroups consisting only of translations (`pure-translation subgroups'), provided $r \geq r^*(d)$; we then estimate the number of the latter. We do this by analysing the former subgroups in terms of their lattices of translations and their point groups, using a cohomology argument borrowed from one of the standard proofs of Bieberbach's Third Theorem, and using a combinatorial argument to show that a uniform random sublattice of $\mathbb{Z}^d$ with index at most $x$ is unlikely to have any non-trivial hyperoctahedral symmetry.

It follows from this part of the proof that for any $x \geq 0$, all but an $o(1)$-fraction of the connected unlabelled graphs on at most $x$ vertices that are $r$-locally $\mathbb{L}^d$ (for $r \geq r^*(d)$) are quotient lattices of $\mathbb{L}^d$ inside topological tori of the form $\mathbb{R}^d/\Gamma_0$ (where $\Gamma_0$ is the group of translations by vectors in a rank-$d$ sublattice of $\mathbb{Z}^d$). We may call such graphs `generalized discrete torus' graphs, since they generalize the usual discrete torus $C_k^d$; the latter is obtained when $\Gamma_0 = \langle \{x \mapsto x+ke_i :\ i \in [d]\}\rangle$. In general, the `fundamental domain' of a generalized discrete torus graph is a parallelepiped that need not have any of its edges parallel to a coordinate axis.

It is in place to remark that in the case $3 \leq r < \lceil (d-1)/2 \rceil$, a positive fraction of the connected unlabelled graphs on at most $x$ vertices that are $r$-locally $\mathbb{L}^d$ are quotient lattices of $\mathbb{L}^d$ inside a topological orbifold obtained by taking the quotient of a topological torus $\mathbb{R}^d/\Gamma_0$ by some involution of the form $x \mapsto c-x$, where $c \in (\tfrac{1}{2} \cdot \mathbb{Z})^d$, so the assumption that $r$ grows linearly with $d$ is necessary for this step of the proof to work. (The rest of the proof works under the weaker assumption $r \geq 2+1_{\{d \geq 3\}}$.)

\item Combining the previous two steps allows us to estimate, for each $x \geq 0$, the number of {\em connected} unlabelled graphs with at most $x$ vertices, which are $r$-locally $\mathbb{L}^d$. We use this, combined with a generating function argument and a variant of a result of Brigham on the asymptotics of partition functions, to obtain the estimate in Theorem \ref{thm:enumeration} on the number of (possibly disconnected) unlabelled graphs with $n$ vertices, which are $r$-locally $\mathbb{L}^d$. We remark that the aforesaid variant we need (with a more precise error term than Brigham's) does not seem to appear in the literature, but can be proved using classical techniques from analytic number theory; for the reader's convenience, we give a full proof (in the Appendix) which is an adaptation of Brigham's proof of his result, appealing to a theorem of Odlyzko \cite{odlyzko} instead of the theorem of Hardy and Ramanujan from \cite{hardy-ramanujan-general}, which Brigham uses.
\end{itemize}

Using some of the same tools, we also prove the following.
\begin{theorem}
\label{thm:sampling}
For any fixed integers $d,r \in \mathbb{N}$ with $d \geq 2$ and $r \geq 2+1_{\{d \geq 3\}}$, the random graph $G_n(\mathbb{L}^d,r)$ can be sampled in time polynomial in $n$, using a polynomial (in $n$) number of coins with rational biases.
\end{theorem}

We then move on to consider the (much) more general case of connected, locally finite Cayley graphs of torsion-free groups of polynomial growth. (This is a well-studied class of graphs, and a natural extension of $\mathbb{L}^d$ case. We recall that Gromov's celebrated theorem \cite{gromov} on groups of polynomial growth states that a finitely generated group has polynomial growth if and only if it is virtually nilpotent.) Our first main result in this more general case is as follows.

\begin{theorem}
\label{thm:largest-component}
Let $F$ be a connected, locally finite Cayley graph of a torsion-free group of polynomial growth. Then there exists $r_0 = r_0(F) \in \mathbb{N}$ such that for all $r \geq r_0$, the random graph $G_n = G_n(F,r)$ has largest component of order at most $n^{5/6}$ with probability at least $1-\exp(-n^{1/13})$, provided $n$ is sufficiently large depending on $F$ and $r$.
\end{theorem}

\begin{remark}
\label{remark:l-d}
In the case where $F=\mathbb{L}^d$ for $d \geq 2$, we may take $r_0 = 2+1_{\{d \geq 3\}}$ in the above theorem.
\end{remark}

\begin{remark}
\label{remark:nonempty}
We also show that if $F$ is a connected, locally finite Cayley graph of a torsion-free group of polynomial growth, and $r \in \mathbb{N}$, then for all $n$ sufficiently large depending on $F$ and $r$, the set $\mathcal{S}_n = \mathcal{S}_n(F,r)$ in Definition \ref{defn:rgmodel} is nonempty (so that the statement of Theorem \ref{thm:largest-component} is meaningful for all sufficiently large $n$); this is the content of Lemma \ref{lemma:suff-large}.
\end{remark}

Our proof of Theorem \ref{thm:largest-component} proceeds as follows. The first step is to obtain an upper bound on the number of {\em connected}, unlabelled, $n$-vertex graphs that are $r$-locally $F$, for sufficiently large $r$ and all $n \in \mathbb{N}$; this upper bound is deduced from Gromov's theorem, Theorem \ref{thm:dlst-intro}, and certain other known results from Geometric Group Theory, including another result of De La Salle and Tessera (from \cite{tessera}), on the structure of automorphism groups of Cayley graphs of torsion-free groups of polynomial growth. The second step is to obtain a lower bound on the above quantity, for an appropriate infinite sequence of positive integers $n$; this step requires a rather intricate group-theoretic argument, together with the use of a classical theorem of Gruenberg. In the third and last step of the proof, we use the upper and lower bounds obtained in the first two steps to estimate the probability that a uniform random unlabelled, $n$-vertex graph that is $r$-locally $F$, has largest component of order $n^{5/6}$. This part of the proof is combinatorial, and is performed by the construction of explicit maps between certain carefully chosen classes of unlabelled graphs that are $r$-locally $F$, together with the use of classical results from the theory of integer partitions. The upper and lower bounds obtained in the first and second steps in our proof are (in general) far apart and are (probably both) far from the truth; the challenge in the third step is to find arguments robust enough to yield good estimates on the required probability, in spite of this gap. Generating function techniques (often used in similar problems where the gap is narrower) could not be directly used in our case, for example.
\newline 

Our next main result is that with high probability, the random graph $G_n(F,r)$ has at least $\exp(\poly(n))$ automorphisms:
\begin{theorem}
\label{thm:aut-gen}
Let $F$ be a connected, locally finite Cayley graph of a torsion-free group of polynomial growth. There exists $r_0 = r_0(F) \in \mathbb{N}$ and $\delta_0 = \delta_0(F) >0$ such that for all $r \geq r_0$, 
\begin{equation}
\label{eq:many-auts}
 \Pr[|\Aut(G_n(F,r))| \geq 2^{n^{\delta_0}}] \geq 1-\exp(-n^{\delta_0})\end{equation}
provided $n$ is sufficiently large depending on $F$ and $r$.
\end{theorem}

In fact, our proof of Theorem \ref{thm:aut-gen} yields something somewhat stronger than (\ref{eq:many-auts}), namely, that (under all of the same conditions), with probability at least $1-\exp(-n^{\delta_0})$, the random graph $G_n(F,r)$ has at least $n^{\delta_0}$ vertex-transitive components. The high-level structure of the proof is similar to that of Theorem \ref{thm:largest-component}, except that in the second step, it is necessary to obtain a lower bound on the number of connected, {\em vertex-transitive} unlabelled $n$-vertex graphs that are $r$-locally $F$, for an appropriate infinite sequence of positive integers $n$.

Turning to the enumeration question, we prove a somewhat less precise result than Theorem \ref{thm:enumeration} in the more general case. Viz., using some of the same tools used in the proofs of Theorems \ref{thm:largest-component} and \ref{thm:aut-gen}, together with a classical theorem of Brigham on the asymptotics of partition functions, we prove that if $F$ is a connected, locally finite Cayley graph of a torsion-free group of polynomial growth, then there exists $r_0 = r_0(F) \in \mathbb{N}$ such that for each $r \geq r_0$, the number $b_{F,r}(n)$ of unlabelled, $n$-vertex graphs that are $r$-locally $F$ grows like a stretched exponential.

\begin{theorem}
\label{thm:stretched-exp}
Let $F$ be a connected, locally finite Cayley graph of a finitely generated, torsion-free group of polynomial growth, and let $r_0 = r_0(F)$ be as in Theorem \ref{thm:dlst-intro}. For each $r \geq r_0$, there exist $n_0 \in \mathbb{N}$, $C>0$, $c>0$ and $\alpha\geq 0$ such that
$$c \sqrt{n} \leq \log b_{F,r}(n) \leq Cn^{(\alpha+1)/(\alpha+2)}$$
for all $n \in \mathbb{N}$ such that $n \geq n_0$.
\end{theorem} 

We do not, however, determine $\lim_{n \to \infty} (\log \log b_{F,r}(n))/ \log n$ in the general case, as we do in the case of $\mathbb{L}^d$ where $r \geq r^*(d)$. 

In the special case where $F = \mathbb{L}^d$, we make the following conjecture, in the spirit of Theorem \ref{thm:largest-component}.

\begin{conjecture}
\label{conj:stronger-component}
Define $r_0(2)=2$ and $r_0(d) = 3$ for all $d \geq 3$. Let $d \in \mathbb{N}$ with $d \geq 2$, and let $r \in \mathbb{N}$ with $r \geq r_0(d)$. Then with high probability, the largest component of $G_n(\mathbb{L}^d,r)$ has order
$$\Theta_{d,r}(n^{1/(d+1)} \log n).$$
\end{conjecture}

We remark that the statement of Conjecture \ref{conj:stronger-component} holds for $d=1$ with $r_0(1)=1$. This is easily deduced from known results from the theory of integer partitions. (See Section \ref{subsection:integer-partitions}.)

It is instructive to contrast Theorems \ref{thm:largest-component} and \ref{thm:aut-gen} with the situation where $F=T_{d}$, the infinite $d$-regular tree, for $d \geq 3$. As remarked above, a graph $G$ is $r$-locally $T_d$ if and only if it is a $d$-regular graph of girth at least $2r+2$. By Theorem \ref{thm:short-cycles}, for fixed integers $d,g \geq 3$, a uniform random labelled $d$-regular graph on $\{1,2,\ldots,n\}$ has girth at least $g$ with probability $\Omega_{d,g}(1)$ as $n \to \infty$. Using Theorems \ref{thm:d-connected} and \ref{thm:bollobas-asym}, it follows that if $d,g \geq 3$ are fixed integers and $G_n$ is the random graph chosen uniformly from the set of all {\em unlabelled}, $n$-vertex, $d$-regular graphs of girth at least $g$, then with high probability, $G_n$ is asymmetric and $d$-connected. It follows that if $d \geq 3$ and $r \geq 1$, then with high probability, $G_n(T_d,r)$ is asymmetric and $d$-connected.
\label{page:asym}

We remark that, instead of working with unlabelled graphs, one could consider instead the properties of the random graph $H_n = H_n(F,r)$ chosen uniformly from the set of all {\em labelled} graphs with vertex-set $\{1,2,\ldots,n\}$ that are $r$-locally $F$, for various choices of $F$ and $r$, and the number of such graphs. The behaviour of $H_n$ does not vary so strikingly with the choice of $F$ and $r$, as in our case: for example, if $d,r \in \mathbb{N}$ are fixed with $d \geq 2$, and $F$ is an infinite, $d$-regular, connected, vertex-transitive graph, then the number of labelled graphs on $\{1,2,\ldots.n\}$ that are $r$-locally $F$ is $\exp(\Theta_{d,r}(n \log n))$, as is the number of labelled $d$-regular graphs on $\{1,2,\ldots,n\}$. Hence, we believe that the random unlabelled graph $G_n$ is a more interesting object of study in this context. While the recent trend in Probabilistic Combinatorics (perhaps for reasons of tractability) has been to focus more on random {\em labelled} structures and their enumeration, unlabelled structures have also attracted much attention in Combinatorics, Probability Theory and related fields. We mention for example the work of Tutte \cite{tutte1,tutte2} in the 1960s and 70s on the enumeration of planar maps, and the more recent work of Angel and Schramm \cite{as}, Le Gall \cite{legall-inv}, Gurel-Gurevich and Nachmias \cite{ori-asaf}, and others, on random planar maps. (See also the survey papers \cite{legall-ecm,legall-icm} of Le Gall, and the references contained therein.)

The remainder of this paper is structured as follows. In Section \ref{sec:related}, we discuss a broad collection of problems related to our work. In Section \ref{sec:background}, we describe the background and tools we will need from topological graph theory, group theory and topology, and we prove some preliminary lemmas using these tools. In Section \ref{sec:proof-enum}, we prove Theorem \ref{thm:enumeration}, our enumeration result for unlabelled graphs that are $r$-locally $\mathbb{L}^d$. In Section \ref{sec:general}, we consider the much more general case where $\mathbb{L}^d$ is replaced by a connected, locally finite Cayley graph of a torsion-free group of polynomial growth, and we prove Theorems \ref{thm:largest-component}, \ref{thm:aut-gen} and \ref{thm:stretched-exp}. In Section \ref{sec:typical}, we prove some more precise results on the typical properties of graphs that are $r$-locally $\mathbb{L}^d$, including a more precise version of Theorem \ref{thm:aut-gen} (on the order of the automorphism group), for such graphs. In the Appendix, we give a full proof of the result we need on the asymptotics of partition functions (needed for proving Theorem \ref{thm:enumeration}), by adapting the argument of Brigham in \cite{brigham}; this is a fairly standard argument and is included for the convenience of readers not so familiar with this topic. (We note that the bulk of the paper is devoted to the proofs of Theorems \ref{thm:enumeration}, \ref{thm:largest-component} and \ref{thm:aut-gen}, and the necessary preliminaries.)

\section{Related problems and conjectures}
\label{sec:related}

Our results and conjectures above are of course part of a more general class of questions, which we proceed to discuss in this section. Let $F$ be an infinite, locally finite, connected, vertex-transitive graph, and let $r \in \mathbb{N}$. Let us say that the pair $(F,r)$ is {\em connectivity-forcing} if the random graph $G_n = G_n(F,r)$ in Definition \ref{defn:rgmodel} is connected with high probability. (If there exists no nonempty finite graph that is $r$-locally $F$, then by convention we say that $(F,r)$ is connectivity-forcing; otherwise, we have $\mathcal{S}_n(F,r) \neq \emptyset$ for infinitely many $n$, so the term `with high probability' is well-defined; see the discussion following Definition \ref{defn:rgmodel}.) It is natural to pose the following (perhaps rather ambitious) problem.
\begin{problem}
Characterise the pairs $(F,r)$ that are connectivity-forcing.
\end{problem}
Note that it follows immediately from Theorem \ref{thm:largest-component} that if $F$ is a connected, locally finite Cayley graph of a torsion-free group of polynomial growth, then there exists $r_0 \in \mathbb{N}$ depending upon $F$ alone such that for any $r \geq r_0$, the pair $(F,r)$ is not connectivity-forcing. On the other hand, it follows from the discussion on page \pageref{page:asym} that the pair $(T_d,r)$ is connectivity-forcing, for all $r \in \mathbb{N}$.

In the above problem, one may of course replace connectivity with any other graph property (with the same proviso concerning when $\mathcal{S}_n(F,r) = \emptyset$ for all $n \in \mathbb{N}$). A closely related question is for which pairs $(F,r)$ it holds that $G_n(F,r)$ has a `giant component' with high probability. (Recall that if $(G_n)$ is a sequence of random graphs, we say that $G_n$ {\em has a giant component with high probability} if there exists $\epsilon >0$ such that
$$\Prob\{G_n \text{ has a component of order at least }\epsilon n\} \to 1$$
as $n \to \infty$.) In this case, we say that the pair $(F,r)$ {\em forces a giant component}. We pose the following.
\begin{problem}
Characterise the pairs $(F,r)$ that force a giant component.
\end{problem}
Theorem \ref{thm:largest-component} implies that if $F$ is a connected, locally finite Cayley graph of a torsion-free group of polynomial growth and $r \geq r_0(F)$, then $(F,r)$ does not force a giant component.

It is also natural to ask for which pairs $(F,r)$ it holds that $G_n$ is an expander with high probability. Let us state this problem in full. Recall that if $G = (V,E)$ is a finite graph, and $S \subset V(G)$, we define the {\em edge-boundary} $\partial S := E(S,V \setminus S)$, i.e.\ it is the set of edges of $G$ between $S$ and $V \setminus S$. We define the {\em vertex-boundary} $b(S) : = N(S) \setminus S$, i.e.\ it is the set of vertices of $G$ which are not in $S$, but which are adjacent to a point in $S$. We define the {\em edge-expansion ratio} of $G$ to be
$$h(G): = \min\{|\partial S|/|S|:\ 0 < |S| \leq |V|/2\}.$$
Similarly, we define the {\em vertex-expansion ratio} of $G$ to be
$$h_{v}(G) := \min\{|b(S)|/|S|:\ 0 < |S| \leq |V|/2\}.$$
As usual, let $\Delta(G)$ denote the maximum degree of $G$. Note that we have $|\partial S| \leq \Delta(G) |b(S)|$ for all $S \subset V$, so $h_v(G) \geq h(G)/\Delta(G)$.

Now let $(G_n)_{n \in \mathbb{N}}$ be any sequence of graphs with $|V(G_n)| \to \infty$ as $n \to \infty$, and with uniformly bounded maximum degree (i.e., $\sup_{n \in \mathbb{N}} \Delta(G_n) < \infty$). We say that $G_n$ {\em is an expander with high probability} if there exists $\epsilon >0$ such that
$$\Prob\{h(G_n) \geq \epsilon\} \to 1\quad \textrm{as } n \to \infty.$$
Note that this is actually a property of the sequence of random graphs $(G_n)$, not of a single graph.

Let us say that the pair $(F,r)$ is {\em expansion-forcing} if $G_n(F,r)$ is an expander with high probability. We pose the following.
\begin{problem}
Characterise the pairs $(F,r)$ that are expansion-forcing.
\end{problem}
Note that if $(F,r)$ is expansion-forcing, then it is clearly connectivity-forcing. Bollob\'as \cite{bollobas-expansion} proved that if $d \geq 3$ is a fixed integer, the random (labelled) $d$-regular graph $G_n(d)$ has edge-expansion ratio at least $d/18$ with high probability. It follows from this, Theorem \ref{thm:short-cycles} and Theorem \ref{thm:bollobas-asym}, that for any $d \geq 3$ and any $r \in \mathbb{N}$, the pair $(T_d,r)$ is expansion-forcing. Furthermore, we have $h_v(G_n(T_d,r)) \geq 1/18$ with high probability. On the other hand, if $F$ is a connected, locally finite Cayley graph of a torsion-free group of polynomial growth and $r \geq r_0(F)$, then $(F,r)$ is not connectivity-forcing, so it is certainly not expansion-forcing.

In the case where $F$ is (the graph of) a regular tiling of the hyperbolic plane, we make the following conjecture.
\begin{conjecture}
Let $F$ be the graph of a regular tiling of the hyperbolic plane. Then for any $r \in \mathbb{N}$, the pair $(F,r)$ is connectivity-forcing.
\end{conjecture}

If $G$ is a finite graph, let us write $l(G)$ for the order of the largest component of $G$. Let $\mathcal{R}$ be the set of all $x \in [0,1]$ such that there exists a pair $(F,r)$ such that with probability 1, $\log_n(l(G_n(F,r))) \to x$ as $n \to \infty$. Conjecture \ref{conj:stronger-component} (together with a mild assumption on the rate of convergence to 1 of the probability concerned) would imply that $1/(d+1) \in \mathcal{R}$ for all $d \in \mathbb{N}$; Corollary \ref{corr:1locally} implies that $1/2 \in \mathcal{R}$. The fact that $G_n(d)$ is connected with probability at least $1-O(1/n^2)$ for each fixed integer $d \geq 3$, together with Theorems \ref{thm:short-cycles} and \ref{thm:bollobas-asym}, implies that $1 \in \mathcal{R}$. It would be of interest to further investigate the set $\mathcal{R}$. It would also be interesting to determine whether or not there exists a pair $(F,r)$ and an absolute constant $c>0$ such that with high probability, $cn < l(G_n(F,r)) < n$, i.e.\ whether it is possible for $G_n = G_n(F,r)$ to have a giant component and yet be disconnected, with high probability.

If $\Gamma$ is a finitely generated group, let us say that $\Gamma$ is {\em strongly-connectivity-forcing} if for any finite generating set $S$ of $\Gamma$ with $\Id \notin S$ and $S^{-1} = S$, and for any $r \in \mathbb{N}$, the pair $(\Cay(\Gamma,S),r)$ is connectivity-forcing. We pose the following.

\begin{problem}
\label{question:non-amenable}
Characterise the finitely generated groups that are strongly-connectivity-forcing.
\end{problem}

We conjecture that any finitely generated group with super-exponential subgroup growth is strongly-connectivity-forcing. (See the book of Lubotzky and Segal \cite{ls} for background on such groups.) 

We also conjecture that if $\Gamma$ is a finitely generated, residually finite group, and $F$ is a connected, locally finite Cayley graph of $\Gamma$, then the property of $(F,r)$ being connectivity-forcing / expansion-forcing / forcing a giant component, for all sufficiently large $r$, does not depend upon the generating set chosen for $F$. More precisely, we make the following.

\begin{conjecture}
Let $\Gamma$ be a finitely generated, residually finite group, and let $S_i$ (for $i = 1,2$) be two finite generating sets of $\Gamma$ with $\Id \notin S_i$ and $S_i^{-1} = S_i$. Let $F_i = \Cay(\Gamma,S_i)$. Then $(F_1,r)$ has property $P$ for all sufficiently large $r$ if and only if $(F_2,r)$ has property $P$ for all sufficiently large $r$, where $P$ is the property of being connectivity-forcing, or the property of being expansion-forcing, or the property of forcing a giant component.
\end{conjecture}

(The hypothesis of residual finiteness is added to ensure that we never have $\mathcal{S}_n(F,r) = \emptyset$ for all $n \in \mathbb{N}$; see the discussion after Definition \ref{defn:rgmodel}.)

We now turn to questions of symmetry. We say a pair $(F,r)$ is {\em asymmetry-forcing} if $G_n(F,r)$ has no nontrivial automorphisms, with high probability. We pose the following.
\begin{problem}
Characterise the pairs $(F,r)$ that are asymmetry-forcing.
\end{problem}
It follows from the discussion on page \pageref{page:asym}, that for any $d \geq 3$ and any $r \in \mathbb{N}$, the pair $(T_d,r)$ is asymmetry-forcing. On the other hand, it follows immediately from Theorem \ref{thm:aut-gen} that if $F$ is a connected, locally finite Cayley graph of a torsion-free group of polynomial growth, then there exists $r_0 \in \mathbb{N}$ depending upon $F$ alone such that for any $r \geq r_0$, the pair $(F,r)$ is not asymmetry-forcing; on the contrary, with high probability, $G_n(F,r)$ has at least $\exp(\poly(n))$ automorphisms. Motivated by the latter, we say a pair $(F,r)$ is {\em stretched-exponential-symmetry-forcing} if there exists $\delta >0$ such that with high probability, $G_n(F,r)$ has at least $\exp(n^{\delta})$ automorphisms; it would also be of interest to characterise the pairs $(F,r)$ that are stretched-exponential-symmetry-forcing. We thank an anonymous referee for suggesting this problem.

We conclude with a more general suggestion for future research. It is straightforward to generalise Definition \ref{defn:r-locally} and the random graph model in Definition \ref{defn:rgmodel}, to other well-studied combinatorial structures, in particular to $k$-uniform hypergraphs and $d$-dimensional simplicial complexes, for $k \geq 3$ and $d \geq 2$. It would be interesting to investigate the typical properties of random $k$-uniform hypergraphs, or random $d$-dimensional simplicial complexes, with given $r$-local structures. We remark that the topology of random $d$-dimensional simplicial complexes (for $d > 1$) is a very active area of current research. The most commonly studied model is perhaps the $Y_d(n,p)$ model introduced in 2006 by Linial and Meshulam \cite{lm}. (This is the random $d$-dimensional simplicial complex with $n$ vertices and full $(d-1)$-dimensional skeleton, where each $d$-dimensional simplex is included independently with probability $p$.) For $d$-dimensional simplicial complexes with $d >1$, there are two natural analogues of (graph-)connectivity, namely, $d$-collapsibility and the vanishing of the $d$th homology group over $\mathbb{R}$. In a series of recent papers \cite{top,collapse,luc,ann}, the thresholds for $d$-collapsibility and for the vanishing of the $d$th real homology group of $Y_d(n,p)$ were established by Aronshtam, Linial, \L{}uczak, Peled and Meshulam; see also \cite{kahle-survey} for a survey of related results. In \cite{lub-mesh}, Lubotzky and Meshulam introduced a new model of random 2-dimensional simplicial complexes with bounded upper degrees; these were shown to have positive coboundary expansion with high probability. There is currently much interest in developing a model of random `regular' $d$-dimensional simplicial complexes (for $d > 1$) with good typical coboundary expansion; see e.g.\ \cite{lubotzky-icm}. We believe our random $r$-local model is unlikely to solve this particular problem, but may have other interesting combinatorial and topological properties.

 \section{Definitions, background and tools}
 \label{sec:background}
 This section is structured as follows. In Section \ref{subsec:basic}, we list the basic (and standard) graph-theoretic definitions, notations and conventions which we use. In Section \ref{subsec:topgraph}, we list the standard definitions and results we need from topological graph theory. In Section \ref{subsection:top}, we list the definitions and classical results that we need from group theory and topology, and we also obtain a useful group-theoretic lemma (Lemma \ref{lemma:fin}) using the proof of the classical Third Theorem of Bieberbach, and some well-known facts about group cohomology. In Section \ref{subsection:integer-partitions}, we use classical results from the theory of integer partitions to deduce that the statement of Conjecture \ref{conj:stronger-component} holds in the case of $\mathbb{L}^1$.
 
 \subsection{Basic graph-theoretic definitions, notation and conventions}
\label{subsec:basic}
Unless otherwise stated, all graphs will be undirected and simple (that is, without loops or multiple edges); they need not be finite. An undirected, simple graph is defined to be an ordered pair of sets $(V,E)$, where $E \subset {V \choose 2}$; $V$ is called the {\em vertex-set} and $E$ the {\em edge-set}. An edge $\{v,w\}$ of a graph will often be written $vw$, for brevity.

If $G$ is a graph and $S \subset V(G)$, we let $N(S)$ denote the {\em neighbourhood} of $S$ in $G$, i.e.
$$N(S) = S \cup \{v \in V(G):\ sv \in E(G)\textrm{ for some } s \in S\}.$$
If $v \in V(G)$, we let $\Gamma(v)=\{u \in V(G):\ uv \in E(G)\}$ denote the set of neighbours of $v$. We say the graph $G$ is {\em locally finite} if each of its vertices has only finitely many neighbours.

If $G$ is a graph, and $u,v \in V(G)$ are in the same component of $G$, the {\em distance from $u$ to $v$ in $G$} is the minimum length of a path from $u$ to $v$; it is denoted by $d_{G}(u,v)$ (or by $d(u,v)$, when the graph $G$ is understood). If $G$ is a graph, and $v$ is a vertex of $G$, the {\em ball of radius $r$ around $v$} is defined by
$$B_r(v,G) := \{w \in V(G):\ d_{G}(v,w) \leq r\},$$
i.e.\ it is the set of vertices of $G$ of distance at most $r$ from $v$.

An {\em automorphism} of a graph $G$ is a bijection $\phi:V(G) \to V(G)$ such that $\{\phi(v),\phi(w)\} \in E(G)$ if and only if $\{v,w\} \in E(G)$, for all $v,w \in V(G)$. 

If $\Gamma$ is a group, and $S \subset \Gamma$ is symmetric (meaning that $S^{-1}=S$) and $\Id \notin S$, the {\em Cayley graph of $\Gamma$ with respect to $S$} is the graph with vertex-set $\Gamma$ and edge-set $\{\{g,gs\}:\ g \in \Gamma,\ s \in S\}$; we denote it by $\Cay(\Gamma,S)$.

All logarithms in this paper are to the base $e$.
 
\subsection{Background and tools from topological graph theory}
\label{subsec:topgraph}
We follow \cite{gross-tucker,nedela-survey}.

\begin{definition}
\label{defn:cover}
If $F$ and $G$ are graphs, and $p:V(F) \to V(G)$ is a graph homomorphism from $F$ to $G$, we say that $p$ is a {\em covering map} if $p$ maps $\Gamma(v)$ bijectively onto $\Gamma(p(v))$, for all $v \in V(F)$. In this case, we say that $F$ {\em covers} $G$.
\end{definition}

\noindent It is easy to see that if $F$ and $G$ are graphs with $G$ connected, and $p:V(F) \to V(G)$ is a covering map, then $p$ is surjective.

\begin{definition}
Let $F$ and $G$ be graphs, and let $p:V(F) \to V(G)$ be a covering of $G$ by $F$. The pre-image of a vertex of $G$ under $p$ is called a {\em fibre} of $p$.
\end{definition}

\begin{definition}
Let $F$ and $G$ be graphs, and let $p:V(F) \to V(G)$ be a covering of $G$ by $F$. An automorphism $\phi \in \Aut(F)$ is said to be a {\em covering transformation of $p$} if $p \circ \phi = p$. The group of covering transformations of $p$ is denoted by $\CT(p)$.
\end{definition}

\noindent Note that any covering transformation of $p$ acts on each fibre of $p$, but it need not act transitively on any fibre of $p$.

\begin{definition}
Let $F$ and $G$ be graphs, and let $p:V(F) \to V(G)$ be a covering of $G$ by $F$. We say that $p$ is a {\em normal} covering if $\CT(p)$ acts transitively on each fibre of $p$.
\end{definition}

\noindent It is well known (and easy to check) that if $F$ is a connected graph, and $p:V(F) \to V(G)$ is a covering of $G$ by $F$, then if $\CT(p)$ acts transitively on some fibre of $p$, it acts transitively on every fibre of $p$. Hence, in the previous definition, `on each fibre' may be replaced by `on some fibre'.

\begin{definition}
Let $\Gamma$ be a group, let $X$ be a set, and let $\alpha:\Gamma \times X \to X$ be an action of $\Gamma$ on $X$. For each $x \in X$, we write $\Orb_{\Gamma}(x): = \{\alpha(\gamma,x):\ \gamma \in \Gamma\}$ for the $\Gamma$-orbit of $x$, and $\Stab_{\Gamma}(x) := \{\gamma \in \Gamma:\ \alpha(\gamma,x)=x\}$ for the stabiliser of $x$ in $\Gamma$. When the group $\Gamma$ is understood, we suppress the subscript $\Gamma$.
\end{definition}

\begin{definition}
If $F$ is a graph and $\Gamma \leq \Aut(F)$, the {\em minimum displacement} of $\Gamma$ is defined to be $D(\Gamma) : = \min\{d_F(x,\gamma(x)):\ x \in V(F),\ \gamma \in \Gamma \setminus \{\Id\}\}$.
\end{definition}

\begin{definition}
\label{definition:free-action}
Let $\Gamma$ be a group, let $X$ be a set, and let $\alpha:\Gamma \times X \to X$ be an action of $\Gamma$ on $X$. We say that $\alpha$ is {\em free} if $\alpha(\gamma,x) \neq x$ for all $x \in X$ and all $\gamma \in \Gamma \setminus \{\textrm{Id}\}$.
\end{definition}

\begin{definition}(Quotient of a graph.)
\label{definition:quotient}

Let $F$ be a simple graph, and let $\Gamma \leq \Aut(F)$. Then $\Gamma$ acts on $V(F)$ via the natural left action $(\gamma,x) \mapsto \gamma(x)$, and on $E(F)$ via the natural (induced) action $(\gamma,\{x,y\}) \mapsto \{\gamma(x),\gamma(y)\}$. We define the {\em quotient graph} $F/\Gamma$ to be the multigraph whose vertices are the $\Gamma$-orbits of $V(F)$, and whose edges are the $\Gamma$-orbits of $E(F)$, where for any edge $\{x,y\} \in E(F)$, the edge $\Orb(\{x,y\})$ has endpoints $\Orb(x)$ and $\Orb(y)$. Note that $F/\Gamma$ may have loops (if $\{x,\gamma(x)\} \in E(F)$ for some $\gamma \in \Gamma$ and some $x \in V(F)$), and it may also have multiple edges (if there exist $\{u_1,u_2\},\{v_1,v_2\} \in E(F)$ with $\gamma_1(u_1) = v_1$ and $\gamma_2(u_2) = v_2$ for some $\gamma_1,\gamma_2 \in \Gamma$, but $\{\gamma(u_1),\gamma(u_2)\} \neq \{v_1,v_2\}$ for all $\gamma \in \Gamma$).
\end{definition}

\begin{definition}
Let $G$ be a graph, and let $\Gamma \leq \Aut(G)$. We say that $\Gamma$ {\em acts freely on $G$} if the natural actions of $\Gamma$ on $V(G)$ and $E(G)$ are both free actions, or equivalently, if no element of $\Gamma \setminus \{\Id\}$ fixes any vertex or edge of $G$.  
\end{definition}

\noindent The following two lemmas are well-known, and easy to check.

\begin{lemma}
\label{lemma:free}
Let $F$ be a connected (possibly infinite) graph, let $G$ be a graph, and let $p:V(F) \to V(G)$ be a covering map from $F$ to $G$. Then $\CT(p)$ acts freely on $F$.
\end{lemma}

\begin{lemma}
\label{lemma:iso}
Let $F$ and $G$ be (possibly infinite) graphs with $G$ connected, and suppose that $p:V(F) \to V(G)$ is a normal covering map from $F$ to $G$. Then there is a graph isomorphism between $G$ and $F/\CT(p)$. Moreover, if $\Gamma \leq \Aut(F)$, $\Gamma$ acts freely on $F$ and $F/\Gamma$ is a simple graph, then the natural quotient map $q:V(F) \mapsto V(F)/\Gamma$ is a normal covering map with $\CT(q)=\Gamma$.
\end{lemma}

\subsection{Some preliminaries from group theory and topology}
\label{subsection:top}
\begin{fact}
The group $\Isom(\mathbb{R}^d)$ of isometries of $d$-dimensional Euclidean space satisfies
\begin{align*} \Isom(\mathbb{R}^d) & = \{t \circ \sigma:\ t \in T(\mathbb{R}^d),\ \sigma \in O(d)\}\\
& = \{\sigma \circ t:\ t \in T(\mathbb{R}^d),\ \sigma \in O(d)\}\\
& = T(\mathbb{R}^d) \rtimes O(d),
\end{align*}
where
$$T(\mathbb{R}^d) := \{x \mapsto x+v:\ v \in \mathbb{R}^d\}$$
denotes the group of all translations in $\mathbb{R}^d$, and $O(d) \leq \GL(\mathbb{R}^d)$ denotes the group of all real orthogonal $d \times d$ matrices.
\end{fact}

\begin{fact}
\label{fact:aut-ld}
For any $d \in \mathbb{N}$, we have
\begin{align*} \Aut(\mathbb{L}^d) & = \{t \circ \sigma:\ t \in T(\mathbb{Z}^d),\ \sigma \in B_d\}\\
& = \{\sigma \circ t:\ t \in T(\mathbb{Z}^d),\ \sigma \in B_d\}\\\
&= T(\mathbb{Z}^d) \rtimes B_d,
\end{align*}
where
$$T(\mathbb{Z}^d) :=  \{x \mapsto x+v:\ v \in \mathbb{Z}^d\}$$
denotes the group of all translations by elements of $\mathbb{Z}^d$, and
$$B_d = \{\sigma \in GL(\mathbb{R}^d):\ \sigma(\{\pm e_i:\ i \in [d]\}) = \{\pm e_i: \ i \in [d]\}\},$$
denotes the {\em $d$-dimensional hyperoctahedral group}, which is the symmetry group of the $d$-dimensional (solid) cube with set of vertices $\{-1,1\}^d$, and can be identified with the permutation group
$$\{\sigma \in \Sym([d] \cup \{-i:\ i \in [d]\}):\ \sigma(-i) = -\sigma(i)\ \forall i\},$$
in the natural way (identifying $e_i$ with $i$ and $-e_i$ with $-i$ for all $i \in [d]$). We therefore have $|B_d| = 2^d d!$.
\end{fact}

\begin{remark}
\label{remark:obvious-embedding}
It is clear that every element of $\Aut(\mathbb{L}^d)$ can be uniquely extended to an element of $\Isom(\mathbb{R}^d)$. We can therefore view $\Aut(\mathbb{L}^d)$ as a subgroup of $\Isom(\mathbb{R}^d)$.
\end{remark}

\begin{definition}
If $X$ is a topological space, and $\Gamma$ is a group acting on $X$, the {\em orbit space} $X/ \Gamma$ is the (topological) quotient space $X/ \sim$, where $x \sim y$ iff $y \in \Orb_{\Gamma}(x)$, i.e. iff $x$ and $y$ are in the same $\Gamma$-orbit.
\end{definition}

\begin{definition}
If $X$ is a topological space, a group $\Gamma$ of homeomorphisms of $X$ is said to be {\em discrete} if the relative topology on $\Gamma$ (induced by the compact open topology on the group of all homeomorphisms of $X$) is the discrete topology.
\end{definition}

\begin{definition}
If $X$ is a topological space, and $\Gamma$ is a discrete group of homeomorphisms of $X$, we say that $\Gamma$ acts {\em properly discontinuously} on $X$ if for any $x,y \in X$, there exist open neighbourhoods $U$ of $x$ and $V$ of $y$ such that $|\{\gamma \in \Gamma:\ \gamma(U) \cap V \neq \emptyset\}| < \infty$.
\end{definition}

\begin{fact}
\label{fact:discrete}
If $\Gamma \leq \Isom(\mathbb{R}^d)$, then $\Gamma$ is discrete if and only if for any $x \in \mathbb{R}^d$, the orbit $\{\gamma(x):\ \gamma \in \Gamma\}$ is a discrete subset of $\mathbb{R}^d$. Hence, $\Aut(\mathbb{L}^d)$ is a discrete subgroup of $\Isom(\mathbb{R}^d)$.
\end{fact}

\begin{fact}
\label{fact:proper-disc}
If $\Gamma \leq \Isom(\mathbb{R}^d)$ is discrete, then $\Gamma$ acts properly discontinuously on $\mathbb{R}^d$. (Note that it is clear directly from the definition that $\Aut(\mathbb{L}^d)$, and any subgroup thereof, acts properly discontinuously on $\mathbb{R}^d$.)
\end{fact}

\begin{definition}
Let $\Gamma$ be a discrete subgroup of $\Isom(\mathbb{R}^d)$. The {\em translation subgroup} $T_{\Gamma}$ of $\Gamma$ is the (normal) subgroup of all translations in $\Gamma$. The {\em lattice of translations} of $\Gamma$ is the lattice $L_{\Gamma} := \{\gamma(0):\ \gamma \in T_{\Gamma}\} \subset \mathbb{R}^d$. We have $L_{\Gamma} \cong \mathbb{Z}^r$ for some $r \in \{0,1,\ldots,d\}$; the integer $r$ is called the {\em rank} of the lattice $L_{\Gamma}$. We say that $\Gamma$ is a {\em pure-translation subgroup} if $\Gamma = T_{\Gamma}$.
\end{definition}

\begin{definition}
\label{definition:crystallographic}
A discrete subgroup $\Gamma \leq \Isom(\mathbb{R}^d)$ is said to be a {\em $d$-dimensional crystallographic group} if its lattice of translations has rank $d$ (or, equivalently, if the orbit space $\mathbb{R}^d/\Gamma$ is compact).
\end{definition}

\begin{definition}
If $\Gamma$ is a $d$-dimensional crystallographic group, its {\em point group} $P_{\Gamma}$ is defined by
 $$P_{\Gamma} = \{\sigma \in O(d):\ t \circ \sigma \in \Gamma \textrm{ for some }t \in T(\mathbb{R}^d)\}.$$
 \end{definition}

\begin{fact}
\label{fact:point-group-finite}
If $\Gamma$ is a $d$-dimensional crystallographic group, then its point group $P_{\Gamma}$ is finite, and $P_{\Gamma} \cong \Gamma/T_{\Gamma}$. Hence, $\Gamma$ is a pure-translation subgroup if and only if $P_{\Gamma} = \{\Id\}$.
\end{fact}

\begin{remark}
\label{remark:finite-quotient}
If $\Gamma \leq \Aut(\mathbb{L}^d)$ and $|\mathbb{Z}^d/\Gamma| < \infty$, then $\Gamma$ is a $d$-dimensional crystallographic group, viewing $\Aut(\mathbb{L}^d)$ as a subgroup of $\Isom(\mathbb{R}^d)$ in the usual way (see Remark \ref{remark:obvious-embedding}). Moreover, in this case, $L_{\Gamma}$ is a sublattice of $\mathbb{Z}^d$ (with rank $d$), and $P_{\Gamma}$ is a subgroup of $B_d$.
\end{remark}

We will also need two group-theoretic lemmas. The first is well-known; we provide a proof for completeness.

\begin{lemma}
\label{lemma:action}
If $\Gamma$ is a $d$-dimensional crystallographic group, then $\sigma(L_{\Gamma}) = L_{\Gamma}$ for all $\sigma \in P_{\Gamma}$, i.e.\ the point group of $\Gamma$ acts on the lattice of translations of $\Gamma$.
\end{lemma}
\begin{proof}
Let $\Gamma$ be a $d$-dimensional crystallographic group and let $\sigma \in P_{\Gamma}$. Let $a \in L_{\Gamma}$. Then there exists $\gamma \in T_{\Gamma}$ such that $a = \gamma(0)$ and there exists $t \in T(\mathbb{R}^d)$ such that $t \circ \sigma \in \Gamma$. Then, since $T_{\Gamma}$ is a normal subgroup of $\Gamma$ and $P_{\Gamma}$ is a group of linear transformations, we have $\sigma \circ \gamma \circ \sigma^{-1} = (t \circ \sigma) \circ \gamma \circ (t \circ \sigma)^{-1} \in T_{\Gamma}$, and $\sigma(a) = \sigma(\gamma(0)) = (\sigma \circ \gamma \circ \sigma^{-1})(0) \in L_{\Gamma}$. Hence, $P_{\Gamma}$ acts on $L_{\Gamma}$, as required.
\end{proof}

The next lemma is a straightforward consequence of one of the proofs of Bieberbach's Third Theorem.
\begin{lemma}
\label{lemma:fin}
For each $d \in \mathbb{N}$, there exists $k=k(d) \in \mathbb{N}$ such that for any rank-$d$ lattice $L$ in $\mathbb{R}^d$ and any group $P \leq O(d)$, there are at most $k(d)$ $d$-dimensional crystallographic groups $\Gamma \leq \Isom(\mathbb{L}^d)$ with $L_{\Gamma}=L$ and $P_{\Gamma}=P$, up to conjugation by translations.
\end{lemma}

\begin{proof}
We follow the exposition in \cite{hiller}. Let $L$ be a rank-$d$ lattice in $\mathbb{R}^d$, and let $P \leq O(d)$. We may assume that there exists at least one $d$-dimensional crystallographic group with lattice of translations $L$ and point group $P$, otherwise we are done. By Fact \ref{fact:point-group-finite}, $P$ is a finite group. By Lemma \ref{lemma:action}, $P$ acts on $L$, and therefore $P$ acts on the torus $\mathbb{R}^d/L$, which can be viewed as a $P$-module.

For each $d$-dimensional crystallographic group $\Gamma$ with lattice of translations $L$ and point group $P$, we define a map $\phi_{\Gamma}: P \to \mathbb{R}^d$ as follows. For each $\sigma \in P$, choose a translation $t_c:x \mapsto x+c$ such that $t_c \circ \sigma \in \Gamma$, and define $\phi_{\Gamma}(\sigma) = c$. Now let $s_{\Gamma}: P \to \mathbb{R}^d/L$ be the composition of $\phi_{\Gamma}$ with the natural quotient map $q:\mathbb{R}^d \to \mathbb{R}^d/L;\ x \mapsto x+L$, i.e., define
$$s_{\Gamma}(\sigma) = \phi_{\Gamma}(\sigma)+L\quad \forall \sigma \in P.$$

Observe that the map $s=s_{\Gamma}$ satisfies the two conditions
\begin{align} s(\Id) &= 0,\label{eq:cocycle-1}\\
s(\sigma \tau) &= s(\sigma)+\sigma(s(\tau))\quad \forall \sigma,\tau \in P. \label{eq:cocycle-2}
\end{align}
In the language of the group cohomology, the set of maps $s:P\to \mathbb{R}^d/L$ satisfying the conditions (\ref{eq:cocycle-1}) and (\ref{eq:cocycle-2}) is termed the {\em group of 1-cocycles} (of $P$, with coefficients in $\mathbb{R}^d/L$), and is denoted by $Z^{1}(P,\mathbb{R}^d/L)$. (It is easily checked that the set of all 1-cocycles forms an Abelian group under the operation of pointwise addition.) 

Observe also that $\Gamma$ can be recovered from $s_{\Gamma}$, since
$$\Gamma = \{t_a \circ \sigma:\ \sigma \in P,\ a \in s_{\Gamma}(\sigma)\}.$$
Indeed, clearly $\Gamma \supset \{t_a \circ \sigma:\ \sigma \in P,\ a \in s_{\Gamma}(\sigma)\}$. On the other hand, if $\gamma \in \Gamma$, then $\gamma = t_{a} \circ \sigma$ for some $\sigma \in P$. Let $c = \phi_{\Gamma}(\sigma)$. Then we have
$$\Gamma \ni t_a \circ \sigma \circ (t_{c} \circ \sigma)^{-1} = t_{a-c},$$
so $a-c \in L$ and therefore $a \in c+L = s_{\Gamma}(\sigma)$. Therefore, $\gamma \in \{t_a \circ \sigma:\ \sigma \in P,\ a \in s_{\Gamma}(\sigma)\}$. It follows that $\Gamma \subset \{t_a \circ \sigma:\ \sigma \in P,\ a \in s_{\Gamma}(\sigma)\}$.

In fact, it is easily checked that for {\em any} 1-cocycle $s$, $\Gamma_s: = \{t_a \circ \sigma:\ \sigma \in P,\ a \in s(\sigma)\}$ is a $d$-dimensional crystallographic group with lattice of translations $L$ and point group $P$. Indeed, the group axioms follow straightforwardly from properties (\ref{eq:cocycle-1}) and (\ref{eq:cocycle-2}) of $s$. The group $\Gamma_s$ is discrete, since for any $x \in \mathbb{R}^d$,
$$\{\gamma(x):\ \gamma \in \Gamma_s\} = \bigcup_{\sigma \in P} \{\sigma(x)+\phi(\sigma)+v:\ v \in L\},$$
which is a discrete subset of $\mathbb{R}^d$, since $|P| < \infty$. (See Fact \ref{fact:discrete}.) Finally, the lattice of translations of $\Gamma_s$ is $L$, by property (\ref{eq:cocycle-1}) of $s$. Hence, $\Gamma_s$ is indeed a $d$-dimensional crystallographic group with lattice of translations $L$ and point group $P$.

It follows that the set of $d$-dimensional crystallographic groups with lattice of translations $L$ and point group $P$ is in one-to-one correspondence with the group of 1-cocycles $Z^{1}(P,\mathbb{R}^d/L)$.

We now consider the effect on $s_{\Gamma}$ of conjugating $\Gamma$ by a translation $t_{a}:x \mapsto x+a$, where $a \in \mathbb{R}$. Let $\sigma \in P$, and let $c= \phi_{\Gamma}(\sigma)$. Then for all $x \in \mathbb{R}^d$, we have
$$(t_{a} \circ (t_{c} \circ \sigma) \circ t_{a}^{-1})(x) = a+c + \sigma(x-a) = a-\sigma(a)+c+\sigma(x) = (t_{a-\sigma(a)+c} \circ \sigma)(x),$$
and therefore
$$t_{a} \circ (t_{c} \circ \sigma) \circ t_{a}^{-1} = t_{c+a-\sigma(a)} \circ \sigma.$$
Passing to the quotient $\mathbb{R}^d/L$, we see that
$$s_{t_a \Gamma t_{-a}}(\sigma) = s_{\Gamma}(\sigma)+a-\sigma(a)\quad \forall \sigma \in P,$$
i.e.\ conjugating $\Gamma$ by $t_{a}$ changes $s_{\Gamma}$ by adding to $s_{\Gamma}$ the 1-cocycle $s_{a}$, where we define $s_a(\sigma) = a-\sigma(a)$ for all $\sigma \in P$. In the language of group cohomology, the set $\{s_{a}:\ a \in \mathbb{R}^d\}$ is termed the {\em group of 1-coboundaries} (of $P$, with coefficients in $\mathbb{R}^d/L$), and is denoted by $B^{1}(P,\mathbb{R}^d/L)$. (Again, it is clear that the set of 1-coboundaries forms an Abelian group under the operation of pointwise addition.) It follows that the set of equivalence classes of $d$-dimensional crystallographic group with lattice of translations $L$ and point group $P$ (where two such groups are defined to be equivalent if they are conjugate via a translation by an element of $\mathbb{R}^d$), is in one-to-one correspondence with the quotient group
\begin{equation}\label{eq:cohomology} Z^{1}(P,\mathbb{R}^d/L) / B^{1}(P,\mathbb{R}^d/L),\end{equation}
which is precisely the {\em 1-dimensional cohomology group of $P$ with coefficients in $\mathbb{R}^d/L$} (denoted $H^1(P,\mathbb{R}^d/L)$). It follows from well-known facts from the theory of group cohomology that the cohomology group (\ref{eq:cohomology}) is a finite group. Indeed, for any finite group $P$, any $P$-module $M$ and any $n \in \mathbb{N}$, the $n$th cohomology group of $P$ with coefficients in $M$, denoted $H^n(P,M)$, is a finitely generated Abelian group with every element of order dividing $|P|$, so is finite; see \cite{brown}. Alternatively, see \cite[p.~130]{schwarz-cohomology} for an elementary and direct proof of the finiteness of (\ref{eq:cohomology}). This completes the proof of the lemma.
\end{proof}

\subsection{The $\mathbb{L}^1$ case, via the theory of integer partitions}
\label{subsection:integer-partitions}
In this section, we use known results from the theory of integer partitions to deduce that the statement of Conjecture \ref{conj:stronger-component} holds for $d=1$, with $r_0(1)=1$.

Observe that a graph is $r$-locally $\mathbb{L}^1$ if and only if it is a vertex-disjoint union of cycles each of length at least $2r+2$, so choosing an unlabelled $n$-vertex graph uniformly at random from the set of all unlabelled, $n$-vertex graphs which are $r$-locally $\mathbb{L}^1$ corresponds precisely to choosing a partition of $n$ uniformly at random from the set of all partitions of $n$ with each part of size at least $2r+2$. (Recall that if $n \in \mathbb{N}$, a {\em partition} of $n$ is a non-increasing sequence of positive integers $(\lambda_1,\lambda_2,\ldots,\lambda_l)$ with $\sum_{i=1}^{l} \lambda_i= n$. If $\lambda = (\lambda_1,\ldots,\lambda_l)$ is a partition of $n$, the $\lambda_i$'s are called the {\em parts} of $\lambda$.) For $n \in \mathbb{N}$, we let $p(n)$ denote the number of partitions of $n$. Hardy and Ramanujan \cite{hardy-ramanujan} proved in 1918 that
\begin{equation}\label{eq:partition-asymptotics} p(n)  = (1+O(1/\sqrt{n})) \frac{1}{4n \sqrt{3}} \exp(\pi \sqrt {2n/3}).\end{equation}
(This was one of the first applications of the celebrated Hardy-Littlewood circle method. It was also proved independently by Uspensky \cite{usp} in 1920.) It suffices for us to prove the following.

\begin{proposition}
\label{prop:partitions}
Let $r \in \mathbb{N}$. For each $n \geq 2r+2$, let $\lambda$ be a partition chosen uniformly at random from the set of all partitions of $n$ with each part of size at least $2r+2$. Then with probability at least $1-O_r(1/n^2)$, the largest part $\lambda_1$ of the partition $\lambda$ satisfies $\lambda_1 = \Theta_r(\sqrt{n} \log n)$.
\end{proposition}
\begin{proof}
For each $s \in \mathbb{N}$, let $p_{\geq s}(n)$ denote the number of partitions of $n$ with all parts of size at least $s$. It follows from results of Fristedt \cite{fristedt} that for any $s \in \mathbb{N}$, we have
\begin{equation} \label{eq:fristedt} p_{\geq s}(n) = \Theta_s(n^{-(s-1)/2})p(n),\end{equation}
Hence, the probability that a random partition of $n$ has all its parts of size at least $2r+2$ is
\begin{equation}\label{eq:prob-part}\frac{p_{\geq 2r+2}(n)}{p(n)} = \Theta_r(n^{-(2r+1)/2}).\end{equation}

On the other hand, for any $s \in \mathbb{N}$, the number of partitions of $n$ with a part of size $s$ is $p(n-s)$. Hence, for any $m \in [n]$, the probability that a uniform random partition of $n$ has at least one part of size at least $m$ is at most 
$$n\frac{p(n-m)}{p(n)}$$
(using the fact that $p(n)$ is a non-decreasing function of $n$). Choosing $m = C \sqrt{n} \log n$, and using (\ref{eq:partition-asymptotics}) and performing Taylor expansions, we obtain
$$n \frac{p(n-m)}{p(n)} = n \exp(-(1+o(1))\frac{\pi}{\sqrt{6}}C \log n) = O(1/n^C)$$
if $C \geq 3$. Hence, a uniform random partition of $n$ has largest part of size $O(C\sqrt{n} \log n)$ with probability at least $1-O(1/n^C)$ provided $C \geq 3$.

Moreover, Erd\H{o}s and Lehner \cite{erdos-lehner} proved that for any fixed $x \in \mathbb{R}$, if $\lambda$ is a uniform random partition of $n$, then its largest part $\lambda_1$ satisfies
$$\lambda_1 \leq \frac{\sqrt{n}\log n}{c_0} + x\sqrt{n}$$
with probability
$$(1+o(1)) \exp(-\tfrac{2}{c_0} \exp(-\tfrac{1}{2}c_0x)),$$
where $c_0:=\pi \sqrt{2/3}$. It follows from their proof that the same statement holds when $-\log \log n \leq x \leq 2 \log n$, and therefore
$$\lambda_1 \geq \frac{\sqrt{n}\log n}{c_0} - \sqrt{n}\log \log n$$
with probability at least $1-O(\exp(-\tfrac{2}{c_0} (\log n)^{c_0/2}))$. Hence, for all $C \geq 3$, a uniform random partition of $n$ has largest part of size $\Theta_{C}(\sqrt{n} \log n)$ with probability at least $1-O(1/n^C)$. By choosing $C$ to be sufficiently large depending on $r$, it follows from this and (\ref{eq:prob-part}) that if we choose a partition of $n$ uniformly at random from the set of all partitions with all parts of size at least $2r+2$, then with probability at least $1-O_r(1/n^2)$, the largest part has size $\Theta_r(\sqrt{n} \log n)$. 
\end{proof}

The following corollary is immediate.

\begin{corollary}
\label{corr:1locally}
Let $r \in \mathbb{N}$. For each $n \geq 2r+2$, with probability at least $1-O_r(1/n^2)$, the largest component of $G_n(\mathbb{L}^1,r)$ has order $\Theta_r(\sqrt{n} \log n)$.
\end{corollary}

\begin{remark}
\label{remark:1-stronger}
Observe that using (\ref{eq:partition-asymptotics}), (\ref{eq:fristedt}) and the correspondence above, we see that for any $r \in \mathbb{N}$, the number $a_{1,r}(n)$ of unlabelled, $n$-vertex graphs that are $r$-locally $\mathbb{L}^1$ satisfies
$$a_{1,r}(n) = p_{\geq 2r+2}(n) = \Theta_r(n^{-(2r+1)/2})p(n) = \Theta_r(n^{-r})\exp(\pi \sqrt {2n/3}),$$
so
$$\label{eq:more-precise} \log a_{1,r}(n) = \pi \sqrt {2n/3} -r \log n + O_r(1).$$
(Note that this is a much more precise estimate than in Theorem \ref{thm:enumeration}.)
\end{remark}

\section{Proof of Theorem \ref{thm:enumeration}}
\label{sec:proof-enum}
In this section, we prove Theorem \ref{thm:enumeration}, our enumeration result for graphs that are $r$-locally $\mathbb{L}^d$, using the strategy sketched in the Introduction (page \pageref{proof-sketch}).

We first show that the connected unlabelled $n$-vertex graphs that are $r$-locally $\mathbb{L}^d$, are in one-to-one correspondence with conjugacy-classes of subgroups $\Gamma \leq \Aut(\mathbb{L}^d)$ such that $|\mathbb{Z}^d/\Gamma| = n$ and $D(\Gamma) \geq 2r+2$, provided $r \geq 2+1_{\{d \geq 3\}}$. 

We recall the following straightforward result, proved in \cite{be}.

\begin{proposition}
\label{prop:d-unique} 
Let $d \in \mathbb{N}$ with $d\geq 2$, and let $G$ be a connected graph that is 2-locally $\mathbb{L}^d$. Let $v_0 \in V(G)$, and let $\psi:\Link_{2}(0,\mathbb{L}^d) \to \Link_2(v_0,G)$ be a graph isomorphism with $\psi(0)=v_0$. Then there is at most one covering map $p$ from $\mathbb{L}^d$ to $G$ such that $p(x)=\psi(x)$ for all $x \in N(0)$.
\end{proposition}

(We note that this was stated in \cite{be} under a slightly stronger hypothesis, but it is clear that the proof in \cite{be} relies only upon $G$ being 2-locally $\mathbb{L}^d$.) We also need the following lemmas.

\begin{lemma}
\label{lemma:conjugacy-2}
Suppose $G_1$ and $G_2$ are connected graphs which are $2$-locally-$\mathbb{L}^2$. Suppose $p_1:\mathbb{Z}^2 \to V(G_1)$ is a normal covering map from $\mathbb{L}^2$ to $G_1$, and $p_2:\mathbb{Z}^2 \to V(G_2)$ is a normal covering map from $\mathbb{L}^2$ to $G_2$. Then $G_1$ is isomorphic to $G_2$ if and only if the subgroups $\CT(p_1)$ and $\CT(p_2)$ are conjugate in $\Aut(\mathbb{L}^2)$.
\end{lemma}
\begin{proof}
Suppose that $G_1 \cong G_2$. Let $f:V(G_2) \to V(G_1)$ be a graph isomorphism from $G_2$ to $G_1$. Then $p:=f \circ p_2:\mathbb{Z}^2 \to V(G_1)$ is a normal covering map from $\mathbb{L}^2$ to $G_1$. Choose $\phi \in \Aut(\mathbb{L}^2)$ such that $p_1 \circ \phi$ and $p$ agree with one another on $N(0)$. Then, by Proposition \ref{prop:d-unique}, we have $p_1 \circ \phi = p$. Hence, $\CT(p_2) = \CT(f \circ p_2) = \CT(p) = \CT(p_1 \circ \phi) = \phi^{-1} \CT(p_1) \phi$.

Conversely, suppose $\CT(p_2) = \phi^{-1} \CT(p_1) \phi$ for some $\phi \in \Aut(\mathbb{L}^2)$. Then define $p = p_1 \circ \phi$. Then $p:\mathbb{Z}^2 \to V(G_1)$ is a normal covering map from $\mathbb{L}^2$ to $G_1$, and $\CT(p) = \phi^{-1} \CT(p_1) \phi = \CT(p_2)$. But by Lemma \ref{lemma:iso}, $G_1 \cong \mathbb{L}^2/\CT(p)$ and $G_2 \cong \mathbb{L}^2/\CT(p_2) = \mathbb{L}^2/\CT(p)$, so $G_1 \cong G_2$.  
\end{proof}

In exactly the same way, we may prove the following.
\begin{lemma}
\label{lemma:conjugacy-d}
Let $d \in \mathbb{N}$ with $d\geq 2$. Suppose $G_1$ and $G_2$ are connected graphs which are $3$-locally $\mathbb{L}^d$. Suppose $p_1:\mathbb{Z}^d \to V(G_1)$ is a normal covering map from $\mathbb{L}^d$ to $G_1$, and $p_2:\mathbb{Z}^d \to V(G_2)$ is a normal covering map from $\mathbb{L}^d$ to $G_2$. Then $G_1$ is isomorphic to $G_2$ if and only if the subgroups $\CT(p_1)$ and $\CT(p_2)$ are conjugate in $\Aut(\mathbb{L}^d)$. 
\end{lemma}

Clearly, a subgroup $\Gamma \leq \Aut(\mathbb{L}^d)$ with minimum displacement at least 2, acts freely on $\mathbb{L}^d$. Hence, from Theorem \ref{thm:quotient-d} and Lemmas \ref{lemma:iso} and \ref{lemma:conjugacy-2}, we may deduce the following.
\begin{corollary}
\label{corr:num-2}
If $r \in \mathbb{N}$ with $r \geq 2$, then the unlabelled, connected, $n$-vertex graphs which are $r$-locally $\mathbb{L}^2$ are in an explicit one-to-one correspondence with conjugacy-classes of subgroups $\Gamma \leq \Aut(\mathbb{L}^2)$ with minimum displacement at least $2r+2$ and with $|\mathbb{Z}^2/\Gamma|=n$; this correspondence is given by simply taking the quotient graph of $\mathbb{L}^2$ by $\Gamma$, i.e., $\Gamma \leftrightarrow \mathbb{L}^2/\Gamma$.
\end{corollary}

Similarly, using Lemma \ref{lemma:conjugacy-d} in place of Lemma \ref{lemma:conjugacy-2}, we may deduce the following.

\begin{corollary}
\label{corr:num-d}
If $r,d \in \mathbb{N}$ with $r,d \geq 3$, then the unlabelled, connected, $n$-vertex graphs which are $r$-locally $\mathbb{L}^d$ are in an explicit one-to-one correspondence with the conjugacy-classes of subgroups $\Gamma \leq \Aut(\mathbb{L}^d)$ with minimum displacement at least $2r+2$ and with $|\mathbb{Z}^d/\Gamma| = n$; this correspondence is given by taking the quotient graph of $\mathbb{L}^d$ by $\Gamma$, i.e., $\Gamma \leftrightarrow \mathbb{L}^d/\Gamma$.
\end{corollary}

Our next aim is to deal with the subgroups $\Gamma$ (as in Corollaries \ref{corr:num-2} and \ref{corr:num-d}) that are {\em pure-translation subgroups}. Clearly, a pure-translation subgroup $\Gamma \leq \Aut(\mathbb{L}^d)$ is determined by its lattice of translations $L_{\Gamma}$, which must be a sublattice of $\mathbb{Z}^d$. We now recall some definitions and basic facts concerning sublattices of $\mathbb{Z}^d$.

A {\em sublattice of $\mathbb{Z}^d$} is a subgroup $L$ of $\mathbb{Z}^d$; such a subgroup must be isomorphic (as a group) to $\mathbb{Z}^{R}$ for some $R \leq d$, and the integer $R$ is called the {\em rank} of $L$. If $L$ is a sublattice of $\mathbb{Z}^d$ of rank $d$, $\mathbb{Z}^d/L$ is finite, and $|\mathbb{Z}^d/L|$ is called the {\em index} of $L$; if $L$ has rank less than $d$, then its index is infinite. The {\em minimum distance} of $L$ is defined to be $\min\{\|x\|_1:\ x \in L \setminus \{0\}\}$. If $L$ is a sublattice of $\mathbb{Z}^d$ and $\sigma \in \Aut(\mathbb{L}^d)$, we say that $L$ is {\em $\sigma$-invariant} if $\sigma(L)=L$.

It is well-known (and easy to see) that there is a one-to-one correspondence between sublattices of $\mathbb{Z}^d$ with index $n$, and upper-triangular, integer matrices $B = (b_{ij})_{i,j \in [d]}$ with $0 \leq b_{ij} < b_{ii}$ for all $j > i$ and all $i \in [d]$, and with $\prod_{i=1}^{d} b_{ii} = n$. The matrix $B$ corresponds to the sublattice for which a $\mathbb{Z}$-basis is the set of columns of $B$, i.e. the set\
\begin{equation}
\label{eq:matrix-corr}
\left\{\sum_{i=1}^{j} b_{ij} e_i:\ j \in [d]\right\}.
\end{equation}
It follows that the number of sublattices of $\mathbb{Z}^d$ of index $n$ is
$$\sum_{c_1,\ldots,c_d \in \mathbb{N}:\atop c_1 c_2 \ldots c_d = n} c_1^{d-1} c_2^{d-2} c_3^{d-3} \ldots c_{d-2}^2 c_{d-1}.$$
Hence, the number of sublattices of $\mathbb{Z}^d$ of index at most $x$ is
$$\sum_{c_1 c_2 \ldots c_d \leq x} c_1^{d-1} c_2^{d-2} c_3^{d-3} \ldots c_{d-2}^2 c_{d-1}:=N_d(x).$$
It is well-known that
\begin{equation} \label{eq:zeta} N_d(x) = (1+O((\log x)/x))\frac{x^d}{d}\prod_{i=2}^{d} \zeta(i).\end{equation}
(For an elementary proof of this, see for example \cite{gillet-grayson}.) We will use this estimate in the sequel.

Observe that any pure-translation subgroup $\Gamma \leq \Aut(\mathbb{L}^d)$ is invariant under conjugation by any translation. Hence, two pure-translation subgroups $\Gamma,\Gamma' \leq \Aut(\mathbb{L}^d)$ are conjugate in $\Aut(\mathbb{L}^d)$ if and only if they are conjugate via an element of the hyperoctahedral group $B_d$. Note that if $\Gamma$ and $\Gamma'$ are pure-translation subgroups of $\Aut(\mathbb{L}^d)$, and $\sigma \in B_d$, then $\sigma \Gamma \sigma^{-1} = \Gamma'$ if and only if $\sigma(L_{\Gamma}) = L_{\Gamma'}$. Hence, we obtain the following.

\begin{fact}
\label{fact:pt} The number of conjugacy-classes of pure-translation subgroups $\Gamma \leq \Aut(\mathbb{L}^d)$ with $|\mathbb{Z}^d/\Gamma| = n$ is precisely the number of $B_d$-orbits of sublattices $L$ of $\mathbb{Z}^d$ with $|\mathbb{Z}^d/L| = n$.
\end{fact}

Notice that any sublattice of $\mathbb{Z}^d$ is invariant under both $\Id$ and $-\Id$. Our next result says that for any other element $\sigma \in B_d \setminus \{\pm \Id\}$, a uniform random sublattice of $\mathbb{Z}^d$ with index at most $x$ is very unlikely to be invariant under $\sigma$. 

\begin{lemma}
\label{lemma:asym}
Let $\sigma \in B_d \setminus \{\pm \Id\}$. Let $N_{d,\sigma}(n)$ denote the number of index-$n$ sublattices of $\mathbb{Z}^d$ which are invariant under $\sigma$. Then for $x > 1$, we have
$$\sum_{n \leq x} N_{d,\sigma}(n) \leq O_{d}(1) x^{d-1+O(1/\log \log x)}.$$
\end{lemma}
\begin{proof}
For each $i \in [d]$, let $\ell_i = \{e_i,-e_i\}$, and let $P_d = \{\ell_i:\ i \in [d]\}$. Observe that $B_d$ acts on $P_d$. Let $\sigma \in B_d \setminus \{\pm \Id\}$. Suppose firstly that $\sigma$ acts trivially on $P_d$. Then, since $\sigma \neq \pm \Id$, there exist $i,j \in [d]$ such that $\sigma(e_i)=-e_i$ and $\sigma(e_j) = e_j$. Without loss of generality, we may assume that $i=1$ and $j=2$. Let $L$ be a $\sigma$-invariant sublattice of $\mathbb{Z}^d$ with index at most $x$. By the linearity of $\sigma$, the matrix $B$ corresponding to $L$ satisfies
$$-b_{12}e_{1}+b_{22}e_{2} \in \langle B \rangle,$$
where $\langle B \rangle$ denotes the $\mathbb{Z}$-linear span of the columns of $B$. This condition holds only if there exist $\lambda_1,\lambda_2 \in \mathbb{Z}$ such that
$$-b_{12}e_{1}+b_{22}e_{2} = \lambda_1 b_{11}e_1+\lambda_2 (b_{12}e_1+b_{22}e_{2}),$$
which implies $\lambda_2=1$ and $\lambda_1b_{11} = 2b_{12}$. Since $b_{12} < b_{11}$, this condition holds only if $\lambda_1 \in \{0,-1\}$, i.e. only if $b_{12} = 0$ or $b_{12} = b_{11}/2$. Therefore,
$$N_{d,\sigma}(n) \leq 2\sum_{c_1 c_2 \ldots c_d = n} c_1^{d-2} c_2^{d-2} c_3^{d-3} \ldots c_{d-2}^2 c_{d-1},$$
so
\begin{align*} \sum_{n \leq x} N_{d,\sigma}(n) & \leq 2 \sum_{c_1 c_2 \ldots c_d \leq x} c_1^{d-2} c_2^{d-2} c_3^{d-3} \ldots c_{d-2}^2 c_{d-1}\\
& \leq 2\sum_{c_2 c_3 \ldots c_d \leq x} c_2^{d-2} c_3^{d-3} \ldots c_{d-2}^2 c_{d-1} \sum_{c_1 \leq x/(c_2 c_3 \ldots c_d)} c_1^{d-2}\\
& \leq 2\sum_{c_2 c_3 \ldots c_d \leq x} c_2^{d-2} c_3^{d-3} \ldots c_{d-2}^2 c_{d-1}(x/(c_2 c_3 \ldots c_d))^{d-1}\\
& = 2x^{d-1} \sum_{c_2 c_3 \ldots c_d \leq x} c_2^{-1} c_3^{-2} \ldots c_{d-1}^{-(d-2)} c_{d}^{-(d-1)}\\
 & \leq 2x^{d-1} \left(\sum_{c_2 \leq x} c_2^{-1}\right) \left(\sum_{c_3 \leq x} c_3^{-2}\right) \ldots \left( \sum_{c_d \leq x} c_d^{-(d-1)}\right)\\
 & = O(x^{d-1} \log x).
 \end{align*}
Now suppose instead that $\sigma$ does not act trivially on $P_d$. Then there exists $i \in [d]$ such that $\sigma(\ell_i) \neq \ell_i$. Without loss of generality, we may assume that $\sigma(\ell_1) = \ell_2$. Hence, $\sigma(e_1)=\pm e_2$. Since $\sigma (L) = L$ if and only if $(-\sigma)(L) = L$, by considering $-\sigma$ if necessary, we may assume that $\sigma(e_1) = e_2$. Then $\sigma(e_2) \neq \pm e_2$. There are two cases.
\begin{description}
\item[Case 1:] $\sigma(e_2) = \pm e_1$.
\item[Case 2:] $\sigma(e_2) = \pm e_j$, for some $j \geq 3$.
\end{description}
First suppose that Case 1 occurs. Then, if $\sigma (L)=L$, the corresponding matrix $B$ satisfies
$$b_{12}e_{2} \pm b_{22}e_{1} \in \langle B \rangle.$$
This condition holds if and only if there exist $\lambda_1,\lambda_2 \in \mathbb{Z}$ such that
$$b_{12}e_{2}\pm b_{22}e_{1} = \lambda_1 b_{11}e_1+\lambda_2 (b_{12}e_1+b_{22}e_{2}),$$
which implies
$$\pm b_{22} = \lambda_1 b_{11} + \lambda_2 b_{12},\quad b_{12} = \lambda_2 b_{22},$$
which implies $\lambda_1 b_{11} = \pm b_{22} - b_{12}^2/b_{22}$, which implies
$$b_{11} \mid \pm b_{22} - b_{12}^2/b_{22}.$$
Once we have chosen $b_{d,d},b_{d-1,d-1},\ldots,b_{22}$, there are at most $x/(b_{22} b_{33} \ldots b_{dd})$ choices for each of $b_{12},\ldots,b_{1,d}$ (since we must have $b_{1,j} \leq x/(b_{22} b_{33} \ldots b_{dd})$ for each $j \geq 2$), and then there are at most $x^{O(1/\log \log x)}$ choices for $b_{11}$, since 
$$b_{11} \mid \pm b_{22} - b_{12}^2/b_{22},$$
$$|\pm b_{22} - b_{12}^2/b_{22}| \leq x^2+1 \leq 2x^2,$$
and the divisor function $\tau = \tau(m)$ ($=$ number of divisors of $m$) satisfies
$$\tau(m) \leq m^{O(1/\log \log m)}.$$
Hence,
\begin{align*} \sum_{n \leq x} N_{d,\sigma}(n) & \leq \sum_{c_2 \ldots c_d \leq x} c_2^{d-2} c_3^{d-3} \ldots c_{d-2}^2 c_{d-1} (x/(c_2 c_3 \ldots c_d))^{d-1} x^{O(1/\log \log x)}\\
& = x^{d-1+O(1/\log \log x)} \sum_{c_2 c_3 \ldots c_d \leq x} c_2^{-1} c_3^{-2} \ldots c_{d-1}^{-(d-2)} c_{d}^{-(d-1)}\\
& \leq O_{d}(1) x^{d-1+O(1/\log \log x)} \log x\\
& \leq O_{d}(1) x^{d-1+O(1/\log \log x)}.
\end{align*}
Now assume that Case 2 occurs. Without loss of generality, we may assume that $j=3$. Then, if $\sigma(L)=L$, the corresponding matrix $B$ satisfies
$$b_{12}e_{2} \pm b_{22}e_{3} \in \langle B \rangle.$$
This condition holds if and only if there exist $\lambda_1,\lambda_2,\lambda_3 \in \mathbb{Z}$ such that
$$b_{12}e_{2}\pm b_{22}e_{3} = \lambda_1 b_{11}e_1+\lambda_2 (b_{12}e_1+b_{22}e_{2})+\lambda_3(b_{13}e_1+b_{23}e_2+b_{33}e_3),$$
which implies
$$0 = \lambda_1 b_{11} + \lambda_2 b_{12}+\lambda_3 b_{13},\quad b_{12}=\lambda_2 b_{22}+\lambda_3 b_{23},\quad \pm b_{22} = \lambda_3 b_{33},$$
which implies
$$\lambda_1 b_{11} = -b_{12}(b_{12} - (\pm b_{22}/b_{33}) b_{23})/b_{22} - b_{13} (\pm b_{22}/b_{33}),$$
which implies
$$b_{11} \mid -b_{12}(b_{12} - (\pm b_{22}/b_{33}) b_{23})/b_{22} - b_{13} (\pm b_{22}/b_{33}).$$
Note that
$$|-b_{12}(b_{12} - (\pm b_{22}/b_{33}) b_{23})/b_{22} - b_{13} (\pm b_{22}/b_{33})| \leq 2x^2.$$
Hence, exactly as before, once we have chosen $b_{d,d},b_{d-1,d-1},\ldots,b_{22}$, there are at most $x/(b_{22} b_{33} \ldots b_{dd})$ choices for each of $b_{12},\ldots,b_{1,d}$, and then there are at most $x^{O(1/\log \log x)}$ choices for $b_{11}$. Therefore, by the same calculation as before, we have
$$\sum_{n \leq x} N_{d,\sigma}(n) \leq O_{d}(1) x^{d-1+O(1/\log \log x)}.$$
\end{proof}
Lemma \ref{lemma:asym} may be restated in the following appealing way.
\begin{corollary}
Let $\sigma \in B_d \setminus \{\pm \Id\}$ and let $x > 1$. If $L$ is a lattice chosen uniformly at random from all sublattices of $\mathbb{Z}^d$ of index at most $x$, then
$$\Pr\{\sigma (L) = L\} = O_{d}(x^{-1+O(1/\log \log x)}).$$
\end{corollary}

We can now estimate the number of $B_d$-orbits of sublattices of $\mathbb{Z}^d$ with index at most $x$. Let us write $\mathcal{L}_{\leq x}^{d}$ for the set of sublattices of $\mathbb{Z}^d$ with index at most $x$ (so that $|\mathcal{L}_{\leq x}^{d}|=N_d(x)$), and let us write $\tilde{\mathcal{L}}_{\leq x}^{d}$ for the set of $B_d$-orbits of sublattices of $\mathbb{Z}^d$ with index at most $x$.

\begin{proposition}
\label{prop:sublattice-orbits}
The number ($|\tilde{\mathcal{L}}_{\leq x}^{d}|$) of $B_d$-orbits of sublattices of $\mathbb{Z}^d$ with index at most $x$ (equivalently, the number of conjugacy-classes of pure-translation subgroups $\Gamma$ of $\Aut(\mathbb{L}^d)$ with $|\mathbb{Z}^d/\Gamma| \leq x$) satisfies
$$|\tilde{\mathcal{L}}^d_{\leq x}| = (1+O_{d}(x^{-1+O(1/\log \log x)}))c_d x^d,$$
where 
$$c_d := \frac{\prod_{i=2}^{d} \zeta(i)}{2^{d-1}d!d}.$$
\end{proposition}

\begin{proof}
Any sublattice $L$ of $\mathbb{Z}^d$ is invariant under $\Id$ and $-\Id$, but for any $\sigma \in B_d \setminus \{\pm \Id\}$, the number of sublattices of $\mathbb{Z}^d$ of index at most $x$ which are invariant under $\sigma$ is at most $O_{d}(x^{d-1+O(1/ \log \log x)})$, by Lemma \ref{lemma:asym}. Since $|B_d| = 2^{d} d! = O_d(1)$, it follows that all but at most $O_{d}(x^{d-1+O(1/\log \log x)})$ of the sublattices of index at most $x$ are not invariant under any element of $B_d \setminus \{\pm \Id\}$. For any such sublattice $L$, 
$$\sigma (L) = \sigma'(L)\quad \Leftrightarrow \quad \sigma^{-1} \sigma' \in \{\pm \Id\},$$
so there are precisely $2^{d-1}d!$ distinct lattices $L' \in \mathcal{L}^{d}_{\leq x}$ which are in the same $B_d$-orbit as $L$ (one for each left coset of $\{\pm \Id\}$ in $B_d$). For any lattice $L \in \mathcal{L}^d_{\leq x}$, there are at most $2^{d-1}d!$ distinct lattices in the same $B_d$-orbit as $L$. Hence, the total number $|\tilde{\mathcal{L}}^d_{\leq x}|$ of $B_d$-orbits in $\mathcal{L}^{d}_{\leq x}$ satisfies
\begin{align*} \frac{1}{2^{d-1}d!}|\mathcal{L}^d_{\leq x}| \leq |\tilde{\mathcal{L}}^d_{\leq x}| \leq & \frac{1}{2^{d-1}d!}(1-O_{d}(x^{-1+O(1/\log \log x)}))|\mathcal{L}^d_{\leq x}|\\
&+O_{d}(x^{-1+O(1/\log \log x)})|\mathcal{L}^d_{\leq x}|,
\end{align*}
so
\begin{align*}
|\tilde{\mathcal{L}}^d_{\leq x}| &= \frac{1}{2^{d-1}d!}(1+O_{d}(x^{-1+O(1/\log \log x)}))|\mathcal{L}^d_{\leq x}|\\
& = (1+O_{d}(x^{-1+O(1/\log \log x)}))\frac{\prod_{i=2}^{d} \zeta(i)}{2^{d-1}d! d}x^d \\
& = (1+O_{d}(x^{-1+O(1/\log \log x)}))c_d x^d,
\end{align*}
using (\ref{eq:zeta}). By Fact \ref{fact:pt}, $|\tilde{\mathcal{L}}^d_{\leq x}|$ is precisely the number of conjugacy-classes of pure-translation subgroups $\Gamma \leq \Aut(\mathbb{L}^d)$ with $|\mathbb{Z}^d/\Gamma| \leq x$.
\end{proof}

Our next aim is to show that for fixed $r \in \mathbb{N}$, very few sublattices of $\mathbb{Z}^d$ of index at most $x$, have minimum distance at most $r$.

\begin{lemma}
\label{lemma:small-distance}
If $d,r \in \mathbb{N}$ with $d \geq 2$, then the number of sublattices of $\mathbb{Z}^d$ with index at most $x$ and minimum distance at most $r$, is 
$$O_{d,r}(x^{d-1} \log x).$$
\end{lemma}
\begin{proof}
Fix an element $u \in \mathbb{Z}^d$ with $0 < \|u\|_1 \leq r$. We shall bound the number of sublattices of $\mathbb{Z}^d$ with index at most $x$, which contain $u$. Note that $u = (u^{(1)},\ldots,u^{(d)})$ has at most $r$ non-zero coordinates. By symmetry, we may assume that $u^{(i)} = 0$ for all $i > r$. If a lattice $L$ (with index at most $x$) contains $u$, then the corresponding matrix $B$ has $u \in \langle B \rangle$, so there exist $\lambda_1,\lambda_2,\ldots,\lambda_r \in \mathbb{Z}$ such that
$$u = \sum_{j=1}^{r} \lambda_j \sum_{i =1}^{j} b_{ij}e_i.$$
Equating $i$-coordinates, we obtain
$$u^{(i)} = \sum_{j=i}^{r} \lambda_j b_{ij}\quad (1 \leq i \leq r).$$
Therefore, we have
$$\left|\sum_{j=i}^{r} \lambda_j b_{ij}\right| = |u^{(i)}| \leq r\quad (1 \leq i \leq r).$$
In particular, $|\lambda_r b_{rr}| \leq r$, so $|\lambda_r| \leq r$ (since $b_{rr} \geq 1$). We make the following.
\begin{claim}
We have $|\lambda_{r-k}| \leq 2^{k-1}r$ for all $k \geq 1$.
\end{claim}
\begin{proof}[Proof of claim.]
By induction on $k$. For $k=1$, we have
$$|\lambda_{r-1} b_{r-1,r-1} + \lambda_r b_{r,r-1}| = |u^{(r-1)}| \leq r,$$
so
$$|\lambda_{r-1}| b_{r-1,r-1} \leq r + |\lambda_r| b_{r,r-1} \leq r + r(b_{r-1,r-1}-1) = rb_{r-1,r-1},$$
so $|\lambda_{r-1}| \leq r$ (since $b_{r-1,r-1} \geq 1$), as needed. For the induction step, suppose that $k \geq 2$ and $|\lambda_{r-l}| \leq 2^{l-1} r$ for all $l < k$. We have
$$\left|\sum_{j=r-k}^{r} \lambda_j b_{ij}\right| = |u^{(r-k)}| \leq r,$$
so
\begin{align*} |\lambda_{r-k}| b_{r-k,r-k} & \leq r+\sum_{j=r-k+1}^{r} |\lambda_j| b_{j,r-k} \\
& \leq r + \left(\sum_{j=r-k+1}^{r} |\lambda_j|\right) (b_{r-k,r-k}-1)\\
& \leq r+(1+1+2+\ldots+2^{k-2})r(b_{r-k,r-k}-1)\\
& = r+2^{k-1}r(b_{r-k,r-k}-1)\\
& \leq 2^{k-1} r b_{r-k,r-k},\end{align*}
so $|\lambda_{r-k}| \leq 2^{k-1}r$, as required.
\end{proof}
It follows that there are at most
$$(2r+1)\prod_{k=1}^{r-1}(2^{k}r+1) = O_r(1)$$
choices for $(\lambda_1,\ldots,\lambda_r)$. Fix one such choice. Since $u \neq 0$, not all of $\lambda_1,\ldots,\lambda_r$ are zero. Let
$$J = \{j:\ \lambda_j \neq 0\} = \{j_1,j_2,\ldots,j_s\}.$$
Then we have
\begin{equation} \label{eq:small-modulus} \left|\sum_{l=1}^{s} \lambda_{j_l} b_{1,j_l}\right| \leq r.\end{equation}
Observe that once we have chosen $b_{d,d},b_{d-1,d-1},\ldots,b_{22}$, there are at most $x/(b_{2,2} b_{3,3} \ldots b_{d,d})$ choices for each $b_{1,j}$ with $j \neq j_1$. Once we have chosen each $b_{1,j}$ with $j \neq j_1$, $b_{1,j_1}$ is determined by the equation
$$u^{(1)} = \sum_{j=1}^{r} \lambda_j b_{1,j}.$$
Hence, summing over all the $O_r(1)$ possible choices of $(\lambda_1,\ldots,\lambda_r)$, the number of sublattices of $\mathbb{Z}^d$ with index at most $x$ and containing $u$ is at most
\begin{align*} & O_{r}(1) \sum_{c_2 \ldots c_d \leq x} c_2^{d-2} c_3^{d-3} \ldots c_{d-2}^2 c_{d-1} (x/(c_2 c_3 \ldots c_d))^{d-1}\\
 =\ & O_r(1) x^{d-1} \sum_{c_2 c_3 \ldots c_d \leq x} c_2^{-1} c_3^{-2} \ldots c_{d-1}^{-(d-2)} c_{d}^{-(d-1)}\\
 \leq\ & O_{d,r}(1) x^{d-1} \log x.
\end{align*}
Crudely, the number of choices for $u$ is at most
$$\sum_{i=1}^{r} {d \choose i} (2r+1)^i = O_{d,r}(1),$$
since $u$ has at most $r$ non-zero coordinates, and each of these has modulus at most $r$. Hence, the total number of sublattices of $\mathbb{Z}^d$ with index at most $x$ and minimum distance at most $r$ is $O_{d,r}(1) x^{d-1} \log x$, as required.
\end{proof}

Our aim is now to show that there are very few conjugacy classes of subgroups $\Gamma$ of $\Aut(\mathbb{L}^d)$ with $|\mathbb{Z}^d/\Gamma| \leq x$ and $D(\Gamma) \geq 2r^{*}(d)+2$, and which are {\em not} conjugacy classes of pure-translation subgroups. (See Theorem \ref{thm:enumeration} for the definition of $r^{*}(d)$.)

We need the following simple claim.
\begin{claim}
\label{claim:involution}
If $\Gamma \leq \Aut(\mathbb{L}^d)$ with $D(\Gamma) \geq d+1$, then $-\Id \notin P_{\Gamma}$.
\end{claim}
\begin{proof}
Let $\Gamma \leq \Aut(\mathbb{L}^d)$ and suppose that $-\Id \in P_{\Gamma}$. Then $\gamma: x \mapsto v - x \in \Gamma$ for some $v \in \mathbb{Z}^d$. Define $w_i = \lfloor v_i / 2 \rfloor$ for each $i \in [d]$. Then $d_{\mathbb{L}^d} (w,\gamma(w)) \leq d$, so $D(\Gamma) \leq d$. 
\end{proof}
By Corollaries \ref{corr:num-2} and \ref{corr:num-d}, for each $d \geq 2$, the finite, connected graphs that are $r$-locally $\mathbb{L}^d$ (where $r \geq 2$ if $d=2$ and $r \geq 3$ otherwise), correspond to subgroups $\Gamma \leq \Aut(\mathbb{L}^d)$ with $D(\Gamma) \geq 2r+2$. If $r \geq r^*(d)$, where $r^*(d)$ as defined as in Theorem \ref{thm:enumeration}, then we have $D(\Gamma) \geq 2r+2 \geq d+1$. Hence, all the subgroups $\Gamma$ relevant to us have $-\Id \notin P_{\Gamma}$. We remark that this is the only place where we rely upon $r^*(d)$ growing linearly with $d$; the rest of our proof works with $r_0(d): = 2+1_{\{d \geq 3\}}$ in place of $r^*(d)$. The subgroups $\Gamma \leq \Aut(\mathbb{L}^d)$ with $|\mathbb{Z}^d/\Gamma| < \infty$ and $P_{\Gamma} = \{\Id,-\Id\}$ are precisely those whose elements are translations and involutions of the form $x \mapsto v-x$; in this case, $\mathbb{L}^d/\Gamma$ is $r$-locally $\mathbb{L}^d$ if and only if the lattice of translations of $\Gamma$ has minimum distance at least $2r+2$, and whenever $(x \mapsto x-v) \in \Gamma$, $v$ has at least $2r+2$ odd components. The number of conjugacy-classes of such subgroups with $|\mathbb{Z}^d/\Gamma| \leq x$ is $\Theta(x^d)$, as in the pure-translation case, but the enumeration of conjugacy-classes of these subgroups (to the required degree of precision) is rather long-winded, so for brevity and clarity of exposition we prefer to rule them out by taking $r^*(d) \geq \lceil (d-1)/2 \rceil$.

For brevity, let us say that a subgroup $\Gamma \leq \Aut(\mathbb{L}^d)$ with $|\mathbb{Z}^d/\Gamma| < \infty$ is {\em highly symmetric} if $P_{\Gamma} \setminus \{\Id, -\Id\} \neq \emptyset$. Recall that $\Gamma$ is a pure-translation subgroup if and only if $P_{\Gamma} = \{\Id\}$. Hence, all the subgroups $\Gamma$ relevant to us are either pure-translation subgroups or are highly symmetric.

Our next lemma says that if two subgroups of $\Aut(\mathbb{L}^d)$ are conjugate via a translation by a vector in $\mathbb{R}^d$, then this vector can be taken to have entries in $(\mathbb{Z} \cup (\mathbb{Z}+1/2))^d$, implying that when we restrict to subgroups of $\Aut(\mathbb{L}^d)$, the equivalence classes supplied by Lemma \ref{lemma:fin} (where two crystallographic groups are `equivalent' if they are conjugate via a translation by a vector in $\mathbb{R}^d$) split into at most $2^d$ conjugacy classes in $\Aut(\mathbb{L}^d)$.

\begin{claim}
\label{claim:half}
Let $\Gamma_1,\Gamma_2 \leq \Aut(\mathbb{L}^d)$ with $\Gamma_2 = t \Gamma_1 t^{-1}$ for some translation $t$ by a vector in $\mathbb{R}^d$. Then there exists $b \in \mathbb{R}^d$ such that $2b \in \mathbb{Z}^d$ and $\Gamma_2 = t_b \Gamma_1 t_{-b}$, where $t_b: x \mapsto x+b$ denotes translation by $b$.
\end{claim}
\begin{proof}
Let $t:x \mapsto x+a$. For any vector $v \in \mathbb{R}^d$, let $t_v:x \mapsto x+v$ denote translation by $v$. Choose $b \in \mathbb{R}^d$ such that $2b \in \mathbb{Z}^d$, $|b_i-a_i| \leq 1/4$ for all $i \in [d]$, and $b_i \in \mathbb{Z}$ if $a_i \pm 1/4 \in \mathbb{Z}$. Let $\Gamma_3 = t_{b}\Gamma_1 t_{-b} = t_{b-a} \Gamma_2 t_{-(b-a)}$. We first observe that $\Gamma_3 \leq \Aut(\mathbb{L}^d)$. Indeed, let $\gamma \in \Gamma_1$; it suffices to prove that $t_b \gamma t_{-b} \in \Aut(\mathbb{L}^d)$. Write $\gamma(x) = Ax + c$, where $A \in B_d$ and $c \in \mathbb{Z}^d$. Since $t_b \gamma t_{-b}$ is an isometry of $\mathbb{R}^d$, it suffices to prove that $(t_b \gamma t_{-b})(x) \in \mathbb{Z}^d$ for all $x \in \mathbb{Z}^d$, i.e.\
$$b+Ax-Ab +c \in \mathbb{Z}^d \quad \forall x \in \mathbb{Z}^d.$$
This is true if and only if $b-Ab \in \mathbb{Z}^d$. Note that $\Gamma_2 = t_a \Gamma_1 t_{-a} \leq \Aut(\mathbb{L}^d)$, and therefore $a-Aa \in \mathbb{Z}^d$. Since $A \in B_d$, for each $i \in [d]$ there exists $j \in [d]$ and $\delta \in \{-1,1\}$ such that $(Ax)_i = \delta x_j$ for all $x \in \mathbb{Z}^d$. We have $(a-Aa)_i = a_i - \delta a_j \in \mathbb{Z}$ (since $a-Aa \in \mathbb{Z}^d$), and therefore, by our choice of $b$, we have $(b-Ab)_i = b_i - \delta b_j \in \mathbb{Z}$. Hence, $b-Ab \in \mathbb{Z}^d$, as required.

It suffices now to prove that $\Gamma_3 = \Gamma_2$. Let $x \in \mathbb{Z}^d$ and $\gamma \in \Gamma_2$. Write $\gamma(x) = Ax + c$ where $A \in B_d$ and $c \in \mathbb{Z}^d$, and define $\epsilon = b-a \in \mathbb{R}^d$. Define
$$w = (t_{b-a} \gamma t_{-(b-a)})(x) - \gamma(x) = (t_{\epsilon} \gamma t_{-\epsilon})(x) -\gamma(x).$$
Since $t_{b-a} \Gamma_2 t_{-(b-a)} = \Gamma_3 \leq \Aut(\mathbb{L}^d)$, we have $w \in \mathbb{Z}^3$. Moreover, we have
$$ w = (t_{\epsilon}  \gamma t_{-\epsilon})(x) - \gamma(x) = \epsilon+ A(x-\epsilon) + c - Ax - c = \epsilon-A\epsilon.$$
For each $i \in [d]$, we have
$$|w_i| = |(\epsilon - A\epsilon)_i| = |\epsilon_i - (A\epsilon)_i| \leq |\epsilon_i| + |(A\epsilon)_i| = |\epsilon_i|+|\epsilon_j|$$
for some $j \in [d]$. By our choice of $b$, we have $|\epsilon_k| \leq 1/4$ for all $k \in [d]$, and therefore $|w_i| \leq 1/2$ for all $i \in [d]$. Since $w \in \mathbb{Z}^d$, it follows that $w=0$. Hence, 
$t_{b-a} \gamma t_{-(b-a)} = \gamma$ for all $\gamma \in \Gamma_2$, and therefore $\Gamma_3 = \Gamma_2$, as required.
\end{proof}

Our next result says that there are very few conjugacy-classes of highly symmetric subgroups $\Gamma \leq \Aut(\mathbb{L}^d)$ with $|\mathbb{Z}^d/\Gamma| \leq x$ (compared to the number of conjugacy-classes of pure-translation subgroups $\Gamma \leq \Aut(\mathbb{L}^d)$ with $|\mathbb{Z}^d/\Gamma| \leq x$).

\begin{corollary}
\label{corr:non-translation}
Let $d \in \mathbb{N}$ with $d \geq 2$. The number of conjugacy-classes of highly symmetric subgroups $\Gamma \leq \Aut(\mathbb{L}^d)$ with $|\mathbb{Z}^d/\Gamma| \leq x$ is at most $O_{d}(1) x^{d-1+O(1/\log \log x)}$.
\end{corollary}
\begin{proof}
Let $\Gamma \leq \Aut(\mathbb{L}^d)$ with $|\mathbb{Z}^d / \Gamma| \leq x$. Then, since $\Gamma/T_{\Gamma} \cong P_{\Gamma}$, we have
\begin{equation}\label{eq:orbits} |\mathbb{Z}^d/L_{\Gamma}| = |\mathbb{Z}^d/T_{\Gamma}| \leq |P_{\Gamma}||\mathbb{Z}^d/\Gamma| \leq 2^{d}d! x = O_d(x).\end{equation}
(We note that this holds whether or not $\Gamma$ is highly symmetric.) Now assume in addition that $\Gamma$ is highly symmetric. Then $L_{\Gamma}$ is invariant under some element of $B_d \setminus \{\Id,-\Id\}$, so by Lemma \ref{lemma:asym}, $L_{\Gamma}$ there are at most $O_{d}(1) x^{d-1+O(1/\log \log x)}$ possibilities for $L_{\Gamma}$. Trivially, since $P_{\Gamma} \leq B_d$, there are at most $2^{|B_d|} = 2^{d!2^d}$ possibilities for $P_{\Gamma}$. By Lemma \ref{lemma:fin}, for any fixed sublattice $L$ of $\mathbb{Z}^d$ and any fixed $P \leq B_d$, there are at most $k(d)$ subgroups $\Gamma \leq \Aut(\mathbb{L}^d)$ with $L_{\Gamma}=L$ and $P_{\Gamma}=P$, up to conjugation by translations (by vectors in $\mathbb{R}^d$). If $\Gamma_1,\Gamma_2 \leq \Aut(\mathbb{L}^d)$ and $\Gamma_2 = t \Gamma_1 t^{-1}$ for some translation $t:x \mapsto x +c \in T(\mathbb{R}^d)$, we need not have $c \in \mathbb{Z}^d$, but by Claim \ref{claim:half}, we may assume that $2c \in \mathbb{Z}^d$. Hence, up to conjugation by translations in $T(\mathbb{Z}^d)$, there are at most $2^d k(d)$ subgroups $\Gamma \leq \Aut(\mathbb{L}^d)$ with $L_{\Gamma}=L$ and $P_{\Gamma}=P$. Hence, there are at most 
$$O_{d}(1) x^{d-1+O(1/\log \log x)} \cdot 2^{d!2^d} \cdot 2^d \cdot k(d) = O_{d}(1) x^{d-1+O(1/\log \log x)}$$
possibilities for the conjugacy-class of $\Gamma$ in $\Aut(\mathbb{L}^d)$, as required.
\end{proof}

Putting everything together, we obtain the following two lemmas.
\begin{lemma}
\label{lemma:number-connected-2}
Let $r \geq 2$, and let $\gamma_{2,r}(n)$ denote the number of connected, unlabelled, $n$-vertex graphs which are $r$-locally $\mathbb{L}^2$. Then
$$\sum_{n \leq x} \gamma_{2,r}(n) = (1+O_{r}(x^{-1+O(1/\log \log x)}))\tfrac{1}{4} \zeta(2) x^2.$$
\end{lemma}
\begin{proof} 
The left-hand side is precisely the number of connected, unlabelled graphs on at most $x$ vertices, which are $r$-locally $\mathbb{L}^2$. By Corollary \ref{corr:num-2}, this is precisely the number of conjugacy-classes of subgroups $\Gamma \leq \Aut(\mathbb{L}^2)$ which have $|\mathbb{Z}^2/\Gamma| \leq x$ and which have minimum displacement at least $2r+2$. If $\Gamma$ has minimum displacement at least 4, then by Claim \ref{claim:involution}, $-\Id \notin P_{\Gamma}$, so either $\Gamma$ is a pure-translation subgroup, or else $\Gamma$ is highly symmetric. Applying Proposition \ref{prop:sublattice-orbits}, Lemma \ref{lemma:small-distance} and Corollary \ref{corr:non-translation} in the case $d=2$, we see that the aforesaid number is
$$|\tilde{\mathcal{L}}^2_{\leq x}| - O_r(x \log x) + O(x^{1+O(1/\log \log x)}) = (1+O_{r}(x^{-1+O(1/\log \log x)}))\tfrac{1}{4} \zeta(2) x^2,$$
as required.
\end{proof}

Similarly, we obtain the following. 
\begin{lemma}
\label{lemma:number-connected-d}
Let $d,r \in \mathbb{N}$ with $d \geq 3$ and $r \geq r^*(d)$. Let $\gamma_{d,r}(n)$ denote the number of connected, unlabelled, $n$-vertex graphs which are $r$-locally $\mathbb{L}^d$. Then
$$\sum_{n \leq x} \gamma_{d,r}(n) = (1+O_{d,r}(x^{-1+O(1/\log \log x)}))c_d x^d,$$
where
$$c_d = \frac{1}{2^{d-1}d!d}\prod_{i=2}^{d} \zeta(i).$$
\end{lemma}
\begin{proof}
The left-hand side is the number of connected, unlabelled graphs on at most $x$ vertices, which are $r$-locally $\mathbb{L}^d$. By Corollary \ref{corr:num-d}, this is precisely the number of conjugacy-classes of subgroups $\Gamma \leq \Aut(\mathbb{L}^d)$ which have $|\mathbb{Z}^d/\Gamma| \leq x$ and which have minimum displacement at least $2r+2$. If $\Gamma$ has minimum displacement at least $2r^*(d) +2 \geq d+1$, then by Claim \ref{claim:involution}, $-\Id \notin P_{\Gamma}$, so either $\Gamma$ is a pure-translation subgroup, or else $\Gamma$ is highly symmetric. Applying Proposition \ref{prop:sublattice-orbits}, Lemma \ref{lemma:small-distance} and Corollary \ref{corr:non-translation}, we see that the aforesaid number is
\begin{align*} &|\tilde{\mathcal{L}}^d_{\leq x}| - O_{d,r}(x^{d-1} \log x) + O_{d}(x^{d-1+O(1/(\log \log x)})\\
=& (1+O_{d,r}(x^{-1+O(1/\log \log x)}))c_d x^d,
\end{align*}
as required.
\end{proof}

We now use a generating function argument to deduce Theorem \ref{thm:enumeration}. Recall that $a_{d,r}(n)$ denotes the number of unlabelled (possibly disconnected) graphs on $n$ vertices which are $r$-locally $\mathbb{L}^d$. We appeal to the following well-known fact (see for example \cite{flajolet-sedgewick}, page 29).

\begin{fact}
\label{fact:multisets}
Let $\mathcal{S}$ be a set, and let $w:\mathcal{S} \to \mathbb{N}$ be a function. (If $S \in \mathcal{S}$, we call $w(S)$ the {\em weight} of $S$.) Suppose that $\mathcal{S}$ contains exactly $\gamma(n)$ elements of weight $n$, where $\gamma(n) \in \mathbb{N} \cup \{0\}$ for each $n \in \mathbb{N}$. If $\mathcal{T}$ is a multiset of elements of $\mathcal{S}$, define the {\em weight} of $\mathcal{T}$ to be the sum of the weights of the elements of $\mathcal{T}$ (counted with their multiplicities). For each $n \in \mathbb{N}$, let $a(n)$ denote the number of weight-$n$ multisets of elements of $\mathcal{S}$. Define $a(0)=1$. Then the (ordinary) generating function of $a(n)$ satisfies
$$\sum_{n \geq 0} a_n z^n = \prod_{j=1}^{\infty} (1-z^{j})^{-\gamma(j)}.$$
\end{fact}

Applying this with $\mathcal{S}$ being the set of all finite, connected, unlabelled graphs that are $r$-locally $\mathbb{L}^d$, and with $w(G)$ being the number of vertices of a graph $G$, we obtain the following.
$$\sum_{n \geq 0} a_{d,r}(n) z^n = \prod_{j=1}^{\infty} (1-z^{j})^{-\gamma_{d,r}(j)}.$$

To estimate $a_{d,r}(n)$ from our knowledge of $\gamma_{d,r}(n)$, we use a variant of the following result of Brigham.

\begin{theorem}[Brigham, \cite{brigham}]
\label{thm:brigham}
Suppose $(b(n))_{n=0}^{\infty}$ is a sequence of non-negative integers with generating function satisfying
$$\sum_{n=0}^{\infty} b(n) z^n = \prod_{j=1}^{\infty} (1-z^j)^{-\gamma(j)},$$
where $(\gamma(j))_{j=1}^{\infty}$ is a sequence of non-negative integers satisfying
$$\sum_{j \leq x} \gamma(j) \sim K x^{u} (\log x)^{v}$$
for some constants $K >0,u >0$, $v \in \mathbb{R}$. Define
$$B(n) = \sum_{k=0}^{n} b(n).$$
Then
$$ \log B(n) \sim \frac{1}{u} (Ku \Gamma(u+2)\zeta(u+1))^{\frac{1}{u+1}} (u+1)^{\frac{u-v}{u+1}} n^{\frac{u}{u+1}}(\log n)^{\frac{v}{u+1}}.$$
If, in addition, every sufficiently large positive integer can be partitioned into integers in the set $\{n:\ \gamma(n) \geq 1\}$, then we have
$$\log b(n) \sim \frac{1}{u} (Ku \Gamma(u+2)\zeta(u+1))^{\frac{1}{u+1}} (u+1)^{\frac{u-v}{u+1}} n^{\frac{u}{u+1}}(\log n)^{\frac{v}{u+1}}.$$
(Here, $\Gamma$ denotes the usual $\Gamma$-function, and $\zeta$ the Riemann zeta function; all logarithms are to base $e$.)
\end{theorem}

Making a slightly different choice of the parameters in Brigham's proof (and appealing to a theorem of Odlyzko \cite{odlyzko} instead of to the theorem of Hardy and Ramanujan, Theorem A in \cite{hardy-ramanujan-general}, which Brigham uses), yields the following theorem in the case $v=0$. (This is an exercise in well-known techniques in Tauberian theory, but we give a proof in the Appendix, for the reader's convenience.)
\begin{theorem}
\label{thm:brigham-error}
Suppose $(b(n))_{n=0}^{\infty}$ is a sequence of non-negative integers with generating function satisfying
$$\sum_{n=0}^{\infty} b(n) z^n = \prod_{j=1}^{\infty} (1-z^j)^{-\gamma(j)},$$
where $(\gamma(j))_{j=1}^{\infty}$ is a sequence of non-negative integers satisfying
$$\sum_{j \leq x} \gamma(j) = (1+O(x^{-\epsilon}))K x^{u}.$$
for some constants $K >0,u >0,\epsilon \in (0,1]$. Define
$$B(n) = \sum_{k=0}^{n} a(n).$$
Then there exists $\delta >0$ such that
$$\log B(n) = (1+O(n^{-\delta}))\frac{1}{u} (Ku \Gamma(u+2)\zeta(u+1))^{\frac{1}{u+1}} (u+1)^{\frac{u}{u+1}} n^{\frac{u}{u+1}}.$$
If, in addition, every sufficiently large positive integer can be partitioned into integers in the set $\{n:\ \gamma(n) \geq 1\}$, then we have
$$\log b(n) = (1+O(n^{-\delta}))\frac{1}{u} (Ku \Gamma(u+2)\zeta(u+1))^{\frac{1}{u+1}} (u+1)^{\frac{u}{u+1}} n^{\frac{u}{u+1}}.$$
\end{theorem}

It is easy to see that $\gamma_{d,r}(n) \geq 1$ for all $n$ sufficiently large depending on $d$ and $r$; indeed, the Cayley graph of $\mathbb{Z}_n$ generated by the set
\begin{equation}
\label{eq:one-Cayley} \{1,(2r+1),(2r+1)^2,\ldots,(2r+1)^{d-1}\} \cup \{-1,-(2r+1),\ldots,-(2r+1)^{d-1}\}\end{equation}
is connected and $r$-locally $\mathbb{L}^d$ for all $n \geq (2r+1)^d$. (We thank an anonymous referee for pointing this out, since it simplifies the argument in our original manuscript.) Therefore, certainly, any sufficiently large positive integer can be partitioned into integers in the set $\{n:\ \gamma_{d,r}(n) \geq 1\}$. Hence, combining Lemmas \ref{lemma:number-connected-2} and \ref{lemma:number-connected-d} and Theorem \ref{thm:brigham-error} (with $u=d$, $K=c_d$ and $\epsilon = 1/2$) yields Theorem \ref{thm:enumeration}.

We now use some of the tools above to prove Theorem \ref{thm:sampling}. Much of the proof will be fairly standard to readers familiar with complexity theory, but we give it in full for the convenience of others.

\begin{proof}[Proof of Theorem \ref{thm:sampling}.]
Let $d,r \in \mathbb{N}$ with $d \geq 2$ and $r \geq 2+1_{\{d \geq 3\}}$, and let $n \in \mathbb{N}$. We write $\poly(n)$ to denote a quantity that is polynomial in $n$ (for fixed $d,r \in \mathbb{N}$). We sample the random graph $G_n(\mathbb{L}^d,r)$ as follows.

We first enumerate, in time $\poly(n)$, the {\em connected} unlabelled graphs on at most $n$ vertices that are $r$-locally $\mathbb{L}^d$. By Corollaries \ref{corr:num-2} and \ref{corr:num-d}, these are in (explicit) one-to-one correspondence with the conjugacy classes of subgroups $\Gamma \leq \Aut(\mathbb{L}^d)$ with minimum displacement at least $2r+2$ and with $|\mathbb{Z}^d/\Gamma| \leq n$, this correspondence being given by taking the quotient graph $\mathbb{L}^d/\Gamma$. If $\Gamma \leq \Aut(\mathbb{L}^d)$ with $|\mathbb{Z}^d/\Gamma| \leq n$, then by (\ref{eq:orbits}) we have $|\mathbb{Z}^d/T_{\Gamma}| \leq 2^d d! n$. Note that $T_{\Gamma}$ can be identified with the corresponding sublattice $\Lambda_{T_{
\Gamma}}$ of $\mathbb{Z}^d$ (which has index at most $2^d d! n$). Recall from (\ref{eq:matrix-corr}) that for any $x >0$, the sublattices of $\mathbb{Z}^d$ of index at most $x$, are in (explicit) one-to-one correspondence with upper-triangular, integer matrices $B=(b_{i,j})_{i,j \in [d]}$ with $0 \leq b_{ij} < b_{ii}$ for all $j > i$ and $\prod_{i=1}^{d} b_{ii}\leq x$. This correspondence is given by taking the sublattice of $\mathbb{Z}^d$ spanned (over $\mathbb{Z}$) by the columns of the corresponding matrix $B$:
\begin{equation}
\label{eq:corr-2}
B \leftrightarrow \text{Span}\left\{\sum_{i=1}^{j}b_{i,j}e_i:\ j \in [d]\right\}.
\end{equation}
We enumerate all the sublattices of $\mathbb{Z}^d$ with index at most $2^d d! n$ in time $\poly(n)$, by listing the corresponding matrices. Recall from Fact \ref{fact:aut-ld} that $\Aut(\mathbb{L}^d) = T(\mathbb{Z}^d) \rtimes B_d$. We can represent the elements of $B_d$ as $d$ by $d$ matrices with entires in $\{0,\pm 1\}$, and therefore we can enumerate them in time $O_d(1)$. For any subgroup $\Gamma \leq \Aut(\mathbb{L}^d)$ with $|\mathbb{Z}^d/\Gamma| < \infty$, we have $\Gamma / T_{\Gamma} \cong P_{\Gamma} \leq B_d$. Therefore, the normal subgroup $T_{\Gamma}$ has index at most $|B_d| = 2^d d!$ in $\Gamma$, which in turn, crudely, has index at most $(2^d d!)^2 n$ in $\Aut(\mathbb{L}^d)$, as does $T_{\Gamma}$. Hence, for each choice $\Delta$ of $T_{\Gamma}$, the corresponding subgroup $\Gamma$ must be such that
\begin{equation}
\Delta \lhd \Gamma,\quad \Gamma \cap T(\mathbb{Z}^d) = \Delta,\quad [\Gamma:\Delta] \leq 2^dd!.
\label{eq:weaker-condition}
\end{equation}
Let
$$\mathcal{F} = \{\Gamma \leq \Aut(\mathbb{L}^d):\ \text{(\ref{eq:weaker-condition}) holds for some subgroup }\Delta\text{ of translations satisfying }|\mathbb{Z}^d/\Delta| \leq 2^d d! n\},$$
i.e.\ $\mathcal{F}$ is the union (over all possible choices of $\Delta$) of the set of subgroups $\Gamma$ satisfying (\ref{eq:weaker-condition}). Note that some subgroups $\Gamma$ in $\mathcal{F}$ may have $|\mathbb{Z}^d/\Gamma|> n$, though all necessarily have $|\mathbb{Z}^d/\Gamma| \leq 2^d d! n$. The elements of $\mathcal{F}$ can, however, easily be enumerated in time $\poly(n)$. We may do this, crudely, by finding representatives for the (left) cosets of $\Delta$ in $\Gamma$ which are of the form $\sigma t$, where $\sigma \in B_d$ and $t = t_v$ is a translation by an element $v \in \mathbb{Z}^d$ satisfying the `box' condition
\begin{equation}
\label{eq:cuboid}
0 \leq v < b_{i,i} \leq n \quad \forall i \in [d],
\end{equation}
(Here, $B$ is the matrix corresponding to the lattice $\Lambda = \Lambda_{\Delta}$ of $\Delta$.) Indeed, for each possible index $M \leq 2^d d!$, and for tuples $(\sigma_1,\ldots,\sigma_M) \in (B_d)^{M}$ and $(t_1,\ldots,t_M) \in (T(\mathbb{Z}^d))^M$ satisfying (\ref{eq:cuboid}) and with $t_1 = \sigma_1 = \text{Id}$, observe that $\{\sigma_i t_i\}_{i \in [M]}$ forms a set of distinct (left) coset representatives of $\Delta$ as an index-$M$ normal subgroup of some $\Gamma \leq \Aut(\mathbb{L}^d)$ if and only if the following three conditions are satisfied:
\begin{enumerate}
\item[(i)] $\sigma_i(\Lambda) =\Lambda$ for all $i \in [M]$ (guarantees normality of $\Delta$ in $\Gamma$);
\item[(ii)] for all $i \in [M]$, there exists $j \in [M]$ such that $\sigma_i t_i \sigma_j t_j \in \Delta$ (guarantees closure of $\Gamma$ under taking inverses);
\item[(iii)] for all $i,j \in [M]$, there exists $k \in [M]$ such that $t_k^{-1} \sigma_k^{-1} \sigma_i t_i \sigma_j t_j \in \Delta$ (guarantees closure of $\Gamma$ under multiplication).
\end{enumerate}
For brevity, for any $m,d \in \mathbb{N}$ we define
$$\text{Box}_d(m) = \{x \mapsto x+v:\ 0 \leq v_i < m\ \forall i \in [d]\}.$$
The condition (\ref{eq:cuboid}) implies that the translation $t$ satisfies $t \in \textrm{Box}_d(n)$. Now, condition (i) can be checked by checking that for each $i \in [M]$, we have $\sigma_i(\mathbf{b}) \in \Lambda$ for each column $\mathbf{b}$ of $B$. This in turn can be done by solving (for $\mathbf{x} \in \mathbb{Z}^d$) the equation $B \mathbf{x} = \sigma_i(\mathbf{b})$, or verifying that no solution exists; solving this has complexity $O_d(\log n \log \log n)$, since integer division (for integers at most $n$) has complexity $O(\log n \log \log n)$ (see \cite{barrett,harvey-hoeven}). Hence, the total time required for checking condition (i) is $O_d(M \log n \log \log n) = \poly(n)$, summing over all $d$ columns of $B$ and over all $i \in [M]$. Condition (ii) can be checked in time $\poly(n)$ by checking, for each $i \in [M]$, whether there exists $j \in [M]$ with $\sigma_i \sigma_j = \text{Id}$ (which guarantees that $\sigma_i t_i \sigma_j t_j \in T(\mathbb{Z}_d)$), and then (if this holds), checking whether there exists such a $j$ for which the translation $\sigma_i t_i \sigma_j t_j$ is a translation by an element of $\Lambda$ (noting that $\sigma_i t_i \sigma_j t_j \in \text{Box}_d(2n)$). Condition (iii) can be checked in time $\poly(n)$, by checking, for each $i,j \in [M]$, whether there exists $k \in [M]$ with $\sigma_i \sigma_j = \sigma_k$ (which guarantees that $t_k^{-1}\sigma_k^{-1}\sigma_i t_i \sigma_j t_j \in T(\mathbb{Z}_d)$), and then (if this holds), checking whether there exists such a $k$ for which the translation $t_k^{-1} \sigma_k^{-1} \sigma_i t_i \sigma_j t_j$ is a translation by an element of $\Lambda$ (noting that $t_k^{-1} \sigma_k^{-1}\sigma_i t_i \sigma_j t_j \in \text{Box}_d(3n)$).

The above immediately implies that $|\mathcal{F}| \leq \poly(n)$. We now claim that $\mathcal{F}$ can be partitioned into $\Aut(\mathbb{L}^d)$-conjugacy classes in time $\poly(n)$. Indeed, for $\Gamma_1,\Gamma_2 \in \mathcal{F}$, let $\{g_{i,1} T_{\Gamma_i},g_{i,2} T_{\Gamma_i},\ldots,g_{i,M_i}T_{\Gamma_i}\}$ be the set of (distinct, left) cosets of $T_{\Gamma_i}$ in $\Gamma_i$ (for $i=1,2$), where each $g_{i,j}$ can be written in the form $\sigma t$ with $\sigma \in B_d$ and $t \in \text{Box}_d(n)$. By the definition of $\mathcal{F}$, we have $M_1,M_2 \leq 2^d d!$. Now observe that $\Gamma_1$ and $\Gamma_2$ are conjugate in $\Aut(\mathbb{L}^d)$ if and only if the following three conditions hold:
\begin{enumerate}
\item[(1)] $M_1=M_2$;
\item[(2)] there exists $h \in \Aut(\mathbb{L}^d)$ such that $h^{-1}T_{\Gamma_1}h = T_{\Gamma_2}$;
\item[(3)] for each $j \in [M_1]$ there exists $\ell = \ell(j) \in [M_1]$ such that $hg_{1,j}h^{-1}T_{\Gamma_2} = g_{2,\ell} T_{\Gamma_2}$, or equivalently $g_{2,\ell}^{-1} hg_{1,j}h^{-1} \in T_{\Gamma_2}$.
\end{enumerate}
Since $T_{\Gamma_2}$ is a normal subgroup of $\Gamma_2$, the conditions $h^{-1}T_{\Gamma_1} h^{-1} = T_{\Gamma_2}$ and $g_{2,\ell}^{-1} hg_{1,j}h^{-1} \in T_{\Gamma_2}$ (in (2) and (3) above) are invariant under multiplying $h$ on the left by an element of $T_{\Gamma_2}$. (For $t \in T_{\Gamma_2}$, we have $g_{2,\ell}^{-1}tg_{2,\ell} \in T_{\Gamma_2}$, so $g_{2,\ell}^{-1} hg_{1,j}h^{-1} \in T_{\Gamma_2}$ iff $(g_{2,\ell}^{-1}tg_{2,\ell})(g_{2,\ell}^{-1} hg_{1,j}h^{-1})(t^{-1}) \in T_{\Gamma_2}$ iff $g_{2,\ell}^{-1} (th) g_{1,j} (th)^{-1} \in T_{\Gamma_2}$.) The subgroup $T_{\Gamma_2}$ has index at most $(2^d d!)^2 n$ in $\Aut(\mathbb{L}^d)$, so to check whether $\Gamma_1$ and $\Gamma_2$ are conjugate in $\Aut(\mathbb{L}^d)$, we need only check $(2^d d!)^2 n$ possibilities for $h$ and $M_1! = (2^d d!)!$ possibilities for the sequence $(\ell(j):\ j \in [M_1])$. For $h=\sigma t$ where $\sigma \in B_d$ and $t \in \text{Box}_d(n)$, we have $h^{-1}T_{\Gamma_1}h = T_{\Gamma_2}$ iff $\sigma(\Lambda_{\Gamma_2}) = \Lambda_{\Gamma_1}$, a condition which can be checked in time $\poly(n)$ in the same way as condition (i) above. Similarly, writing $g_{1,j} = \sigma_1 t_1$ and $g_{2,\ell} = \sigma_2 t_2$, we have $g_{2,\ell}^{-1} hg_{1,j}h^{-1} \in T_{\Gamma_2}$ iff $\sigma^{-1} \sigma_2 \sigma = \sigma_1$ and $t_2^{-1}\sigma_2^{-1}\sigma t \sigma_1 t_1 t^{-1}\sigma^{-1} \in T_{\Gamma_2}$, conditions that can be checked in time $\poly(n)$ similarly to condition (iii) above.

For each of the subgroups $\Gamma \in \mathcal{F}$ found in the previous steps (up to conjugacy), we (simultaneously) check whether $\Gamma$ has minimum displacement at least $2r+2$, find the quotient graph $\mathbb{L}^d/\Gamma$, and verify that $\mathbb{L}^d/\Gamma$ has at most $n$ vertices, all in time $\poly(n)$, as follows. We fix a set $\{g_{1},g_{2},\ldots,g_{M}\}$ of (distinct, left) coset representatives of $T_{\Gamma}$ in $\Gamma$, where each $g_i$ can be written in the form $\sigma t$ with $\sigma \in B_d$ and $t \in \text{Box}_d(n)$; note that $M \leq 2^d d!$. We first construct the quotient graph $\mathbb{L}^d/T_{\Gamma}$ (which is the quotient lattice of $\mathbb{L}^d$ inside a $d$-dimensional torus), using the correspondence (\ref{eq:corr-2}). We may view $\mathbb{L}^d/T_{\Gamma}$ as a (generalised) discrete torus graph. The vertex-set of $\mathbb{L}^d/T_{\Gamma}$ is $\mathbb{Z}^d/\Lambda_{T_\Gamma}$, which is naturally identified with the discrete cuboid
$$C:=\prod_{i=1}^{d} \{0,1,2,\ldots,b_{ii}-1\}$$
via the quotient map $q:\mathbb{Z}^d \to \mathbb{Z}^d/\Lambda_{T_{\Gamma}}$ that reduces modulo $\Lambda_{T_{\Gamma}}$, i.e., modulo the $\mathbb{Z}$-span of the columns of the matrix $B$. The edge-set of $\mathbb{L}^d/T_{\Gamma}$ is simply $\{\{w,q(w+e_i)\}:\ w \in C,\ i \in [d]\}$. For any $w \in C$ and any $i \in [d]$, $q(w+e_i)$ can be evaluated in time $O_d(\log n \log \log n)$ by reducing modulo the $\mathbb{Z}$-span of the columns of the matrix $B$. 

Observe that the action of $\Gamma$ on $\mathbb{L}^d$ induces an action of $\Gamma/T_{\Gamma}$ on the discrete torus $\mathbb{L}^d/T_{\Gamma}$; these two actions have the same quotient graph, namely $\mathbb{L}^d/\Gamma$. Since $\mathbb{L}^d/T_{\Gamma}$ has at most $2^d d! n$ vertices and $|\Gamma/T_{\Gamma}| = M \leq 2^d d!$, the action of $\Gamma/T_{\Gamma}$ on the discrete torus $\mathbb{L}^d/T_{\Gamma}$ can be determined in time $O_d(n)$ (using the left coset representatives $g_i$), and therefore the quotient graph $\mathbb{L}^d/\Gamma$ can be determined in time $\poly(n)$. (For each $i \in [M]$ and each $w \in C$, we have $\max_{j \in [d]}|(g_i(w))_j| \leq 2n$ and therefore $q(g_i(w))$ can be evaluated in time $O_d(\log n \log \log n)$.) Checking that $\Gamma$ has minimum displacement at least $2r+2$ is equivalent to simply checking directly that the quotient graph $\mathbb{L}^d/\Gamma$ is $r$-locally $\mathbb{L}^d$, which can be done in $O_{d,r}(n)$ steps since $\mathbb{L}^d/\Gamma$ has at most $n$ vertices, and each ball of radius $r$ in $\mathbb{L}^d/\Gamma$ contains $O_{d,r}(1)$ vertices.

The rest of the proof follows the method of Nijenhuis and Wilf in \cite{nw}. Let $H_1,H_2,\ldots,H_N$ be an enumeration of the connected, unlabelled graphs on at most $n$ vertices that are $r$-locally $\mathbb{L}^d$, with $|H_1| \leq |H_2| \leq \ldots \leq |H_N|$; note that $|H_N| \leq n$ and that $N \leq \poly(n)$. For notational convenience, we define an ordering $\preceq$ on the set of such graphs, defined by $H_i \preceq H_j$ iff $i \leq j$ (and $H_i \prec H_j$ iff $i < j$). For each $i \in [N]$ and each $k \leq n$, we define $a_{k,H_i}$ to be the number of unlabelled graphs on $k$ vertices that are $r$-locally $\mathbb{L}^d$, have at least one component equal to $H_i$, and no components equal to any $H_j$ with $j > i$. We note the recurrence
$$a_{k,H_i} = \sum_{j \leq i} a_{k-|H_i|,H_j}\quad \forall k \leq n,$$
which simply arises from removing a component isomorphic to $H_i$ and considering all possible choices for the `next largest' component (`next largest' meaning, of course, with respect to the ordering $\preceq$). It is clear that this recurrence can be used to calculate (using $\poly(n)$ addition operations), all the numbers $a_{k,H_i}$ (for $k \leq n$, $i \leq N$). 

To sample a graph $G$, we perform the following process. We first choose the largest component $H$ of $G$ (largest with respect to the ordering $\preceq$), according to the probability distribution
$$\Prob[H=H_i] = \frac{a_{n,H_i}}{\sum_{j=1}^{N} a_{n,H_j}}.$$
We note that this distribution may be simulated by rolling (once) a fair, $\poly(n)$-sided die, which can in turn be simulated by flipping $\poly(n)$ biased coins with rational biases (see e.g.\ \cite{shamir-dice}). We then simply repeat this process, choosing the largest component $H'$ (largest w.r.t.\ $\preceq$) of $G-V(H)$ according to the probability distribution
$$\Prob[H'=H_i] = \frac{a_{n-|H|,H_i}}{\sum_{j=1}^{N} a_{n-|H|,H_j}},$$
etc. It is easy to see that this process generates the uniform random graph $G_n(\mathbb{L}^d,r)$ in time $\poly(n)$.
\end{proof}

\section{Cayley graphs of torsion-free groups of polynomial growth}
\label{sec:general}

In this section, we consider the (much) more general case where $\mathbb{L}^d$ is replaced by a Cayley graph of a torsion-free group of polynomial growth. Our principal objectives here (which we accomplish in order) are to prove Theorems \ref{thm:largest-component}, \ref{thm:aut-gen} and \ref{thm:stretched-exp}.

Before proving Theorem \ref{thm:largest-component}, we first outline some of the notation, definitions and prior results we will use.

Let $\Gamma$ be a finitely generated group. For each $n \in \mathbb{N}$, we let $a_n(\Gamma) \in \mathbb{N} \cup \{0\}$ denote the number of subgroups of $\Gamma$ with index $n$. We say that $\Gamma$ {\em has polynomial subgroup growth} if there exists $\alpha \geq 0$ such that $a_n(\Gamma) \leq n^\alpha$ for all $n \in \mathbb{N}$.

We need the following well-known lemma (see e.g.\ \cite{ls}).

\begin{lemma}
\label{lemma:lms}
Let $\Gamma$ be a finitely generated, virtually nilpotent group. Then $\Gamma$ has polynomial subgroup growth.
\end{lemma}

We recall the celebrated theorem of Gromov \cite{gromov} characterizing the groups of polynomial growth.

\begin{theorem}[Gromov]
\label{thm:gromov}
Let $\Gamma$ be a finitely generated group. Then $\Gamma$ has polynomial growth if and only if it is virtually nilpotent.
\end{theorem}

We need the following result of De La Salle and Tessera \cite{tessera} (relying on a theorem of Trofimov).

\begin{proposition}
\label{prop:vertex-stabilizers}
Let $\Gamma$ be a finitely generated, torsion-free group of polynomial growth, and let $F = \Cay(\Gamma,S)$ be a connected, locally finite Cayley graph of $\Gamma$. Then the vertex-stabilizers of $\Aut(F)$ are finite.
\end{proposition}

The following is an immediate consequence of Proposition \ref{prop:vertex-stabilizers}.
\begin{corollary}
\label{cor:aut-virtually}
Let $\Gamma$ be a finitely generated, torsion-free, virtually nilpotent group, and let $F = \Cay(\Gamma,S)$ be a connected, locally finite Cayley graph of $\Gamma$. Then $[\Aut(F):\Gamma] < \infty$, and therefore $\Aut(F)$ is virtually nilpotent.
\end{corollary}

Lemma \ref{lemma:lms}, Theorem \ref{thm:gromov} and Corollary \ref{cor:aut-virtually} together imply the following.

\begin{corollary}
\label{cor:aut-poly}
Let $\Gamma$ be a finitely generated group of polynomial growth, and let $F$ be a connected, locally finite Cayley graph of $\Gamma$. Then $\Aut(F)$ has polynomial subgroup growth.
\end{corollary}

The following is an immediate consequence of Theorem \ref{thm:dlst-intro} and Corollary \ref{cor:aut-poly}.

\begin{proposition}
\label{prop:conn-upper}
Let $F$ be a connected, locally finite Cayley graph of a finitely generated, torsion-free group of polynomial growth. Then there exists $r_0 = r_0(F) \in \mathbb{N}$ and $\alpha = \alpha(F) \geq 0$ such that the number of unlabelled, connected graphs on at most $x$ vertices that are $r_0$-locally $F$, is at most $x^{\alpha}$.
\end{proposition} 

For each $n,r \in \mathbb{N}$, let $\gamma_{F,r}(n)$ denote the number of connected, unlabelled graphs on $n$ vertices that are $r$-locally $F$; Proposition \ref{prop:conn-upper} says that $\gamma_{F,r}(n) \leq n^{\alpha}$ if $r \geq r_0$. Similarly, for each $n \in \mathbb{N}$, let $b_{F,r}(n)$ denote the number of (not necessarily connected) unlabelled graphs on $n$ vertices that are $r$-locally $F$, and define $b_{F,r}(0)=1$. (We regard the empty graph to be $r$-locally $F$.) When $F$ and $r$ are understood, for brevity, we let $\B(n)$ denote the set of all unlabelled, $n$-vertex graphs that are $r$-locally $F$, so that $b_{F,r}(n) = |\B(n)|$. 

We now introduce the following.
\begin{definition}
Let $F$ be a connected, locally finite Cayley graph of a finitely generated group $\Gamma$, and let $r \in \mathbb{N}$. We say that a subgroup $\Lambda \leq \Gamma$ is {\em good} (for $F$ and $r$) if $[\Gamma:\Lambda] < \infty$ and $\Lambda \cap B_F(\Id,2r) = \emptyset$.
\end{definition}
\begin{lemma}
\label{lemma:good}
Let $F$ be a connected, locally finite Cayley graph of a finitely generated group $\Gamma$, and let $r \in \mathbb{N}$. If $\Lambda$ is a good subgroup of $\Gamma$, then $F/\Lambda$ (where $\Lambda$ acts on $V(F) = \Gamma$ by right multiplication) is connected, $r$-locally $F$ and vertex-transitive. In particular, if $\Gamma$ has a good subgroup of index $n$, then $\gamma_{F,r}(n) \geq 1$.
\end{lemma}

\begin{proof}
All properties are clear except (possibly) for the vertex-transitivity. To see the latter, observe that the vertices of $F/\Lambda$ correspond to the left cosets of $\Lambda$ in $\Gamma$. Now let $F = \Cay(\Gamma,S)$. For any $\gamma,x,y \in \Gamma$, $x \Lambda$ is joined to $y \Lambda$ in $F/\Lambda$ if and only if there exist $\gamma_1,\gamma_2 \in \Lambda$ such that $(x \gamma_1)^{-1} (y \gamma_2) \in S$, which holds if and only if there exist $\gamma_1,\gamma_2 \in \Lambda$ such that $(\gamma x \gamma_1)^{-1} (\gamma y \gamma_2) \in S$, i.e.\ if and only if $\gamma x \Lambda$ is joined to $\gamma y \Lambda$ in $F/\Lambda$. Hence, for any $\gamma \in \Gamma$, the map $x \Lambda \mapsto \gamma x \Lambda$ is an automorphism of $F/\Lambda$, and the left action of $\Gamma$ is clearly transitive on the left cosets of $\Lambda$ in $\Gamma$.
\end{proof}

\noindent Recall the following.
\begin{definition}
Let $\mathcal{P}$ be a property of groups, and let $\Gamma$ be a group. We say that $\Gamma$ is {\em residually $\mathcal{P}$} if for any $\gamma \in \Gamma \setminus \{\Id\}$, there exists a normal subgroup $N \lhd \Gamma$ such that $\gamma \notin N$ and $\Gamma/N$ has the property $\mathcal{P}$.
\end{definition}

It is easy to see that if $\Gamma$ is a finitely generated, residually finite group and $F$ is a connected, locally finite Cayley graph of $\Gamma$, then for every $r \in \mathbb{N}$, there exists a finite graph that is $r$-locally $F$ (see e.g.\ \cite{georgakopoulos}). Interestingly, in the other direction, De La Salle and Tessera proved \cite[Corollary K]{tessera} that if $\Gamma$ is a finitely presented group with an element of infinite order, and for every connected, locally finite Cayley graph $F$ of $\Gamma$ and every $r \in \mathbb{N}$, there exists a graph that is $r$-locally $F$, then $\Gamma$ must be residually finite. It follows, for example, that there exists a connected, locally finite Cayley graph $F$ of the Baumslag-Solitar group $\textrm{BS}(2,3)$, and a positive integer $r$, such that no finite graph is $r$-locally $F$. We need the following slightly stronger variant of the first statement (under stronger hypotheses).

\begin{proposition}
\label{prop:linear}
Let $\Gamma$ be a finitely generated group of polynomial growth, and let $F$ be a connected, locally finite Cayley graph of $\Gamma$. Then there exists $h_1 \in \mathbb{N}$ such that for any multiple $n$ of $h_1$, there is a vertex-transitive, connected, $n$-vertex graph that is $r$-locally $F$ (so in particular, $\gamma_{F,r}(n) \geq 1$ whenever $h_1 \mid n$).
\end{proposition}
\begin{proof}
Let $B_F(\Id,2r) = \{\gamma_1,\gamma_2,\ldots,\gamma_N\}$. By Gromov's theorem (Theorem \ref{thm:gromov}), $\Gamma$ is virtually nilpotent. Let $\Gamma' \leq \Gamma$ such that $\Gamma'$ is nilpotent and $[\Gamma:\Gamma'] < \infty$. Since $\Gamma'$ is a finitely generated nilpotent group, it has finite torsion subgroup, which for convenience we denote by $\{\gamma_{N+1},\ldots,\gamma_{N+M}\}$. (Note that we may have $\gamma_i=\gamma_j$ for some $i \leq N$ and $j >N$.) By a theorem of Hirsch \cite{hirsch}, $\Gamma'$ is residually finite, and so for each $i \in [N+M]$ there exists a subgroup $\Lambda_i \leq \Gamma'$ such that $g_i \notin \Lambda_i$ and $[\Gamma':\Lambda_i] < \infty$. Define $\Lambda = \cap_{i=1}^{N+M} \Lambda_i$; then $\Lambda \cap B_F(\Id,2r) = \emptyset$ and $[\Gamma':\Lambda] < \infty$ (so $[\Gamma:\Lambda] < \infty$), and therefore $\Lambda$ is a good subgroup of $\Gamma$; moreover, $\Gamma$ is torsion-free and nilpotent.

Let $h_1 := [\Gamma:\Lambda]$. Since $\Lambda$ is a finitely generated, torsion-free nilpotent group, it has a subgroup of every finite index; every such subgroup is a good subgroup of $\Gamma$. Hence, $\Gamma$ has a good subgroup of every index dividing $h_1$. The proposition now follows from Lemma \ref{lemma:good}.

\end{proof}

We also need the following.
\begin{lemma}
\label{lemma:suff-large}
Let $\Gamma$ be a finitely generated, torsion-free group of polynomial growth, let $F$ be a connected, locally finite Cayley graph of $\Gamma$ and let $r \in \mathbb{N}$. Then there exists $n_0 = n_0(F,r) \in \mathbb{N}$ such that for any $n \in \mathbb{N}$ with $n \geq n_0$, there exists an $n$-vertex graph that is $r$-locally $F$ (i.e., $b_{F,r}(n) \geq 1$ for all $n \geq n_0$).
\end{lemma}
\begin{proof}
It suffices to show that the highest common factor of the set of integers $\{n \in \mathbb{N}:\ \gamma_{F,r}(n) \geq 1\}$ is equal to 1. By Lemma \ref{lemma:good}, it suffices to find a finite set of good subgroups of $\Gamma$ whose indices have highest common factor 1.

As in the proof of Proposition \ref{prop:linear}, let $B_F(\Id,2r) = \{\gamma_1,\gamma_2,\ldots,\gamma_N\}$, and let $\Gamma' \leq \Gamma$ such that $\Gamma'$ is nilpotent and $[\Gamma:\Gamma'] < \infty$. Since $\Gamma'$ is nilpotent, it is residually finite, and so for each $i \in [N]$ there exists a subgroup $\Lambda_i \leq \Gamma'$ such that $g_i \notin \Lambda_i$ and $[\Gamma':\Lambda_i] < \infty$. Define $\Lambda = \cap_{i=1}^{N} \Lambda_i$; then $\Lambda \cap B_F(\Id,2r) = \emptyset$ and $[\Gamma':\Lambda] < \infty$ (so $[\Gamma:\Lambda] < \infty$), and therefore $\Lambda$ is a good subgroup of $\Gamma$.

We now seek a finite-index subgroup $\Lambda' \leq \Lambda$ such that $\Lambda' \lhd \Gamma$ and for all $i \in [N]$, $\gamma_i \Lambda'$ does not have prime-power order in $\Gamma/\Lambda'$. If $\Lambda$ itself satisfies this condition, then we may take $\Lambda=\Lambda'$; otherwise, by relabelling the $\gamma_i$ if necessary, we may assume that $\gamma_1\Lambda,\ldots,\gamma_k\Lambda$ all have prime-power order in $\Gamma/\Lambda$, where $k \in [N]$, and $\gamma_i \Lambda$ does not have prime-power order in $\Gamma/\Lambda$ for any $i > k$. Let $\gamma_i \Lambda$ have order $p_i^{a_i}$ in $\Gamma/\Lambda$ for each $i \in [k]$, where $p_1,\ldots,p_k$ are primes and $a_1,\ldots,a_k \in \mathbb{N}$. Then by definition, we have $\gamma_{i}^{p_{i}^{a_i}} \in \Lambda$ for each $i \in [k]$, and since $\Gamma$ is torsion-free, we have $\gamma_{i}^{p_{i}^{a_i}} \neq \Id$ for each $i \in [k]$. Let $q$ be a prime distinct from $p_1,\ldots,p_k$. Since $\Lambda$ is torsion-free and nilpotent, by a theorem of Gruenberg \cite{gruenberg} it is residually a finite $p$-group for any prime $p$, and therefore for each $i \in [k]$, it has a normal subgroup $\Lambda_i'$ such that $[\Lambda:\Lambda_i']$ is a power of $q$, and $\gamma_i^{p_i^{a_i}} \notin \Lambda_i'$. Define $\Lambda_0 = \cap_{i=1}^{k} \Lambda_i'$; then $[\Lambda:\Lambda_0]$ is also a power of $q$, and $\gamma_i^{p_i^{a_i}} \notin \Lambda_0$ for all $i \in [k]$. Finally, let
$$\Lambda' = \bigcap_{g \in \Gamma} g \Lambda_0 g^{-1};$$
then $\Lambda' \lhd \Lambda$, $\Lambda' \lhd \Gamma$ and $[\Gamma:\Lambda'] < \infty$.

We now claim that for each $i \in [N]$, $\gamma_i \Lambda'$ does not have prime-power order in $\Gamma/\Lambda'$. Indeed, since $\Lambda' \leq \Lambda$, the order of $\gamma_i \Lambda$ in $\Gamma/\Lambda$ must divide the order of $\gamma_i \Lambda'$ in $\Gamma/\Lambda'$, so this is trivially true for all $i > k$. Suppose then that $i \leq k$; we may assume without loss of generality that $k \geq 1$ and $i=1$. Since $\Lambda/\Lambda_0$ has order a power of $q$, the order of $\gamma_1^{p_1^{a_1}} \Lambda_0$ in $\Lambda/\Lambda_0$ is some power of $q$, say $q^{b}$ where $b \in \mathbb{N}$. In other words, $(\gamma_1^{p_1^{a_1}})^{q^b} \in \Lambda_0$, and if $j \in \mathbb{N}$ with $(\gamma_1^{p_1^{a_1}})^{j} \in \Lambda_0$, then $q^b \mid j$. Let $t$ be the order of $\gamma_1 \Lambda'$ in $\Gamma/\Lambda'$. Then $\gamma_1^t \in \Lambda' \leq \Lambda_0$, so $(\gamma_1^{p_1^{a_1}})^t \in \Lambda_0$, and so $q^b \mid t$. Moreover, since $\gamma_1^t \in \Lambda' \leq \Lambda$, and $\gamma_1 \Lambda$ has order $p_1^{a_1}$ in $\Gamma/\Lambda$, we have $p_1^{a_1} \mid t$. Hence, $q^b p_1^{a_1} \mid t$, so $\gamma_1 \Lambda'$ does not have prime-power order in $\Gamma/\Lambda'$. The same holds for all $i \in [k]$, proving the claim.

Write the order of $\Gamma/\Lambda'$ as a product of primes,
$$|\Gamma/\Lambda'| = \prod_{j=1}^{M} q_j^{b_j}.$$
For each $j \in [M]$, let $H_j$ be a Sylow-$q_j$ subgroup of $\Gamma/\Lambda'$, and let $G_j$ be the corresponding subgroup of $\Gamma$ containing $\Lambda'$, i.e.\ such that $G_j/\Lambda' = H_j$. Then for each $i \in [N]$ and each $j \in [M]$, we have $\gamma_i \notin G_j$, since for every $g \in G_j$, $g\Lambda' \in H_j$ and therefore $g\Lambda'$ has order a power of $q_j$ in $\Gamma /\Lambda'$, whereas $\gamma_i \Lambda'$ does not have prime-power order in $\Gamma/\Lambda'$, by construction. Hence, each of the subgroups $G_j$ are good subgroups of $\Gamma$. For each $j$, we have
$$[\Gamma:G_j] = [\Gamma/\Lambda': G_j / \Lambda'] = [\Gamma/\Lambda': H_j] = \prod_{i \neq j} q_i^{b_i},$$
and therefore the indices $\{[\Gamma:G_j]:\ j \in [M]\}$ have highest common factor 1. This completes the proof of the lemma.
\end{proof}

We next make a useful `monotonicity' observation.
\begin{claim}
\label{claim:mono}
Let $\Gamma$ be a finitely generated, torsion-free group of polynomial growth, let $F$ be a connected, locally finite Cayley graph of $\Gamma$ and let $r \in \mathbb{N}$. Let $n_0$ be as in Lemma \ref{lemma:suff-large}. Then for all $m,n \in \mathbb{N} \cup \{0\}$ with $n \geq m+n_0$, we have $b_{F,r}(n) \geq b_{F,r}(m)$.
\end{claim}
\begin{proof}
Let $m,n$ be as in the statement of the claim. Let $H_0$ be a fixed $(n-m)$-vertex graph that is $r$-locally $F$. Given an $m$-vertex graph $G$ that is $r$-locally $F$, we may produce from $G$ an $n$-vertex graph $\Phi(G)$ that is $r$-locally $F$ by adding to $G$ a vertex-disjoint copy of $H_0$. The map $\Phi: \B(m) \to \B(n)$ is clearly injective, proving the claim.
\end{proof}

We are now in a position to prove Theorem \ref{thm:largest-component}.

\begin{proof}[Proof of Theorem \ref{thm:largest-component}.]
Let $r_0$ be as in Theorem \ref{thm:dlst-intro}, and let $r \geq r_0$. For brevity, let us write $\gamma(n): = \gamma_{F,r}(n)$ and $b(n) : = b_{F,r}(n)$, for each $n \in \mathbb{N}$. For each $k,n \in \mathbb{N}$ with $k \leq n$, let $a(n,k)$ denote the number of unlabelled, $n$-vertex, $r$-locally $F$ graphs with largest component of order $k$.

Let $\delta = \delta(n)>0,\ \epsilon=\epsilon(n) >0$ to be chosen later. Suppose for a contradiction that for infinitely many $n \in \mathbb{N}$, the probability that a uniform random unlabelled $n$-vertex $r$-locally $F$ graph has largest component of order at least $\epsilon n$, is greater than $\delta$. Let us call such integers $n$ the {\em bad} integers. Then for all bad $n \in \mathbb{N}$, we have
$$\sum_{k = \lceil \epsilon n \rceil }^{n} a(n,k) > \delta b(n).$$
Observe that for any $k \leq n$, we have $a(n,k) \leq \gamma(k)b(n-k)$, and therefore for all bad $n \in \mathbb{N}$, we have
$$n \max_{\lceil \epsilon n \rceil \leq k \leq n} \gamma(k) b(n-k) \geq \sum_{k = \lceil \epsilon n \rceil }^{n} a(n,k) > \delta b(n).$$
By Proposition \ref{prop:conn-upper}, we have $\gamma(k) \leq n^{\alpha}$ for all $k \leq n$, and therefore
\begin{equation} \label{eq:small-quotient} \max_{\lceil \epsilon n \rceil \leq k \leq n} b(n-k) > \delta n^{-\alpha-1} b(n)\end{equation}
for all bad $n \in \mathbb{N}$. For appropriate choices of $\delta$ and $\epsilon$, this will contradict the following claim.
\begin{claim}
\label{claim:fast-growth}
For all $n,k \in \mathbb{N}$ such that $n \geq n_0$ and $k \geq n^{5/6}$, we have
$$\frac{b(n-k)}{b(n)} \leq C\exp(-cn^{1/12}),$$
where $n_0$ is as in Lemma \ref{lemma:suff-large}, and $c,C>0$ are constants depending upon $F$ and $r$ alone.
\end{claim}
\begin{proof}[Proof of claim.]
By appropriate choices of $c$ and $C$, we may assume throughout that $n$ is at least any specified constant depending upon $F$ and $r$. Let $n^{5/6} \leq k \leq n$. By an appropriate choice of $c$, and by Claim \ref{claim:mono}, we may assume that $k \leq n/2$. We now split into two cases.

{\em Case (i).} First, suppose that at least half the graphs in $\B(n-k)$ have largest component of order less than $(n-k)^{2/3}$. Let $\G$ denote the set of such graphs; then $|\G| \geq \tfrac{1}{2}|\B(n-k)|$. Given a graph $G \in \G$, we produce a graph $H \in \B(n)$ as follows. Let $q_0 \in \mathbb{N}$ be minimal such that $q_0 \geq n_0$ and $k-q_0$ is a multiple of $h_1$ (where $h_1$ is as in Proposition \ref{prop:linear}), and define $W = \lceil (n-k)^{2/3}/h_1\rceil$. 

Let $M  = \lceil (k-q_0)/((W+1)h_1)\rceil$ and let $m = (k-q_0)/h_1 - MW$; then $m \in \mathbb{N}$ and $m \leq M$. Choose any integer partition $\lambda = (\lambda_1,\lambda_2,\ldots) \vdash m$, and for each $i \in [M]$, add to $G$ a connected, $((W + \lambda_i) h_1)$-vertex component that is $r$-locally $F$. Since $W h_1 \geq (n-k)^{2/3}$, each of these added components has more vertices than any component of a graph in $\G$. Now add an additional vertex-disjoint copy of a fixed, $q_0$-vertex graph $H_0$ that is $r$-locally $F$, producing an $n$-vertex graph $H$ that is $r$-locally $F$. (Note that the total number of vertices added is $M W h_1 + mh_1 + q_0 = k$.) Given a fixed graph $G \in \G$, each integer partition $\lambda \vdash m$ produces from $G$ in this way a different unlabelled graph $H \in \B(n)$, and for $G_1 \neq G_2 \in \G$, the sets of $n$-vertex unlabelled graphs $H \in \B(n)$ produced in this way by $G_1$ and $G_2$ are disjoint, since the graph in $\G$ that produced a given $H \in \B(n)$ may be recovered from $H$ by deleting a copy of $H_0$, and then deleting the $M$ largest remaining components. It follows that
$$b(n) \geq \tfrac{1}{2} p(m) b(n-k),$$
where $p(m)$ denotes the number of integer partitions of $m$. It is easy to see that, if $n^{5/6} \leq k \leq n/2$, then $m = (1+o(1))k/(n-k)^{2/3} \geq (1+o(1))n^{1/6}$ as $n \to \infty$, and therefore, using the Hardy-Ramanujan asymptotic (\ref{eq:partition-asymptotics}) for $p(m)$, we obtain
$$b(n) \geq (1+o(1))\tfrac{1}{8m\sqrt{3}} \exp(\pi \sqrt{2m/3}) b(n-k) \geq \frac{1}{C}\exp(c n^{1/12}) b(n-k),$$
where the last inequality holds for all $n$ and $k$ as in the statement of the proposition, provided $c$ is sufficiently small and $C$ is sufficiently large depending on $h_1$.

{\em Case (ii).} Now suppose that more than half the graphs in $\B(n-k)$ have largest component of order at least $(n-k)^{2/3}$. Let $\cH$ denote the set of such graphs; then $|\cH| > \tfrac{1}{2}|\B(n-k)|$. Let $G \in \cH$. Let $G_1,\ldots,G_L$ be the components of $G$, listed in non-decreasing order of size, i.e.\ $|V(G_1)| \leq |V(G_2)| \leq \ldots \leq |V(G_L)|$.
Let $a,d \in \mathbb{N}$ with $d \mid a$, $a \leq (n-k)^{2/3}$ and $a^2/(2d) \geq n-k$; we claim that there exists $i \in [L]$ such that $|V(G_{i-1})| \leq a$ and $|V(G_{i})|-|V(G_{i-1})| >d$ (where $|V(G_0)| :=0$). Indeed, suppose for a contradiction that for all $i \in [L]$ with $|V(G_{i-1})| \leq a$, we have $|V(G_{i+1})|-|V(G_i)| \leq d$; then
$$n-k = |V(G)| \geq d + 2d + \ldots + (a/d)d > a^2/(2d) \geq n-k,$$
a contradiction. From now on, fix $d= \lfloor \tfrac{1}{2}(n-k)^{1/6} \rfloor$ and $a= 2\lfloor \tfrac{1}{2}(n-k)^{1/6} \rfloor \lceil (n-k)^{5/12}\rceil$; clearly, $d \mid a$, and provided $n$ is at least an absolute constant, we have $a \leq (n-k)^{2/3}$ and $a^2/(2d) \geq n-k$. 
 
By the pigeonhole principle, there exists $s\leq a$ such that, for at least $|\cH|/a$ of the graphs $G$ in $\cH$, there exists $i$ such that $|V(G_{i-1})| =s$ and $|V(G_{i})|-|V(G_{i-1})| >d$. Let $\cJ_s$ be the set of all such graphs; then $|\cJ_s| \geq |\cH|/a \geq |\B(n-k)|/(2a)$. Let $s' > s$ be minimal such that $h_1 \mid s'$, and let $Q = s'/h_1$. Choose $d'$ maximal such that $h_1 \mid d'$ and $d' \leq d+s-s'$. (Note that $d+s-s' > d-h_1 \geq h_1$ provided $n$ is at least some constant depending upon $h_1$.) Let $M = \lfloor (k-d')/s' \rfloor$, and let $q_0 = k-d'-Ms'$. Let $m=d'/h_1 \in \mathbb{N}$. (Note that, provided $n$ is at least some constant depending upon $F$ and $r$, we have $q_0 \geq n_0$ and $m \leq M$.)

Now, given $G \in \cJ_s$ and a partition $\lambda = (\lambda_1,\lambda_2,\ldots) \vdash m$, we produce from $G$ a graph $H \in \B(n)$ as follows. For each $i \in [M]$, add to $G$ a connected, $((Q + \lambda_i) h_1)$-vertex component that is $r$-locally $F$. Since $s < s' = Q h_1 \leq (Q+\lambda_i)h_1 = s' + \lambda_i h_1 \leq s'+mh_1 = s'+d' \leq s+d$, all of these added components have orders between $s+1$ and $s+d$ (inclusive), whereas all orders of components of graphs in $\cJ_s$ are at most $s$ or greater than $s+d$. Now add an additional vertex-disjoint copy of a fixed, $q_0$-vertex graph $H_0$ that is $r$-locally $F$, producing a graph $H \in \B(n)$. (Note that the total number of vertices added is $MQh_1 + mh_1 +q_0 = Ms'+d'+q_0=k$.) Given a fixed graph $G \in \cJ_s$, each integer partition $\lambda \vdash m$ produces from $G$ in this way a different unlabelled graph $H \in \B(n)$, and for $G_1 \neq G_2 \in \cJ_s$, the sets of $n$-vertex unlabelled graphs $H \in \B(n)$ produced in this way by $G_1$ and $G_2$ are disjoint, since the graph in $\cJ_s$ that produced a given $H \in \B(n)$ may be recovered from $H$ by deleting first a copy of $H_0$, and then deleting the unique $M$ components of the remaining graph that have orders between $s+1$ and $s+d$ (inclusive). It follows that
$$b(n) \geq \tfrac{1}{2a} p(m) b(n-k).$$
Since $m > d/h_1-2 = \lfloor \tfrac{1}{2}(n-k)^{1/6} \rfloor / h_1 - 2 \geq \lfloor \tfrac{1}{2}(n/2)^{1/6} \rfloor / h_1 - 2$, and since, crudely, $a \leq n$, we obtain (again using the Hardy-Ramanujan asymptotic for the partition function)
$$b(n) \geq (1+o(1))\tfrac{1}{8nm\sqrt{3}} \exp(\pi \sqrt{2m/3}) b(n-k) \geq \frac{1}{C}\exp(c n^{1/12}) b(n-k),$$
provided $c$ is sufficiently small, and $n$ and $C$ are sufficiently large, depending upon $h_1$ (i.e., upon $F$ and $r$). By adjusting the values of $c$ and $C$ appropriately, one can ensure that the above holds for all $n \geq n_0$. This completes the proof of the claim.
\end{proof}

Applying Claim \ref{claim:fast-growth} yields a contradiction to (\ref{eq:small-quotient}) for all $n$ sufficiently large depending upon $\alpha$ and $h_1$, in the case where $\epsilon = n^{-1/6}$ and $\delta = \exp(-n^{1/13})$. This proves that there are only finitely many bad integers $n$ (with this choice of $\delta,\epsilon$), yielding Theorem \ref{thm:largest-component}.
\end{proof}

In the special case where $F = \mathbb{L}^d$, by Theorem \ref{thm:quotient-d} we may take $r_0=2+1_{\{d \geq 3\}}$ in the above proof, so we obtain the following.
\begin{theorem}
\label{thm:largest-component-d}
Define $r_0(2)=2$ and $r_0(d) = 3$ for all $d \geq 3$. Let $d \in \mathbb{N}$ with $d \geq 2$, and let $r \in \mathbb{N}$ with $r \geq r_0(d)$. Then the largest component of $G_n(\mathbb{L}^d,r)$ has order at most $n^{5/6}$ with probability at least $1-\exp(-n^{1/13})$, provided $n$ is sufficiently large depending upon $d$ and $r$.
\end{theorem}

In the other direction, we prove the following.

\begin{proposition}
\label{prop:largest-component-large}
Let $F$ be a connected, locally finite Cayley graph of a torsion-free group of polynomial growth. Let $r_0 = r_0(F)$ be as in Theorem \ref{thm:dlst-intro}. Then there exists $\epsilon_0 = \epsilon_0(F) >0$ such that for all $r \geq r_0$, the random graph $G_n(F,r)$ has a component of order at least $n^{\epsilon_0}$ with probability at least $1-\exp(-n^{1/4})$, 
provided $n$ is sufficiently large depending on $F$ and $r$.
\end{proposition}

\begin{proof}
Let $r \geq r_0$. Let $n_0 = n_0(F,r) \in \mathbb{N}$ and $\epsilon_0 = \epsilon_0(F) >0$ to be chosen later. Let $n \geq n_0$, and let $\G$ be the set of all graphs in $\B(n)$ with largest component of order less than $n^{\epsilon_0}$. Let $G \in \G$. Since $G$ has $n$ vertices and all its components have order less than $n^{\epsilon_0}$, $G$ has more than $n^{1-\epsilon_0}$ components, so there exists $k_1=k_1(G) < n^{\epsilon_0}$ such that $G$ has at least $n^{1-2\epsilon_0}$ components of order $k_1$. By Proposition \ref{prop:conn-upper}, there are at most $k_1^{\alpha}$ unlabelled, connected graphs on $k_1$ vertices that are $r$-locally $F$, so there exists a connected graph $H_1 = H_1(G)$ on $k_1 < n^{\epsilon_0}$ vertices such that at least $n^{1-2\epsilon_0}/k_1^{\alpha} \geq n^{1-(\alpha+2)\epsilon_0}$ components of $G$ are isomorphic to $H_1$.

We now define a bipartite graph $\mathbb{B}$ with vertex-bipartition $(\G,\B(n) \setminus \G)$, as follows. For each $G \in \G$, choose some connected graph $H_1=H_1(G)$ on $k_1 = k_1(G) < n^{\epsilon_0}$ vertices such that at least $n^{1-(\alpha+2)\epsilon_0}$ components of $G$ are isomorphic to $H_1$ (as above). Let $Q = \lfloor n^{1-(\alpha+2)\epsilon_0}/h_1 \rfloor$ (where $h_1$ is as in Proposition \ref{prop:linear}), and delete from $G$ exactly $Qh_1$ components that are isomorphic to $H_1$, producing a graph $G'$ with $n-Qh_1k_1$ vertices. Let $W = \lceil n^{\epsilon_0}/h_1 \rceil$, let $M = \lceil Qk_1/(W+1) \rceil$ and let $m = Qk_1-MW$; note that $m \leq M$. For any integer partition $\lambda  = (\lambda_1,\lambda_2,\ldots) \vdash m$, we may produce a graph $H_{\lambda,G} \in \B(n) \setminus \G$ by adding to $G'$ a connected, $(W+\lambda_i)h_1$-vertex component that is $r$-locally $F$, for each $i \in [M]$. (Note that the total number of vertices added is $MWh_1 + mh_1 = Qh_1k_1 = |V(G)|-|V(G')|$, and that each added component has order at least $Wh_1 \geq n^{\epsilon_0}$, i.e.\ it has order greater than the order of any component of $G$.) Now we define the edge-set of our bipartite graph $\mathbb{B}$ by joining $G$ to each of the $p(m)$ graphs $H_{\lambda,G}$ obtained in this way (for each $G \in \G$).

Each graph $G \in \G$ has degree exactly $p(m)$ in the bipartite graph $\mathbb{B}$, whereas each graph $H \in \B(n) \setminus \G$ has degree at most $n^{\alpha \epsilon_0}$ in $\mathbb{B}$. Indeed, if in the bipartite graph $\mathbb{B}$, a graph $H \in \B(n) \setminus \G$ is joined to some $G \in \G$, then the `intermediate' graph $G'$ depends only on $H$ and not on $G$, since $G'$ can be recovered from $H$ alone by deleting the $M$ components of $H$ that have order at least $n^{\epsilon_0}$. Given $G'$, there are at most $n^{\alpha \epsilon_0}$ possibilities for $G$, since (by Proposition \ref{prop:conn-upper}) there are at most $n^{\alpha \epsilon_0}$ connected graphs on less than $n^{\epsilon_0}$ vertices that are $r$-locally $F$, so there are at most $n^{\alpha \epsilon_0}$ choices for the component $H_1 = H_1(G)$ (and adding $Qh_1$ copies of $H_1$ to $G'$, we recover $G$).

Counting the edges of $\mathbb{B}$ in two different ways, we obtain
$$|\B(n) \setminus \G| n^{\alpha \epsilon_0} \geq |\G| p(m),$$
so
$$\frac{|\G|}{|\B(n)|} \leq \frac{n^{\alpha \epsilon_0}}{p(m)}.$$ 
It is easy to see that $m \geq (1+o(1))n^{1-(\alpha+3)\epsilon_0}$ as $n \to \infty$, and so using the Hardy-Ramanujan asymptotic (\ref{eq:partition-asymptotics}), we obtain
\begin{align*} \frac{|\G|}{|\B(n)|} & \leq (1+o(1))4\sqrt{3} mn^{\alpha \epsilon_0} \exp(-\pi \sqrt{2m/3})\\
& \leq C n \exp(-cn^{(1-(\alpha+3)\epsilon_0)/2})\\
& \leq \exp(-n^{1/4}),
\end{align*}
if we choose $\epsilon_0 = 1/(2\alpha+7)$ and $n_0$ sufficiently large depending on $F$ and $r$. (Here, $c$ and $C$ are positive absolute constants.) This proves the proposition.
\end{proof}

\begin{remark}
\label{remark:mult}
It is easy to adapt the proof of Proposition \ref{prop:largest-component-large} to show that with probability at least $1-\exp(-n^{1/4})$, the random graph $G_n = G_n(F,r)$ has a component of order that is both at least $n^{\epsilon_0}$ {\em and divisible by $h_1$}, where $h_1$ is as in Proposition \ref{prop:linear}, provided $n$ is sufficiently large depending on $F$ and $r$. (One simply replaces $\G$ in the proof by the set of all graphs in $\B(n)$ with no component of order that is both at least $n^{\epsilon_0}$ and divisible by $h_1$, observing that the components added to $G'$ all have orders that are multiples of $h_1$.) This fact is useful in the sequel.
\end{remark}

We now use Proposition \ref{prop:largest-component-large} and Remark \ref{remark:mult} to prove Theorem \ref{thm:aut-gen}.

\begin{proof}[Proof of Theorem \ref{thm:aut-gen}.]
Let $r_0 = r_0(F)$ be as in Theorem \ref{thm:dlst-intro}, and let $n \in \mathbb{N}$. Let $\delta_0 = \delta_0(F)= \epsilon_0/3$, where $\epsilon_0 = \epsilon_0(F)>0$ is as in Proposition \ref{prop:largest-component-large}. Let $\A$ denote the set of graphs in $\B(n)$ that have at most $n^{\delta_0}$ vertex-transitive components of order divisible by $h_1$, and let $\A'$ denote the set of graphs in $\A$ that have a component of order that is at least $n^{\epsilon_0}$ and divisible by $h_1$. (Here, $h_1$ is as in Proposition \ref{prop:linear}.) We define a bipartite graph $\mathbb{B}$ with vertex-bipartition $(\A',\B(n)\setminus \A)$, as follows. For each graph $G \in \A'$, let $H_2 = H_2(G)$ be a component of $G$ of maximal order divisible by $h_1$, and let $Q = Q_G = |V(H_2)|/h_1$. Delete $H_2$ from $G$, producing a graph $G'$ with $n-Qh_1$ vertices. For any integer partition $\lambda = (\lambda_1,\lambda_2,\ldots,\lambda_m) \vdash Q$, we may produce a graph $H_{\lambda,G} \in \B(n)\setminus \A$ by adding to $G'$ a vertex-transitive component of order $\lambda_i h_1$, for each $i \in [m]$. (Note that the total number of vertices added is $h_1 \sum_{i=1}^{l} \lambda_i = Q h_1$.) Now we define the edge-set of our bipartite graph $\mathbb{B}$ by joining $G$ to each of the $p(Q_G)$ graphs $H_{\lambda,G}$ obtained in this way (for each $G \in \G$). 

A graph $G \in \A'$ has degree exactly $p(Q_G)$ in the bipartite graph $\mathbb{B}$, whereas each graph $H \in \B(n)\setminus \A$ has degree at most $n^{\alpha+n^{\delta_0}}$ in $\mathbb{B}$. Indeed, if in the bipartite graph $\mathbb{B}$, a graph $H \in \B(n) \setminus \A$ is joined to some $G \in \A'$, then there are at most
$$\sum_{i=0}^{\lfloor n^{\delta_0}\rfloor} {n-1 \choose i} \leq n^{n^{\delta_0}}$$
ways of choosing the `intermediate' graph $G'$ (since $G'$ may be obtained from $H$ by deleting all but $j$ of the vertex-transitive components of $G$ of order divisible by $h_1$, for some $j \leq n^{\delta_0}$, and crudely, $H$ has at most $n/h_1 \leq n-1$ vertex-transitive components of order divisible by $h_1$). Given the intermediate graph $G'$, there are then at most $n^{\alpha}$ possibilities for $G$, since (by Proposition \ref{prop:conn-upper}) there are at most $n^{\alpha}$ connected graphs on at most $n$ vertices that are $r$-locally $F$, so there are at most $n^{\alpha}$ choices for the component $H_2 = H_2(G)$ (which, when added to $G'$, produces the graph $G$). 

Counting edges of $\mathbb{B}$ in two different ways, and using the fact that $\delta_0 = \epsilon_0/3$, we obtain
\begin{equation} \label{eq:multi}|\B(n)|n^{\alpha+n^{\epsilon_0/3}} \geq |\B(n) \setminus \A|n^{\alpha+n^{\delta_0}} \geq |\A'| \min_{G \in \mathcal{A}'} p(Q_G) \geq p(\lceil n^{\epsilon_0}/h_1 \rceil ) |\A'|.\end{equation}
On the other hand, by Proposition \ref{prop:largest-component-large} and Remark \ref{remark:mult}, provided $n$ is sufficiently large depending on $F$ and $r$, we have
\begin{equation}\label{eq:large-cpt} |\A \setminus \A'| \leq \exp(-n^{1/4})|\B(n)|.\end{equation}
Combining (\ref{eq:multi}) and (\ref{eq:large-cpt}) yields
$$\frac{|\A|}{|\B(n)|} \leq \exp(-n^{-\epsilon_0/3})$$
for all $n$ sufficiently large depending on $F$ and $r$, using the Hardy-Ramanujan asymptotic (\ref{eq:partition-asymptotics}). It follows that, with probability at least $1-\exp(-n^{-\delta_0})$, a uniform random graph $G_n \in \B(n)$ has more than $n^{\delta_0}$ vertex-transitive components of order divisible by $h_1$, so in particular has $|\Aut(G_n)| \geq (h_1)^{n^{\delta_0}} \geq 2^{n^{\delta_0}}$, proving the theorem.
\end{proof}

We now use the above machinery, combined with Brigham's theorem (Theorem \ref{thm:brigham}), to prove Theorem \ref{thm:stretched-exp}, which says that $b_{F,r}(n)$ grows like a stretched exponential for all $r \geq r_0$.

We will need the following easy corollary of Brigham's theorem.

\begin{corollary}
\label{corr:brigham}
Let $(b(n))_{n=0}^{\infty}$, $(\gamma(j))_{j=1}^{\infty}$ be as in Theorem \ref{thm:brigham}. Let $h$ be the highest common factor of the set $\{n \in \mathbb{N}:\ \gamma(n) \geq 1\}$. Then
$$\log b(n) \sim \frac{1}{u} (Ku \Gamma(u+2)\zeta(u+1))^{\frac{1}{u+1}} (u+1)^{\frac{u-v}{u+1}} n^{\frac{u}{u+1}}(\log n)^{\frac{v}{u+1}}$$
for all $n \in \mathbb{N}$ such that $h \mid n$.
\end{corollary}
(Note that the $h=1$ case of Corollary \ref{corr:brigham} appears in \cite{brigham}; the general case is a straightforward deduction therefrom, using the observation that any sufficiently large multiple of $h$ can be partitioned into integers in the set $\{n:\ \gamma(n) \geq 1\}$.)

\begin{proof}[Proof of Theorem \ref{thm:stretched-exp}.]

Observe that the generating function of $b_{F,r}(n)$ satisfies
\begin{equation}
\label{eq:gen-function-rel}
\sum_{n=0}^{\infty} b_{F,r}(n) z^n = \prod_{j=1}^{\infty} (1-z^j)^{-\gamma_{F,r}(j)}.
\end{equation}

Define the sequence $(\gamma'(n))_{n =1}^{\infty}$ by $\gamma'(n)=1$ if $n$ is a multiple of $h_1$, and $\gamma'(n)=0$ otherwise. Then, by Proposition \ref{prop:linear}, we have $\gamma_{F,r}(n) \geq \gamma'(n)$ for all $n \in \mathbb{N}$. Note that
\begin{equation} \label{eq:linear} \sum_{n \leq x} \gamma'(n) = \lfloor x/h_1 \rfloor\quad \forall x >0.\end{equation}

Define the sequence $(b'(n))_{n=0}^{\infty}$ by
$$\sum_{n=0}^{\infty} b'(n) z^n = \prod_{j=1}^{\infty} (1-z^j)^{-\gamma'(j)};$$
then, since $\gamma_{F,r}(n) \geq \gamma'(n)$ for all $n \in \mathbb{N}$, we have $b_{F,r}(n) \geq b'(n)$ for all $n \in \mathbb{N}$.

By Claim \ref{claim:mono}, we have $b_{F,r}(n) \geq b_{F,r}(m)$ for all $m,n \in \mathbb{N}$ such that $n \geq m+n_0$, so in particular, whenever $n \geq n_0$, we have
\begin{equation}\label{eq:comparison-remainder} b_{F,r}(n) \geq b_{F,r}(\lfloor (n-n_0)/h_1\rfloor h_1) \geq b'(\lfloor (n-n_0)/h_1\rfloor h_1).\end{equation}
Applying Corollary \ref{corr:brigham} (with $h=h_1$) to the sequence $(\gamma'(j))_{j=1}^{\infty}$, and using (\ref{eq:linear}), yields
$$\log b'(n) \geq (1+o(1))\pi \sqrt{2n/(3h_1)}$$
for $n \to \infty$ such that $h_1 \mid n$. Using (\ref{eq:comparison-remainder}) then yields
\begin{equation}
\label{eq:gen-lower}
\log b_{F,r}(n) \geq (1+o(1))\pi \sqrt{2n/(3h_1)}\end{equation}
for (all) $n \to \infty$.

On the other hand, if $r \geq r_0$, then by Proposition \ref{prop:conn-upper}, we have
\begin{equation}\label{eq:conn-upper} \gamma_{F,r}(n) \leq n^{\alpha}\quad \forall n \in \mathbb{N},\end{equation}
where $\alpha = \alpha(F)\geq 0$. Comparing the sequence $(\gamma_{F,r}(j))_{j=1}^{\infty}$ with the sequence $(\gamma''(j))_{j=1}^{\infty}$ defined by
$$\gamma''(j) = \begin{cases} \lfloor j^{\alpha}\rfloor & \text{ if } h_0\mid j;\\ 0 & \text{ if } h_0 \nmid j,\end{cases}$$
and appealing to Corollary \ref{corr:brigham}, yields 
\begin{equation}
\label{eq:gen-upper}
\log b_{F,r}(n) \leq (1+o(1))\frac{1}{\alpha+1} (\Gamma(\alpha+3)\zeta(\alpha+2)/h_0)^{\frac{1}{\alpha+2}} (\alpha+2)^{\frac{\alpha+1}{\alpha+2}} n^{\frac{\alpha+1}{\alpha+2}}
\end{equation}
for $n \to \infty$. Combining the upper bound (\ref{eq:gen-upper}) with the lower bound (\ref{eq:gen-lower}) yields the theorem.
\end{proof}

\section{Typical properties of graphs that are $r$-locally $\mathbb{L}^d$}
\label{sec:typical}
In this section, we use some of the results and techniques of Section \ref{sec:proof-enum} to obtain some more precise results on the typical properties of unlabelled, $n$-vertex graphs which are $r$-locally $\mathbb{L}^d$, for various integers $r$ and $d$. We first prove that if $r \geq r^*(d)$ (where $r^*(d)$ is as defined in Theorem \ref{thm:enumeration}), then with high probability, all but at most a $1/\poly(n)$ fraction of the vertices of the random graph $G_n = G_n(\mathbb{L}^d,r)$ lie in components that are isomorphic to a quotient of $\mathbb{L}^d$ by a pure-translation subgroup of $\Aut(\mathbb{L}^d)$. (Such quotients can be viewed as quotient lattices of $\mathbb{L}^d$ inside some $d$-dimensional torus $\mathbb{R}^d/\Lambda$, where $\Lambda$ is a rank-$d$ sublattice of $\mathbb{Z}^d$, i.e.\ the lattice of translations of some pure-translation subgroup $\Gamma \leq \Aut(\mathbb{L}^d)$ with $|\mathbb{Z}^d/\Gamma| < \infty$; see page \pageref{page:orbifolds}, or alternatively \cite{be}.) This is the content of Theorem \ref{thm:local-limit-d}. We then obtain (in Proposition \ref{prop:aut-large}) a more precise lower bound on $|\Aut(G_n(\mathbb{L}^d,r))|$ than that which can be deduced from Theorem \ref{thm:aut-gen}.

We will need the following ratio estimate, an easy consequence of Theorem \ref{thm:enumeration}.
\begin{lemma}
\label{lemma:ratio}
Let $d \in \mathbb{N}$ with $d \geq 2$, and let $K_d,\epsilon_d,r^*(d)$ be as in Theorem \ref{thm:enumeration}. Let $r \in \mathbb{N}$ with $r \geq r^*(d)$. There exists $\beta_{d,r} >0$ such for all $m \geq \beta_{d,r} n^{1-\epsilon_d}$, we have
$$\frac{a_{d,r}(n-m)}{a_{d,r}(n)} \leq \exp(-\tfrac{1}{5}K_dmn^{-\frac{1}{d+1}}).$$
\end{lemma}
\begin{proof}
By taking $\beta_{d,r}$ sufficiently large, we may assume that $n \geq n_0(d,r)$ for any function $n_0 = n_0(d,r)$. We first suppose that $m \leq n/2$. Using the fact that $(1-x)^{a} \leq 1-a x$ for all $a,x \in (0,1)$, we have
\begin{align*} \log (a_{d,r}(n)) - \log (a_{d,r}(n-m)) & = (1+O(n^{-\epsilon_d})) K_d n^{\frac{d}{d+1}}\\
& - (1+O((n-m)^{-\epsilon_d})) K_d (n-m)^{\frac{d}{d+1}}\\
& = K_dn^{\frac{d}{d+1}} (1+O(n^{-\epsilon_d}) - (1-m/n)^{\frac{d}{d+1}})\\
& \geq K_d n^{\frac{d}{d+1}}(1+O(n^{-\epsilon_d}) - 1 + \tfrac{d}{d+1}\tfrac{m}{n})\\
& = K_d n^{\frac{d}{d+1}}(\tfrac{d}{d+1}\tfrac{m}{n} + O(n^{-\epsilon_d}))\\
& \geq \tfrac{1}{2}K_d n^{\frac{d}{d+1}}m/n\\
& = \tfrac{1}{2}K_d mn^{-\frac{1}{d+1}},\end{align*}
provided $\beta_{d,r}$ is sufficiently large depending on $d$ and $r$. Now suppose that $m \geq n/2$. Given a graph on $k$ vertices which is $r$-locally $\mathbb{L}^d$, we may produce a graph on $k+(2r+2)^d$ vertices which is $r$-locally $\mathbb{L}^d$ by adding a vertex-disjoint copy of the $d$-dimensional discrete torus $C_{2r+2}^{d}$. Hence, $a_{d,r}(k) \leq a_{d,r}(k+(2r+2)^d)$ for all $k \geq 0$. It follows that $a_{d,r}(n-m) \leq a_{d,r}(n - m')$ for some $m' \in \mathbb{N}$ with $n/2 - (2r+2)^d \leq m' \leq n/2$. Hence, we have
\begin{align*} \log (a_{d,r}(n)) - \log (a_d(n-m)) & \geq \log (a_{d,r}(n)) - \log (a_{d,r}(n-m'))\\
& \geq \tfrac{1}{2}K_d m'n^{-\frac{1}{d+1}}\\
& \geq \tfrac{1}{5}K_dmn^{-\frac{1}{d+1}},\end{align*}
provided $n$ is sufficiently large depending on $d$ and $r$. Taking exponents proves the lemma.
\end{proof}

We now prove the following. 

\begin{theorem}
\label{thm:local-limit-d}
Let $d \in \mathbb{N}$ with $d \geq 2$, and let $r,R \in \mathbb{N}$ with $R \geq r \geq r^*(d)$. There exists $\kappa = \kappa_{d,R} >0$ depending upon $d$ and $R$ alone, such that with high probability, all but at most $n^{1-\kappa}$ of the vertices of $G_n = G_n(\mathbb{L}^d,r)$ are in components that are $R$-locally $\mathbb{L}^d$ and isomorphic to a quotient of $\mathbb{L}^d$ by a pure-translation subgroup of $\Aut(\mathbb{L}^d)$.
\end{theorem}

To prove Theorem \ref{thm:local-limit-d}, we need a straightforward corollary of Theorem \ref{thm:brigham}.
\begin{corollary}
\label{corr:upper-bound}
Suppose $(b(n))_{n=0}^{\infty}$ is a sequence of positive integers with generating function satisfying
\begin{equation} \label{eq:def-b} \sum_{n=0}^{\infty} b(n) z^n = \prod_{j=1}^{\infty} (1-z^j)^{-\gamma(j)},\end{equation}
where $(\gamma(j))_{j=1}^{\infty}$ is a sequence of non-negative integers satisfying
$$\sum_{j \leq x} \gamma(j) \leq K x^{u} (\log x)^{v}$$
for some constants $K >0,\ u >0$, $v \in \mathbb{R}$. Then there exists a constant $K'>0$ (depending only upon $u,v$ and $K$) such that 
$$\log b(n) \leq K' n^{u/(u+1)}(\log n)^{v/(u+1)} \quad \forall n \in \mathbb{N}.$$
\end{corollary}
\begin{proof}
Let $(\eta(j))_{j=1}^{\infty}$ be the sequence of non-negative integers defined by
\begin{equation} \label{eq:partial-sums} \sum_{j \leq N}\eta(j) = \lceil K(N+1)^{u} (\log (N+1))^{v} \rceil+1\quad (\forall N \in \mathbb{N})\end{equation}
Then we have
\begin{equation} \label{eq:ineq} \sum_{j \leq N}\gamma(j) < \sum_{j \leq N} \eta(j) \quad (\forall N \in \mathbb{N}).\end{equation}
Define $(a(n))_{n=0}^{\infty}$ by
\begin{equation} \label{eq:def-a} \sum_{n=0}^{\infty} a(n) z^n = \prod_{j=1}^{\infty} (1-z^j)^{-\eta(j)}.\end{equation}
We make the following.
\begin{claim} If (\ref{eq:ineq}) holds, and $(a(n)),(b(n))$ are defined by (\ref{eq:def-a}) and (\ref{eq:def-b}) respectively, then $b(n) \leq a(n)$ for all $n \in \mathbb{N}$.
\end{claim}
\begin{proof}[Proof of claim.]
Recall from Fact \ref{fact:multisets} that $b(n)$ is the number of weight-$n$ multisets of elements of $\mathcal{T}$, where $\mathcal{T}$ contains exactly $\gamma(j)$ elements of weight $j$ for all $j \in \mathbb{N}$, and $a(n)$ is the number of weight-$n$ multisets of elements of $\mathcal{S}$, where $\mathcal{S}$ contains exactly $\eta(j)$ elements of weight $j$ for all $j \in \mathbb{N}$. Let $\mathcal{A}(n)$ (respectively $\mathcal{B}(n)$) denote the set of weight-$n$ multisets of elements of $\mathcal{S}$ (respectively $\mathcal{T}$); then $a(n) = |\mathcal{A}(n)|$ and $b(n) = |\mathcal{B}(n)|$. Fix $n \in \mathbb{N}$. It suffices to construct an injection $\Phi:\mathcal{B}(n) \to \mathcal{A}(n)$. Choose any element $S_1 \in \mathcal{S}$ of weight 1. Let $\mathcal{S}_{\leq n}$, $\mathcal{T}_{\leq n}$ denote the set of elements of $\mathcal{S}$ (respectively $\mathcal{T}$) of weight at most $n$. Let $w:\mathcal{S} \cup \mathcal{T} \to \mathbb{N}$ denote the weight function. We first construct an injection $f:\mathcal{T}_{\leq n} \to \mathcal{S}_{\leq n} \setminus \{S_1\}$ such that $w(T) \geq w(f(T))$ for all $T \in \mathcal{T}_{\leq n}$. Define $\eta'(1) = \eta(1)-1$ and $\eta'(j) = \eta(j)$ for all $j \geq 2$. Then the set $\mathcal{S} \setminus \{S_1\}$ contains exactly $\eta'(j)$ elements of weight $j$, for each $j \in \mathbb{N}$, and by (\ref{eq:ineq}), we have
\begin{equation} \label{eq:ineq2} \sum_{j \leq N}\gamma(j) \leq \sum_{j \leq N} \eta'(j) \quad (\forall N \in \mathbb{N}).\end{equation}
We define $f$ inductively. Order the elements of $\mathcal{T}_{\leq n}$ in non-decreasing order of weight, say as $T_1,T_2,\ldots,T_M$. If we have already defined $f(T_1),\ldots,f(T_{i-1})$, define $f(T_i)$ to be an element of $(\mathcal{S}\setminus \{S_1\}) \setminus \{f(T_j):\ j < i\}$ of minimal weight. By (\ref{eq:ineq2}), we have $w(f(T)) \leq w(T)$ for all $T \in \mathcal{T}_{\leq n}$, so $f(\mathcal{T}_{\leq n}) \subset \mathcal{S}_{\leq n} \setminus \{S_1\}$, as needed. Now for any multiset $X = (B_1,B_2,\ldots,B_k) \in \mathcal{B}(n)$, define
$$\Phi(X) = (f(B_1),f(B_2),\ldots,f(B_k),S_1,\ldots,S_1),$$
where the number of $S_1$'s is equal to $\sum_{i=1}^{k} w(B_i) - \sum_{i=1}^{k} w(f(B_i))$. Clearly, $\Phi(X)$ is an injection from $\mathcal{B}(n)$ to $\mathcal{A}(n)$. This proves the claim.
\end{proof}
By (\ref{eq:partial-sums}), we have
$$ \sum_{j \leq x} \eta(x) = (1+o(1))K x^{u} (\log x)^{v}.$$
Hence, by Theorem \ref{thm:brigham}, we have
$$\log a(n) = (1+o(1)) K_1 n^{u/(u+1)}(\log n)^{v/(u+1)},$$
for some constant $K_1>0$ depending only upon $u$, $v$ and $K$. Hence, we have
$$\log b(n) \leq \log a(n) \leq K' n^{u/(u+1)}(\log n)^{v/(u+1)}$$
for some constant $K'>0$ depending only upon $u$, $v$ and $K$, proving the corollary.
\end{proof}

\begin{proof}[Proof of Theorem \ref{thm:local-limit-d}]
Let $d \geq 2$ and $R \geq r \geq r^*(d)$. Note that, by adjusting the value of $\kappa_{d,R}$ if necessary, we may assume throughout that $n \geq n_0(d,R)$, for any function $n_0 = n_0(d,R)$.

Define $\gamma_{d,r,R}(n)$ (respectively, $b_{d,r,R}(n)$) to be the number of connected (respectively, possibly disconnected), unlabelled, $n$-vertex graphs which are $r$-locally $\mathbb{L}^d$ and either not $R$-locally $\mathbb{L}^d$ or else not isomorphic to a quotient of $\mathbb{L}^d$ by a pure-translation subgroup. Then, using Fact \ref{fact:multisets}, we have
$$\sum_{n=0}^{\infty} b_{d,r,R}(n) z^n = \prod_{j=1}^{\infty} (1-z^j)^{-\gamma_{d,r,R}(j)}.$$
Observe that
\begin{equation}\label{eq:bad-small} \sum_{n \leq x} \gamma_{d,r,R}(n) = O_{d,R}(x^{d-1+O(1/\log \log n)})\quad \forall x \geq 1.\end{equation}
Indeed, the above sum is precisely the number of conjugacy-classes of subgroups $\Gamma \leq \Aut(\mathbb{L}^d)$ such that $|\mathbb{Z}^d/\Gamma| \leq x$ and $2r+2 \leq D(\Gamma) \leq 2R+1$. The total number of pure-translation subgroups with this property is precisely the number of sublattices of $\mathbb{Z}^d$ with index at most $x$ and minimum distance in $\{2r+2,2r+3,\ldots,2R+1\}$, which, by Lemma \ref{lemma:small-distance}, is at most
$$O_{d,R}(x^{d-1} \log x).$$
Moreover, the total number of subgroups $\Gamma \leq \Aut(\mathbb{L}^d)$ with $|\mathbb{Z}^d/\Gamma| \leq x$ and $D(\Gamma) \geq 2r^*(d)+1$, such that $\Gamma$ is {\em not} a pure-translation subgroup, is at most 
$$O_{d}(x^{d-1+O(1/\log \log x)}),$$
by Corollary \ref{corr:non-translation} and Claim \ref{claim:involution}. Hence, we have
$$\sum_{n \leq x} \gamma_{d,r,R}(n) = O_{d,R}(x^{d-1+O(1/\log \log x)}),$$
as desired.

Slightly more crudely, it follows that
$$\sum_{n \leq x} \gamma_{d,r,R}(n) = O_{d,R}(x^{d-1/2}).$$
It follows from Corollary \ref{corr:upper-bound} that
$$\log b_{d,r,R}(n) \leq K_{d,R} n^{(d-1/2)/(d+1/2)} = K_{d,R} n^{(2d-1)/(2d+1)} \quad \forall n \in \mathbb{N},$$
where $K_{d,R} >0$ is a constant depending upon $d$ and $R$ alone.

Now let $H$ be a finite, simple, connected graph which is $r$-locally $\mathbb{L}^d$. Let us say that $H$ is {\em $R$-bad} if $H$ is not $R$-locally $\mathbb{L}^d$ or if $H$ is not isomorphic to a quotient of $\mathbb{L}^d$ by a pure-translation subgroup. Let $q_{d,R}=q_{d,R}(n)$ be a function of $n$ to be chosen later. We shall bound the probability that the random graph $G_n(\mathbb{L}^d,r)$ has at least $q_{d,R}(n)$ of its vertices in $R$-bad components. Fix an integer $m \geq q_{d,R}(n)$. Using Lemma \ref{lemma:ratio}, it follows that the probability that $G_n$ has exactly $m$ vertices in $R$-bad components is at most
\begin{align*} \frac{b_{d,r,R}(m) a_{d,r}(n-m)}{a_{d,r}(n)} & \leq \exp(K_{d,R} m^{\frac{2d-1}{2d+1}}) \exp(-\tfrac{1}{5}K_dmn^{-\frac{1}{d+1}})\\
&  \leq \exp(-\tfrac{1}{6}K_dn^{(2d-1)/(2d+2)}),
\end{align*}
provided we choose
$$q_{d,R}(n) = \max\{\max\{\beta_{d,r}:r \leq R\}n^{1-\epsilon_d}, L_{d,R} n^{(2d+1)/(2d+2)}\}$$
for some $L_{d,R} >0$ sufficiently large. Hence, by the union bound, the probability that $G_n$ has at least $q_{d,R}(n)$ vertices in $R$-bad components is at most
$$n \exp(-\tfrac{1}{6}K_d n^{(2d-1)/(2d+2)}) = o(1).$$
Note that for each fixed $d,R$, we have $q_{d,R}(n) \leq n^{1-\kappa_{d,R}}$ for $n$ sufficiently large depending on $d$ and $R$, provided we choose $\kappa_{d,R}>0$ to be sufficiently small. This proves Theorem \ref{thm:local-limit-d}.
\end{proof}

From Theorem \ref{thm:local-limit-d}, it follows immediately that the local limit of $(G_n)$ is the rooted lattice $(\mathbb{L}^d,0)$. For the reader's convenience, we recall the definition of a local limit. If $(F,u)$ and $(G,w)$ are two rooted graphs, we say that they are {\em isomorphic as rooted graphs} if there exists a graph isomorphism $\phi:V(F) \to V(G)$ such that $\phi(u)=w$. Let $(G,w)$ be a random rooted graph, i.e.\ a probability distribution on the set of rooted graphs. Following \cite{bs}, we say that a sequence of graphs $(G_n)_{n \in \mathbb{N}}$ has {\em local limit} $(G,w)$ as $n \to \infty$, if for every $R \in \mathbb{N}$ and for every rooted graph $(H,v)$, as $n \to \infty$, the probability that $\Link_R(G_n,w_n)$ is isomorphic to $(H,v)$ (as a rooted graph) converges to the probability that $\Link_R(G,w)$ is isomorphic to $(H,v)$ (as a rooted graph), where $w_n$ is a vertex chosen uniformly at random from $V(G_n)$. In particular, if $(G,w)$ is constant, the sequence $(G_n)_{n \in \mathbb{N}}$ has local limit $(G,w)$ if the probability that $\Link_R(G_n,w_n)$ is isomorphic to $\Link_R(G,w)$ (as a rooted graph) tends to 1 as $n \to \infty$, where $w_n$ is a vertex chosen uniformly at random from $V(G_n)$.

Now let $d \geq 2$, let $r \geq r^*(d)$ and let $G_n = G_n(\mathbb{L}^d,r)$. Let $T = \{n: \mathcal{S}_n \neq \emptyset\}$, where $\mathcal{S}_n$ denotes the set of all unlabelled, $n$-vertex graphs that are $r$-locally $\mathbb{L}^d$. (Recall that, by (\ref{eq:one-Cayley}), we have $n \in T$ for all $n$ sufficiently large depending on $d$ and $r$.) Theorem \ref{thm:local-limit-d} implies that for any $R \in \mathbb{N}$, the probability that $\Link_R(G_n,w_n)$ is isomorphic to $\Link_R(\mathbb{L}^d,0)$ tends to 1 as $n \to \infty$, so the local limit of $(G_n)_{n \in T}$ is $(\mathbb{L}^d,0)$, as claimed.

Theorems \ref{thm:largest-component-d} and \ref{thm:local-limit-d} also easily imply the following.

\begin{proposition}
\label{prop:aut-large}
Let $d \in \mathbb{N}$ with $d \geq 2$, and let $r \in \mathbb{N}$ with $r \geq r^*(d)$. Then with high probability,
$$|\Aut(G_n(\mathbb{L}^d,r))| \geq \exp(\Omega(n^{1/6} \log n)).$$
\end{proposition}

\begin{proof}
We may assume that $n \geq n_0(d,r)$, for any function $n_0 = n_0(d,r)$. Observe that if $H = \mathbb{L}^d/\Gamma$ for some pure-translation subgroup $\Gamma \leq \Aut(\mathbb{L}^d)$, then $H$ is vertex-transitive, since the translation map
$$t_a:\mathbb{Z}^d/L_{\Gamma} \to \mathbb{Z}^d/L_{\Gamma};\quad x+L_{\Gamma} \mapsto x+a+L_{\Gamma}$$
is an automorphism of $H$, for any $a \in \mathbb{Z}^d$. By Theorem \ref{thm:largest-component-d}, with high probability the largest component of $G_n$ has order at most $n^{5/6}$, and by Theorem \ref{thm:local-limit-d}, with high probability, all but at most $n/2$ vertices of $G_n$ are in components of $G_n$ that are isomorphic to a quotient of $\mathbb{L}^d$ by a pure-translation subgroup of $\Aut(\mathbb{L}^d)$, provided $n_0$ is sufficiently large depending on $d$ and $r$. If these two conditions hold, and there are $k$ components ($H_1,\ldots,H_k$, say) that are isomorphic to a quotient of $\mathbb{L}^d$ by a pure-translation subgroup of $\Aut(\mathbb{L}^d)$, then we have
$$\sum_{i=1}^{k} |H_i| \geq n/2,\quad 3 \leq |H_i| \leq n^{5/6}\ \forall i \in [k],$$
so
$$|\Aut(G_n)| \geq \prod_{i=1}^{k} |\Aut(H_i)| \geq \prod_{i=1}^{k}|H_i| \geq (n^{5/6})^{\lfloor n/(2n^{5/6}) \rfloor} = \exp(\Omega(n^{1/6} \log n)).$$
This proves the proposition.
\end{proof}

\section{Appendix}
In this section, we prove Theorem \ref{thm:brigham-error}, by adapting Brigham's proof in \cite{brigham}, and using a theorem of Odlyzko \cite{odlyzko} in place of Theorem A of Hardy and Ramanujan in \cite{hardy-ramanujan-general}, as in the proof of Brigham's theorem sketched in \cite{odlyzko}.

\begin{proof}[Proof of Theorem \ref{thm:brigham-error}.]
Throughout the proof, `for $s$ sufficiently small' will mean for all $s$ with $0 < s < s_0$, where $s_0 >0$ depends only upon the constants $K$, $u$ and $\epsilon$. If $f,g:X \to \mathbb{R}$ are functions, we will write $g = O(f)$ to mean that there exists a constant $C$ depending on $K,u$ and $\epsilon$, such that $|g(x)| \leq C|f(x)|$ for all $x \in X$. Moreover, $\kappa$ will be used to denote several (possibly different) positive constants depending upon $u$ and $\epsilon$ alone.

Define
$$g(s) := \prod_{n=1}^{\infty} (1-e^{-sn})^{-\gamma(n)} = \sum_{n=0}^{\infty} b(n) e^{-sn},$$
and define
$$h(s): = \log g(s).$$
Then we have
\begin{equation}\label{eq:taylor-expansion} h(s) = \sum_{n=1}^{\infty} \gamma(n) \log  ((1-e^{-sn})^{-1}) = \sum_{m=1}^{\infty} \frac{1}{m} \phi(ms),\end{equation}
where
$$\phi(s) := \sum_{n=1}^{\infty} \gamma(n) e^{-sn}.$$
We first make the following.
\begin{claim}
\label{claim:approx}
Under the hypotheses of Theorem \ref{thm:brigham-error}, there exists $\eta = \eta(u,\epsilon)>0$ such that 
$$h(s) = (1+O(s^{\eta})) K \Gamma(u+1) \zeta(u+1) s^{-u} ,$$
for all $s \in (0,1]$.
\end{claim}
\begin{proof}[Proof of claim.]
By replacing $\epsilon$ with $\min\{\epsilon,u\}$, we may assume throughout that $\epsilon \leq u$. Note that $\log g(s)$ is a monotone decreasing function of $s$. Therefore, once we have proved the claim for $s \in (0,s_0)$ for some $s_0 >0$, it will follow for all $s \in (0,1]$, by adjusting the value of the constant implicit in the $O(s^{\eta})$ term, if necessary. Hence, it suffices to prove the claim for $s$ sufficiently small. 

Let
$$F(x) := \sum_{n \leq x} \gamma(n).$$
Then, by hypothesis, we have
$$F(x) = (1+O(x^{-\epsilon})Kx^u.$$
From the definition of $\phi$, we have
\begin{align}
\label{eq:phi-formula} \phi(s) & =-\sum_{q=1}^{\infty} \sum_{n=1}^{q} \gamma(n) (e^{-(q+1)s} - e^{-qs})\nonumber \\
& = \sum_{q=1}^{\infty} \sum_{n=1}^{q} \gamma(n) \int_{q}^{q+1} se^{-sy} dy \nonumber \\
& = \sum_{q=1}^{\infty}  F(q) \int_{q}^{q+1} se^{-sy} dy \nonumber \\
& = \sum_{q=1}^{\infty}  \int_{q}^{q+1} F(y) se^{-sy} dy \nonumber \\
& = \int_{1}^{\infty} F(y) se^{-sy} dy\\
& = \int_{s}^{\infty} F(x/s) e^{-x} dx. \nonumber
\end{align}

Let us write
\begin{equation}\label{eq:split} \phi(s) = \int_{s}^{s^{1/2}} F(x/s) e^{-x} dx + \int_{s^{1/2}}^{\infty} F(x/s)e^{-x} ds.\end{equation}
Since $F$ is monotone increasing, we have
$$F(x/s) \leq F(s^{-1/2}) \leq (1+O(s^{\epsilon/2}))K s^{-u/2} = O(s^{-u/2})\quad \forall x \in [s,s^{1/2}],$$
for all $s\in (0,1]$. Hence, for the first integral in (\ref{eq:split}), we have
$$\left| \int_{s}^{s^{1/2}} F(x/s) e^{-x} dx \right| \leq O(s^{-u/2})s^{1/2} = O(s^{-u/2+1/2}).$$
By hypothesis, we have
$$F(x/s) = K(x/s)^{u} (1+O((x/s)^{-\epsilon})) = K(x/s)^{u}(1+O(s^{\epsilon/2})) \quad \forall x \geq s^{1/2},$$
so
$$\int_{s^{1/2}}^{\infty} F(x/s)e^{-x} ds = Ks^{-u} (1+O(s^{\epsilon/2}))\int_{s^{1/2}}^{\infty} x^{u}e^{-x} dx.$$
We have
$$\int_{s^{1/2}}^{\infty} x^{u}e^{-x} dx = \int_{0}^{\infty} x^{u}e^{-x} dx - \int_{0}^{s^{1/2}} x^{u}e^{-x} dx = \Gamma(u+1) -  \int_{0}^{s^{1/2}} x^{u}e^{-x} dx,$$
and
$$\int_{0}^{s^{1/2}} x^{u}e^{-x} dx \leq \int_{0}^{s^{1/2}} x^u dx = \frac{s^{(u+1)/2}}{u+1} = O(s^{1/2}).$$
Hence, the second integral in (\ref{eq:split}) satisfies
$$\int_{s^{1/2}}^{\infty} F(x/s)e^{-x} ds = Ks^{-u} (1+O(s^{\epsilon/2}))\Gamma(u+1) - O(s^{-u+1/2}).$$
Putting everything together, we get
\begin{align}
\label{eq:phi-bound} \phi(s) & = K\Gamma(u+1) s^{-u} (1+O(s^{\epsilon/2})) + O(s^{-u/2+1/2}) + O(s^{-u+1/2})\nonumber \\
& = K\Gamma(u+1) s^{-u} (1+O(s^{\epsilon/2}))\end{align}
for all $s \in (0,1]$. 

Now write 
\begin{align*} h(s) & = \sum_{m=1}^{\infty} \frac{1}{m} \phi(ms)\\
& = \sum_{m \leq s^{-1/2}}\frac{\phi(ms)}{m} + \sum_{s^{-1/2} < m \leq s^{-1}}\frac{\phi(ms)}{m} + \sum_{m > s^{-1}} \frac{\phi(ms)}{m}\\
& =: \psi_1(s)+\psi_2(s)+\psi_3(s).\end{align*}

We have
\begin{align*} \psi_1(s) & = \sum_{m \leq s^{-1/2}} \frac{1}{m}\phi(ms)\\
&= K\Gamma(u+1) \sum_{m \leq s^{-1/2}} \frac{1}{m} (ms)^{-u} (1+O((ms)^{\epsilon/2}))\\
& = K\Gamma(u+1) s^{-u} \left( \sum_{m \leq s^{-1/2}} \frac{1}{m^{u+1}} + O(s^{\epsilon/2}) \sum_{m \leq s^{-1/2}} \frac{1}{m^{u+1 - \epsilon/2}}\right)\\
& = K\Gamma(u+1) s^{-u} \left( \sum_{m \leq s^{-1/2}} \frac{1}{m^{u+1}} + O(s^{\epsilon/2})\right)\\
& = K\Gamma(u+1) s^{-u} \left( \sum_{m=1}^{\infty} \frac{1}{m^{u+1}} - \sum_{m > s^{-1/2}}  \frac{1}{m^{u+1}} +  O(s^{\epsilon/2})\right)\\
& = K\Gamma(u+1) s^{-u}(\zeta(u+1) - O(s^{u/2}) +  O(s^{\epsilon/2}))\\
& = K\Gamma(u+1) s^{-u}\zeta(u+1)(1+O(s^{\epsilon/2})),\end{align*}
using the fact that $\epsilon \leq u$.

By (\ref{eq:phi-bound}), we have $\phi(s) = O(s^{-u})$ for all $s \in (0,1]$. Hence, 
\begin{align*} \psi_2(s) & = \sum_{s^{-1/2} < m \leq s^{-1}} \frac{\phi(ms)}{m}\\
& \leq O(1) \sum_{s^{-1/2} < m \leq s^{-1}} \frac{(ms)^{-u}}{m}\\
& = O(s^{-u}) \sum_{s^{-1/2} < m \leq s^{-1}} \frac{1}{m^{u+1}}\\
& = O(s^{-u}) \int_{s^{1/2}}^{s^{-1}} \frac{1}{y^{u+1}} dy\\
& = O(s^{-u}) \frac{1}{u} \left[-\frac{1}{y^u}\right]_{s^{-1/2}}^{s^{-1}}\\
& = O(s^{-u}) O(s^{u/2})\\
& = O(s^{-u/2}) \end{align*}

Finally, we deal with
$$\psi_3(s) = \sum_{m > s^{-1}} \frac{\phi(ms)}{m}.$$

By (\ref{eq:phi-formula}), for $t \geq 1$, we have
\begin{align*} \phi(t)& = \int_{1}^{\infty} F(y) te^{-ty} dy\\
& = t \int_{1}^{\infty} F(y) e^{-ty} dy \\
& = O(t) \int_{1}^{\infty} y^u e^{-ty} dy \\
&= O(te^{-t})\\
&= O(2^{-t}).\end{align*}

Hence,
\begin{align*} \psi_3(s) & = \sum_{m > s^{-1}} \frac{\phi(ms)}{m}\\
& = O(1) \sum_{m > s^{-1}} \frac{2^{-ms}}{m}\\
& = O(s) \sum_{m > s^{-1}} 2^{-ms}\\
& = O(s) \frac{1}{1-2^{-s}}\\
& = O(1).\end{align*}

Putting everything together, we have
\begin{align*} h(s) & = K\Gamma(u+1) s^{-u}\zeta(u+1)(1+O(s^{\epsilon/2})) + O(s^{-u/2}) + O(1)\\
& = (1+O(s^{\epsilon/2}))K\Gamma(u+1) \zeta(u+1) s^{-u},\end{align*}
proving the claim.
\end{proof}

Given $y \geq 0$, let $s(y)$ be the value of $s$ minimising $h(s) + sy$. Then we have
\begin{equation}\label{eq:stationary} h'(s(y)) = -y.\end{equation}

We now make the following.
\begin{claim}
\label{claim:derivative}
$$h'(s)  = -(1+O(s^{\kappa}))\frac{Ku\Gamma(u+2)\zeta(u+1)}{u+1} s^{-u-1}$$
for all $s \in (0,1]$.
\end{claim}

\begin{proof}[Proof of claim.]
By (\ref{eq:taylor-expansion}), we have
$$h'(s) = \sum_{m=1}^{\infty} \phi'(ms),$$
and
$$\phi'(s) = -\sum_{n=1}^{\infty} n\gamma(n) e^{-sn}.$$
Let
$$F_1(x) = \sum_{n \leq x} n \gamma (n).$$
Recall the Abel summation formula: if $a,f:\mathbb{N} \to \mathbb{N}$, and $A(n):= \sum_{i=1}^{n} f(i)$, then we have
\begin{equation} \label{eq:Abel} \sum_{n=1}^{N} a(n) f(n) = A(N)f(N) - \sum_{n=1}^{N-1} A(n) (f(n+1)-f(n)).\end{equation}
Applying this with $a(n) = \gamma(n)$ and $f(n) = n$ gives
\begin{align*} F_1(N) & = \sum_{n=1}^{N} n\gamma(n)\\
 & = NF(N) - \sum_{n=1}^{N-1} F(n)\\
& = N(1+O(N^{-\epsilon})) KN^u - \sum_{n=1}^{N-1} Kn^u(1+O(n^{-\epsilon}))\\
& = \frac{Ku}{u+1}(1+O(N^{-\epsilon}))N^{u+1}.\end{align*}
Therefore,
$$F_1(x) = \frac{Ku}{u+1}(1+O(x^{-\epsilon}))x^{u+1}.$$
Hence, using exactly the same argument as in the proof of Claim \ref{claim:approx}, with $-\phi'$ in place of $\phi$, and with $F_1$ in place of $F$, we obtain
\begin{align} \label{eq:F1integral} -\phi'(s) & = \int_{1}^{\infty} F_1(y) se^{-sy} dy\\
& = \int_{s}^{\infty} F(x/s) e^{-x} dx \nonumber \end{align}
and
\begin{equation}\label{eq:deriv-approx}
-\phi'(s) = \frac{Ku\Gamma(u+2)}{u+1} s^{-u-1} (1+O(s^{\epsilon/2}))
\end{equation}
for all $s \in (0,1]$. Using the functional equation of $\Gamma(x+1) = x\Gamma(x)$ of the Gamma function, (\ref{eq:deriv-approx}) can be rewritten as
\begin{equation}\label{eq:deriv-approx-2}
-\phi'(s) = Ku\Gamma(u+1) s^{-u-1} (1+O(s^{\epsilon/2}))
\end{equation}

As before, we write
\begin{align*} -h'(s) & = \sum_{m=1}^{\infty} -\phi'(ms)\\
& = \sum_{m \leq s^{-1/2}}-\phi'(ms) + \sum_{s^{-1/2} < m \leq s^{-1}}-\phi'(ms) + \sum_{m > s^{-1}} -\phi'(ms)\\
& =: \chi_1(s)+\chi_2(s)+\chi_3(s).\end{align*}

We have
\begin{align*} \chi_1(s) & = \sum_{m \leq s^{-1/2}} -\phi'(ms)\\
&= Ku\Gamma(u+1) \sum_{m \leq s^{-1/2}} (ms)^{-u-1} (1+O((ms)^{\epsilon/2}))\\
& = Ku\Gamma(u+1) s^{-u-1} \left( \sum_{m \leq s^{-1/2}} \frac{1}{m^{u+1}} + O(s^{\epsilon/2}) \sum_{m \leq s^{-1/2}} \frac{1}{m^{u+1 - \epsilon/2}}\right)\\
& = Ku\Gamma(u+1) s^{-u-1} \zeta(u+1)(1+O(s^{\epsilon/2})),\end{align*}
just as in the proof of Claim \ref{claim:approx}.

Note from (\ref{eq:deriv-approx-2}) that $-\phi'(s) = O(s^{-u-1})$. It follows that
\begin{align*} \chi_2(s) & = \sum_{s^{-1/2} < m \leq s^{-1}}-\phi'(ms)\\
& \leq O(s^{-u-1}) \sum_{s^{-1/2} < m < s^{-1}} \frac{1}{m^{u+1}}\\
& = O(s^{-u-1}) O(s^{u/2})\\
& = O(s^{-u/2-1}),\end{align*}
using the integral comparison test, just as in the proof of Claim \ref{claim:approx}.

Finally, for $t \geq 1$, by (\ref{eq:F1integral}), we have
\begin{align*} -\phi'(t)& = \int_{1}^{\infty} F_1(y) te^{-ty} dy\\
& = t \int_{1}^{\infty} F_1(y) e^{-ty} dy \\
& = O(t) \int_{1}^{\infty} y^{u+1} e^{-ty} dy \\
&= O(te^{-t})\\
&= O(2^{-t}).\end{align*}

Hence,
\begin{align*} \chi_3(s) & = \sum_{m > s^{-1}} -\phi'(ms)\\
& = O(1) \sum_{m > s^{-1}} 2^{-ms}\\
& = O(1) \frac{1}{1-2^{-s}}\\
& = O(s^{-1}).\end{align*}

Putting everything together, we have
\begin{align*} -h'(s) & = Ku\Gamma(u+1)\zeta(u+1) s^{-u-1} \zeta(u+1)(1+O(s^{\epsilon/2})) + O(s^{-u/2-1}) + O(s^{-1})\\
& = (1+O(s^{\epsilon/2}))Ku\Gamma(u+1)\zeta(u+1) s^{-u-1},\end{align*}
proving the claim.
\end{proof}

It follows from (\ref{eq:stationary}) and Claim \ref{claim:derivative} that
$$s(y) = (1+O(y^{-\kappa}))(Ku\Gamma(u+1)\zeta(u+1))^{1/(u+1)}y^{-1/(u+1)}.$$

Note that
$$h''(s) = \sum_{m=1}^{\infty} m\phi''(ms),$$
and
$$\phi''(s) = \sum_{n=1}^{\infty} n^2 \gamma(n) e^{-sn}.$$
Let
$$F_2(x) = \sum_{n \leq x} n^2 \gamma (n).$$
By applying Abel's summation formula (\ref{eq:Abel}) with $a(n) = n\gamma(n)$ and $f(n)=n$, we obtain
$$F_2(N) = \frac{Ku}{u+2}(1+O(N^{-\epsilon}))N^{u+2}.$$
Therefore,
$$F_2(x) = \frac{Ku}{u+2}(1+O(x^{-\epsilon}))x^{u+2}.$$
Using the same argument as in the proof of Claim \ref{claim:approx}, we get
$$\phi''(s) = Ku\Gamma(u+2) s^{-(u+2)} (1+O(s^{\epsilon/2}))$$
for all $s \in (0,1]$. Using the same argument as in the proof of Claim \ref{claim:derivative}, it follows that
\begin{align*} h''(s) & = Ku\Gamma(u+2) s^{-u-2} \zeta(u+1)(1+O(s^{\epsilon/2})) + O(s^{-u/2-2}) + O(s^{-2})\\
& = (1+O(s^{\eta}))Ku\Gamma(u+2)\zeta(u+1) s^{-u-2}.\end{align*}

Finally, we have
$$h'''(s) = \sum_{m=1}^{\infty} m^2\phi''(ms),$$
and
$$-\phi'''(s) = \sum_{n=1}^{\infty} n^3 \gamma(n) e^{-sn}.$$
Let
$$F_3(x) = \sum_{n \leq x} n^3 \gamma (n).$$
By applying Abel's summation formula (\ref{eq:Abel}) with $a(n) = n^2\gamma(n)$ and $f(n)=n$, we obtain
$$F_3(N) = \frac{Ku}{u+3}(1+O(N^{-\epsilon}))N^{u+3}.$$
Therefore,
$$F_3(x) = \frac{Ku}{u+3}(1+O(x^{-\epsilon}))x^{u+3}.$$
Using the same argument as in the proof of Claim \ref{claim:approx}, we get
$$-\phi'''(s) = Ku\Gamma(u+3) s^{-(u+3)} (1+O(s^{\epsilon/2})).$$
for all $s \in (0,1]$. Using the same argument as in the proof of Claim \ref{claim:derivative}, it follows that
\begin{align*} -h'''(s) & = Ku\Gamma(u+3) s^{-u-3} \zeta(u+1)(1+O(s^{\epsilon/2})) + O(s^{-u/2-3}) + O(s^{-3})\\
& = (1+O(s^{\epsilon/2}))Ku\Gamma(u+3)\zeta(u+1) s^{-u-3}.\end{align*}

We are now in a position to apply Proposition 1 and Theorem 2 of \cite{odlyzko}, which in our case (and with our notation), say the following.

\begin{theorem}[Odlyzko]
Suppose $(b(n))_{n=0}^{\infty}$ is a sequence of non-negative integers with generating function satisfying
$$\sum_{n=0}^{\infty} b(n) z^n = \prod_{j=1}^{\infty} (1-z^j)^{-\gamma(j)},$$
where $(\gamma(j))_{j=1}^{\infty}$ is a sequence of non-negative integers. Define
$$B(x) = \sum_{1\leq n \leq x} b(n).$$
Define
\begin{equation} \label{eq:g-definition} g(s) := \sum_{n=0}^{\infty} b(n) e^{-sn},\end{equation}
and define
$$h(s): = \log g(s).$$
Suppose the sum in (\ref{eq:g-definition}) converges for all $s >0$. Then:\newline
\newline
(Proposition 1) For each $y>0$ and each $s >0$, we have
$$B(y) \leq \exp(h(s) + ys).$$
There is a unique $s(y)>0$ which minimises $h(s) + ys$, and of course we have
$$B(y) \leq \exp(h(s(y))+ys(y)).$$
(Theorem 2) Moreover, let
$$C(y) := h''(s(y)).$$
Suppose that $y>0$ is such that $C(y) \geq 10^6$, and suppose that for all $t$ with
$$s(y) \leq t \leq s(y) + 20 C(y)^{-1/2},$$
we have
$$|h'''(t)| \leq 10^{-3} C(y)^{3/2}.$$
Then
$$B(y) - B(y-30C(y)^{-1/2}) \geq \exp(h(s(y)) + y s(y) - 30s(y)C(y)^{1/2} - 100).$$
\end{theorem}

In our case, by Claim \ref{claim:approx}, the sum in (\ref{eq:g-definition}) certainly converges for all $s>0$. We have $s(y) = \Theta(y^{-1/(u+1)})$ and so
$$C(y) = \Theta(y^{(u+2)/(u+1)}) \geq 10^6,$$
provided $y$ is sufficiently large. Hence, for all $t$ such that
$$s(y) \leq t \leq s(y) + 20 C(y)^{-1/2},$$
we have $t = \Theta(s(y))$, and so
$$|h'''(t)| = \Theta(y^{(u+3)/(u+1)}) \leq 10^{-3} C(y)^{3/2} = 10^{-3} \Theta(y^{(3u/2+3)/(u+1)}),$$
provided $y$ is sufficiently large. Hence,
\begin{align*} B(y) & \geq \exp(h(s(y)) + y s(y) - 30s(y)C(y)^{1/2} - 100)\\
& = \exp(h(s(y)) + y s(y) - O(y^{u/(2u+2)})).\end{align*}
On the other hand, we have
$$B(y) \leq \exp(h(s(y)) + ys(y)).$$
It remains to estimate $h(s(y)) + ys(y)$. We have
\begin{align*} h(s(y)) & = (1+O((s(y))^{\kappa})) \Gamma(u+1) \zeta(u+1) s(y)^{-u}\\
& = (1+O(y^{-\kappa}))K \Gamma(u+1) \zeta(u+1)(Ku\Gamma(u+1)\zeta(u+1))^{-u/(u+1)}y^{u/(u+1)}\\
& = (1+O(y^{-\kappa}))(K \Gamma(u+1) \zeta(u+1))^{1/(u+1)}u^{-u/(u+1)}y^{u/(u+1)}.\end{align*}
Hence,
\begin{align*} h(s(y)) + ys(y) & = (1+O(y^{-\kappa}))(K \Gamma(u+1) \zeta(u+1))^{1/(u+1)}u^{-u/(u+1)}y^{u/(u+1)}\\
&+ (1+O(y^{-\kappa}))(Ku\Gamma(u+1)\zeta(u+1))^{1/(u+1)}y^{u/(u+1)}\\
& = (1+O(y^{-\kappa}))(K \Gamma(u+1) \zeta(u+1))^{1/(u+1)}u^{-u/(u+1)}(u+1)y^{u/(u+1)}\\
& = (1+O(y^{-\kappa}))\tfrac{1}{u} (Ku \Gamma(u+2) \zeta(u+1))^{1/(u+1)}(u+1)^{u/(u+1)}y^{u/(u+1)},\end{align*}
using the functional equation $\Gamma(x+1) = x\Gamma(x)$. Putting together our upper and lower bounds, we get
$$B(y) = \exp((1+O(y^{-\kappa}))\tfrac{1}{u} (Ku \Gamma(u+2) \zeta(u+1))^{1/(u+1)}(u+1)^{u/(u+1)}y^{u/(u+1)}),$$
for all $y >0$. (Of course, we can make this true for all $y>0$, not just for $y$ sufficiently large, by adjusting the value of the constant implicit in the $O(y^{-\kappa})$ term.) In particular,
$$B(n) = \exp((1+O(n^{-\kappa}))\tfrac{1}{u}(Ku \Gamma(u+2) \zeta(u+1))^{1/(u+1)}(u+1)^{u/(u+1)}n^{u/(u+1)})$$
for all $n \in \mathbb{N}$. Taking logs, we get
\begin{equation}\label{eq:bound-sum} \log B(n) = (1+O(n^{-\kappa}))\tfrac{1}{u}(Ku \Gamma(u+2) \zeta(u+1))^{\frac{1}{u+1}}(u+1)^{\frac{u}{u+1}}n^{\frac{u}{u+1}}\end{equation}
for all $n \in \mathbb{N}$.

The final step is exactly the same as in Brigham's proof of Theorem \ref{thm:brigham}. Suppose in addition that every integer $n \geq n_0$ can be partitioned into integers in the set $\{n:\ \gamma(n) \geq 1\}$. Then we have $b(n) \geq b(m)$ if $1 \leq m \leq n-n_0$, using the interpretation of the $b(i)$'s in Fact \ref{fact:multisets}, since we can produce $b(m)$ distinct multi-sets of weight $n$ by adding on a fixed element of weight $n-m$ to each multi-set of weight $m$. Hence,
$$\frac{B(n-n_0)}{n-n_0} \leq b(n) \leq B(n),$$
so (\ref{eq:bound-sum}) holds with $b(n)$ in place of $B(n)$. This completes the proof of Theorem \ref{thm:brigham-error}.
\end{proof}

\subsection*{Acknowledgements}
We would like to thank Agelos Georgakopoulos, Alex Lubotzky, Mikael De La Salle, Omri Sarig and Romain Tessera for very helpful discussions. We also thank an anonymous referee for several very helpful suggestions, which we have incorporated.

\end{document}